\documentclass{cmslatex}

\usepackage[paperwidth=7in, paperheight=10in, margin=.875in]{geometry}
\usepackage[backref,colorlinks,linkcolor=red,anchorcolor=green,citecolor=blue]{hyperref}
\usepackage{amssymb}
\usepackage{amsmath}
\usepackage{graphicx}
\usepackage{enumerate}
\usepackage{cleveref}

\renewcommand{\theequation}{\arabic{section}.\arabic{equation}}

\crefname{assumption}{assumption}{assumptions}

\newcommand\numberthis{\addtocounter{equation}{1}\tag{\theequation}}

\allowdisplaybreaks

\begin{document}

\title{A Nonlocal diffusion model with $H^1$ convergence for Dirichlet Boundary\thanks{This work was supported by National Natural Science Foundation of China under grant 12071244.}}

\author{Tangjun Wang\thanks{Department of Mathematical Sciences, Tsinghua University, (wangtj20@mails.tsinghua.edu.cn).}
\and Zuoqiang Shi\thanks{Yau Mathematical Sciences Center, Tsinghua University \& Yanqi Lake Beijing Institute of Mathematical Sciences and Applications (zqshi@tsinghua.edu.cn).}}

\pagestyle{myheadings} \markboth{$H^1$ NONLOCAL DIFFUSION MODEL FOR DIRICHLET BOUNDARY}{Tangjun Wang, Zuoqiang Shi} \maketitle

\begin{abstract} 
In this paper, we present a nonlocal model for Poisson equation and corresponding eigenproblem with Dirichlet boundary condition. In the direct derivation of the nonlocal model, normal derivative is required which is not known for Dirichlet boundary. To overcome this difficulty, we treat the normal derivative as an auxiliary variable and derive corresponding nonlocal approximation of the boundary condition. For this specifically designed nonlocal model, we can prove its well-posedness and convergence to the counterpart local model. The nonlocal model is carefully designed such that coercivity and symmetry are preserved. Based on these good properties, we can prove the nonlocal model converges with first order rate in $H^1$ norm. Our model can be naturally extended to Poisson problems with Robin boundary and corresponding eigenvalue problem.
% Numerical experiments are conducted to verigy the theoretical results.
\end{abstract}



\maketitle

\section{Introduction}
Nonlocal models play a crucial role in many fields and have been widely studied, such as in peridynamical theory of continuum mechanics, nonlocal wave propagation and nonlocal diffusion process~\cite{alfaro2017propagation,bavzant2003nonlocal,blandin2016well,dayal2007real,kao2010random,vazquez2012nonlinear}. We are particularly interested in the problems where nonlocal operators have compact support~\cite{du2019nonlocal, tian2013analysis}, parameterized by the nonlocal horizon $\delta$. %As the horizon parameter $\delta\rightarrow 0$, a natural question is whether the solution of the nonlocal models converge to that of its local counterpart~\cite{du2012analysis}. The question can help us verify the numerical simulations of nonlocal modeling.
In this paper, we consider the nonlocal diffusion model, whose local counterpart is the Poisson equation.
%\[-\Delta u(\mathbf{x}) =f(\mathbf{x}), \quad \mathbf{x} \in \Omega \subset \mathbb{R}^n \]
There have been a large literature on nonlocal diffusion models~\cite{d2020numerical,du2013nonlocal,du2017fast,mengesha2013analysis,ponce2004estimate,trask2019asymptotically,zhang2018accurate}. Such models usually appear in the field of peridynamics~\cite{askari2008peridynamics, bobaru2009convergence,dayal2006kinetics,oterkus2012peridynamic,silling2010crack,taylor2015two}, where the singularity of materials can be effectively captured. They are also closely related to a meshless numerical methods called smoothed particle hydrodynamics~\cite{gingold1977smoothed, lucy1977numerical}. Nonlocal models also play important role in machine learning especially in semi-supervised learning~\cite{wang2018non,tao2018nonlocal,shi2017weighted}. %Nonlocal models can serve as an alternative to traditional PDE-based local models, and help us better understand and simulate existing local PDE models.

However, when boundary emerges, how to find a proper nonlocal analogous of local boundary conditions becomes a major issue. The traditional local boundary condition is not effective since the boundary is measure zero. One idea is to extend the boundary to a small volume adjacent to the boundary which is so-called volume constraints~\cite{du2012analysis}. Compared to  $O(\delta^2)$ convergence in the interior, naive extension only achieve $O(\delta)$ convergence on the boundary~\cite{macia2011theoretical,du2015integral}. High order extension gives high accuracy, but it is difficult to construct and requires high order derivatives of the local solution. Efforts have been made to properly design the nonlocal models or volume constraints in order to achieve better convergence rate, such as for Neumann boundary in one dimension~\cite{tao2017nonlocal} and two dimension~\cite{you2020asymptotically}, which can obtain $O(\delta^2)$ convergence in $L^2$ norm. The point integral method~\cite{li2017point,shi2017convergence} also proposes a nonlocal model to approximate the local Poisson equation with Neumann boundary, and proves $O(\delta)$ convergence in $H^1$ norm.

% \cite{barles2014neumann,cortazar2008approximate,dipierro2017nonlocal} for Neumann boundary condition, and~\cite{alali2015peridynamics,andreu2010nonlocal,du2012analysis,mengesha2016characterization,zhou2010mathematical} for Dirichlet or Robin conditions. 

In parallel with the works on Neumann boundary, the Dirichlet boundary is also an intriguing problem. The most straightforward approach is the constant extension method~\cite{macia2011theoretical}, which assume that Dirichlet boundary condition is constantly extended to a small neighborhood outside the domain, but it only provides first order convergence. To achieve higher order convergence, the predominant approaches in this field enforce the no-slip condition (analogy of Dirichlet boundary condition in fluid dynamics) by extrapolating the velocity of a particle across its tangent plane~\cite{morris1997modeling,yildiz2009sph,holmes2011smooth}. These methods need to calculate or approximate the distance between particles and the domain boundaries, which can be costly. Recently, \cite{yang2022uniform} proves $O(\delta^2)$ convergence given the first-order derivatives on the boundaries, which is generally not given a priori. Lee and Du~\cite{lee2021second} introduce a nonlocal gradient operator to mimic the extrapolation of the boundary data to the volumetric data, and obtain optimal $O(\delta^2)$ convergence in $L^2$ norm, but it is analysed only in one dimension. Shi and Sun use point integral method~\cite{shi2017convergence}, and approximate the Dirichlet boundary condition by Robin boundary condition~\cite{shi2018harmonic}. In this approach, the proved convergence rate depends on the weighted parameter in Robin boundary condition, which is far from sharp. Zhang and Shi~\cite{zhang2021second} propose a $O(\delta^2)$ convergent model in $H^1$ norm, but its model contains extra terms including curvature and normal vector, and requires more regularity and curvature of the boundary, which is hard to obtain in many applications.

In this paper, we propose a new approach to handle Dirichlet boundary condition. 
%Our work distinguishes from the existing literature in that, we do not try to estimate the normal derivative explicitly. 
We consider the Poisson equation with Dirichlet boundary condition~\eqref{eq:poisson}.
\begin{equation}
	\left\{\begin{aligned}
		-\Delta u(\mathbf{x}) &=f(\mathbf{x}), & & \mathbf{x} \in \Omega \\
		u(\mathbf{x}) & = b(\mathbf{x}), & & \mathbf{x} \in \partial \Omega
	\end{aligned}\right.
	\label{eq:poisson}
\end{equation}
Based on the point integral method~\cite{li2017point}, the Poisson equation is well approximated by an integral equation,
\begin{equation}
    \frac{1}{\delta^2}\int_{\Omega} R_\delta\left(\mathbf{x}, \mathbf{y} \right) (u(\mathbf{x}) - u(\mathbf{y}))\mathrm{d} \mathbf{y} - 2\int_{\partial \Omega}\bar{R}_\delta \left(\mathbf{x}, \mathbf{y} \right) \frac{\partial u}{\partial \mathbf{n}}(\mathbf{y})\mathrm{d} S_\mathbf{y} = \int_{\Omega} \bar{R}_\delta(\mathbf{x}, \mathbf{y}) f(\mathbf{y}) \mathrm{d} {\mathbf{y}},\quad \mathbf{x}\in \Omega. \label{eq:pim}
\end{equation}
for some kernel functions $R_\delta \left(\mathbf{x}, \mathbf{y} \right)$ and $\bar{R}_\delta \left(\mathbf{x}, \mathbf{y} \right)$ specified later.

%This approach is based on the 
%Rather, we treat them as variables and incorporate them in the nonlocal operators. Moreover, we can analyse the convergence results in a generalized domain other than low-dimensional cases. Formally, we try to deal with the following Poisson equation with Dirichlet boundary condition,

However, we cannot directly enforce the Dirichlet boundary condition on the integral equation~\eqref{eq:pim}, since the normal derivative $\frac{\partial u}{\partial \mathbf{n}}$ is not given explicitly. We therefore treat $\frac{\partial u}{\partial \mathbf{n}}$ as a new variable, $v_\delta(\mathbf{x})$.
To close the system, we need more equation about $u$ and $\frac{\partial u}{\partial \mathbf{n}}$. This will be given by careful approximation of the Dirichlet boundary condition. 
The idea is to also apply integral approximation \eqref{eq:pim} for $\mathbf{x}\in \partial \Omega$. 
To get symmetry and coercivity, we need to modify the integral approximation on the boundary. First, to preserve the symmetry, we choose the integral kernel to be $\bar{R}_\delta(\mathbf{x},\mathbf{y})$ which implies the integral equation on the boundary as following
\begin{equation}
    \frac{1}{\delta^2}\int_{\Omega} \bar{R}_\delta\left(\mathbf{x}, \mathbf{y} \right) (b(\mathbf{x}) - u(\mathbf{y}))\mathrm{d} \mathbf{y} - 2\int_{\partial \Omega}\bar{\bar{R}}_\delta \left(\mathbf{x}, \mathbf{y} \right) \frac{\partial u}{\partial \mathbf{n}}(\mathbf{y})\mathrm{d} S_\mathbf{y} 
    = \int_{\Omega} \bar{\bar{R}}_\delta(\mathbf{x}, \mathbf{y}) f(\mathbf{y}) \mathrm{d} {\mathbf{y}},~ \mathbf{x}\in \partial \Omega. \label{eq:pim-bd1}
\end{equation}
Here $\bar{\bar{R}}_\delta(\mathbf{x}, \mathbf{y})$ is also a kernel function defined later.

To get coercivity, we further simplify the second term of the left hand side of~\eqref{eq:pim-bd1} by moving $\frac{\partial u}{\partial \mathbf{n}}$ out of the integral.
\begin{equation}
    \frac{1}{\delta^2}\int_{\Omega} \bar{R}_\delta\left(\mathbf{x}, \mathbf{y} \right) (b(\mathbf{x}) - u(\mathbf{y}))\mathrm{d} \mathbf{y} - 2\frac{\partial u}{\partial \mathbf{n}}(\mathbf{x})\int_{\partial \Omega}\bar{\bar{R}}_\delta\left(\mathbf{x}, \mathbf{y} \right) \mathrm{d} S_\mathbf{y} = \int_{\Omega} \bar{\bar{R}}_\delta(\mathbf{x}, \mathbf{y}) f(\mathbf{y}) \mathrm{d} {\mathbf{y}},~ \mathbf{x}\in \partial \Omega. \label{eq:pim-bd} 
\end{equation}
Eventually, we get a nonlocal model as follows
\begin{equation}
	\left\{\begin{aligned}
		\mathcal{L}_{\delta} u_{\delta}(\mathbf{x})-\mathcal{G}_{\delta} v_{\delta}(\mathbf{x})&=F_{\mathrm{in}}(\mathbf{x}), && \mathbf{x} \in \Omega \\[5pt] \bar{\mathcal{L}}_{\delta} u_{\delta}(\mathbf{x})-\bar{\mathcal{G}}_{\delta} v_{\delta}(\mathbf{x}) &=F_{\mathrm{bd}}(\mathbf{x}), && \mathbf{x} \in \partial \Omega
	\end{aligned}\right.
	\label{eq:model}
\end{equation}
where the operators are defined as
\begin{gather*}
    \mathcal{L}_{\delta} u_{\delta}(\mathbf{x})=\frac{1}{\delta^{2}} \int_{\Omega}R_{\delta}(\mathbf{x}, \mathbf{y}) \left(u_{\delta}(\mathbf{x})-u_{\delta}(\mathbf{y})\right) \mathrm{d} {\mathbf{y}}, \quad \bar{\mathcal{L}}_{\delta} u_{\delta}(\mathbf{x})=\frac{1}{\delta^{2}} \int_{\Omega}\bar{R}_{\delta}(\mathbf{x}, \mathbf{y}) \left(b(\mathbf{x})-u_{\delta}(\mathbf{y})\right) \mathrm{d} {\mathbf{y}} \\
    \mathcal{G}_{\delta} v_{\delta}(\mathbf{x})=2 \int_{\partial \Omega} \bar{R}_{\delta}(\mathbf{x}, \mathbf{y}) v_{\delta}(\mathbf{y}) \mathrm{d} S_{\mathbf{y}}, \quad \bar{\mathcal{G}}_{\delta} v_{\delta}(\mathbf{x})=2 \int_{\partial \Omega} \bar{\bar{R}}_{\delta}(\mathbf{x}, \mathbf{y}) v_{\delta}(\mathbf{x}) \mathrm{d} S_{\mathbf{y}} \\
    F_{\mathrm{in}}(\mathbf{x})=\int_{\Omega} \bar{R}_{\delta}(\mathbf{x}, \mathbf{y}) f(\mathbf{y}) \mathrm{d} {\mathbf{y}}, \quad F_{\mathrm{bd}}(\mathbf{x})=\int_{\Omega} \bar{\bar{R}}_{\delta}(\mathbf{x}, \mathbf{y}) f(\mathbf{y}) \mathrm{d} {\mathbf{y}} 
\end{gather*}
Our goal in this paper is to present that model~\eqref{eq:model} assures a unique solution $\left(u_{\delta}, v_{\delta}\right)$ and it converges to the solution $\left(u, \frac{\partial u}{\partial \mathbf{n}}\right)$ of problem~\eqref{eq:poisson} as $\delta \rightarrow 0$.

\begin{remark}
    \label{remark:right_zero}
    We can set the right-hand side of \Cref{eq:pim-bd} to zero, i.e., $F_{\mathrm{bd}}(\mathbf{x})=0$, without affecting the main results in the paper, as explained in the proof of \Cref{thm:spectra_convergence}. %We have designed the current model because it provides a more natural way to utilize the point integral method and maintains a certain symmetry between the integral \cref{eq:pim} within the domain and the integral \cref{eq:pim-bd} on the boundary.
\end{remark}

The main contribution of this paper is that, for Poisson equation with Dirichlet condition, we prove that the solution computed by the nonlocal model~\eqref{eq:model} converges to the counterpart local solution with the rate of $O(\delta)$ in $H^1$ norm. We use standard techniques in numerical analysis, which combine consistency and stability to prove convergence. We prove the coercivity of the nonlocal operator, which implies the stability of our method. Together with the estimate of the truncation error, we get the convergence of our nonlocal model to the local Poisson equation. Our nonlocal model applies in any smooth domain in any dimension even in smooth manifold. This gives a very general approach to handle the Dirichlet boundary condition in nonlocal diffusion model.  

The paper is organized as follows. In \Cref{sec:prelim} we state several basic assumptions and estimates. Main results are presented in \Cref{sec:main} as \Cref{thm:wellposed} and \Cref{thm:convergence}. Proof of two theorems are provided in \Cref{sec:wellposed} and \Cref{sec:convergence} respectively. In \Cref{sec:robin}, we study the extension to Poisson equation with Robin boundary condition. Convergence of the eigenvalue problem is analysed in \Cref{sec:spectra}. Finally, we conclude the paper in \Cref{sec:conclusion}.

\section{Preliminaries}
\label{sec:prelim}

% First we state the following assumptions on the manifold $\Omega$
% \begin{assumption}
% $\Omega, \partial \Omega$ are both compact and $C^\infty$ smooth $k$-dimensional submanifolds isometrically embedded in a Euclidean space $\mathbb{R}^d$.
% \end{assumption}

First we state some basic assumptions used in our analysis.
\begin{assumption}[Assumptions on the domain]
\label{assumption:domain}
$\Omega \in \mathbb{R}^n$ is open, bounded and connected. $\partial \Omega$ is $C^2$ smooth.
\end{assumption}

% \[R_\delta(\mathbf{x}, \mathbf{y}) = C_\delta R\left(\frac{\|\mathbf{x} -\mathbf{y}\|^2}{4\delta^2}\right)\]
% \[\bar{R}_\delta(\mathbf{x}, \mathbf{y}) = C_\delta \bar{R}\left(\frac{\|\mathbf{x} -\mathbf{y}\|^2}{4\delta^2} \right)\]
% \[\bar{\bar{R}}_\delta(\mathbf{x}, \mathbf{y}) = C_\delta \bar{\bar{R}}\left(\frac{\|\mathbf{x} -\mathbf{y}\|^2}{4\delta^2} \right)\]

\begin{assumption}[Assumptions on the kernel function]
\label{assumption:kernel}
\hspace{1pt}
\begin{itemize}
    \item[(a)] Smoothness: $R\in C^1([0,1])$;
    \item[(b)] Nonnegativity: $R(r)\ge 0$ for $\forall r \ge 0$;
    \item[(c)] Compact support: $R(r) = 0$ for $\forall r >1$;
    \item[(d)] Nondegeneracy: $\exists \gamma_0>0$ so that $R(r)\ge\gamma_0$ for $0\le r\le\frac{1}{2}$.
\end{itemize}

\end{assumption}

$\bar{R}$ and $\bar{\bar{R}}$ are defined as
\[\bar{R}(r)=\int_r^{+\infty}R(s)\mathrm{d} s, \quad \bar{\bar{R}}(r)=\int_r^{+\infty}\bar{R}(s)\mathrm{d} s\]

Then the rescaled kernel function $R_\delta(\mathbf{x}, \mathbf{y}), \bar{R}_\delta(\mathbf{x}, \mathbf{y}), \bar{\bar{R}}_\delta(\mathbf{x}, \mathbf{y})$ are given as
\[\tilde{R}_\delta(\mathbf{x}, \mathbf{y}) = C_\delta \tilde{R}\left(\frac{\|\mathbf{x} -\mathbf{y}\|^2}{4\delta^2}\right)\]
for $\tilde{R}\in [R,\bar{R},\bar{\bar{R}}]$. The constant $C_\delta=\alpha_n \delta^{-n}$ is a normalization factor so that 
\[\int_{\mathbb{R}^n}R_\delta(\mathbf{x}, \mathbf{y}) \mathrm{d} \mathbf{y}= \alpha_n S_n \int_0^1 R(\frac{r^2}{4})r^{n-1}\mathrm{d} r = 1\]
with $S_n$ denotes the surface area of unit sphere in $\mathbb{R}^n$.

Here we provide some basic estimates that we use later.
\begin{proposition}
\label{prop:estimate}
Let $\tilde{R}$ be a kernel function satisfying \Cref{assumption:kernel} (a)(b)(c) and 
\[\tilde{R}_\delta (\mathbf{x},\mathbf{y})=\alpha_n\delta^{-n} \tilde{R}\left(\frac{\|\mathbf{x}-\mathbf{y}\|^2}{4\delta^2}\right)\]
Then there exists a constant $\delta_0>0$ only dependent on $\Omega$ and $\tilde{R}$, such that for $\delta \le \delta_0$, $\forall \mathbf{x} \in \Omega\cup \partial \Omega$
\[\frac{C_1}{3} < \int_\Omega \tilde{R}_\delta(\mathbf{x},\mathbf{y}) \mathrm{d} \mathbf{y} \le C_1, \qquad \frac{C_2}{3\delta} < \int_{\partial \Omega} \tilde{R}_\delta(\mathbf{x},\mathbf{y}) \mathrm{d} S_\mathbf{y} \le \frac{C_2}{\delta}\]
where
\[C_1=\alpha_n S_n \int_0^1 \tilde{R}(\frac{r^2}{4})r^{n-1}\mathrm{d} r, \qquad C_2= \alpha_n S_{n-1} \int_0^1 \tilde{R}(\frac{r^2}{4})r^{n-2}\mathrm{d} r\]
are constants independent of $\delta$.
\end{proposition}

\begin{proof}
See \Cref{subsec:a.1}
\end{proof}

Clearly $\tilde{R}\in [R,\bar{R},\bar{\bar{R}}]$ satisfies \Cref{assumption:kernel} (a)(b)(c), thus having the estimates above. We define $w_\delta, \bar{w} _\delta$, and $\bar{\bar{w}}_\delta$ as below, which will be used throughout this article.
\[w_\delta(\mathbf{x}):=\int_{\Omega} R_{\delta}(\mathbf{x}, \mathbf{y}) \mathrm{d} \mathbf{y}, \quad \bar{w}_\delta(\mathbf{x}):=\int_{\Omega} \bar{R}_{\delta}(\mathbf{x}, \mathbf{y}) \mathrm{d} \mathbf{y}, \quad \bar{\bar{w}}_\delta(\mathbf{x}):=\int_{\partial \Omega} \bar{\bar{R}}_{\delta}(\mathbf{x}, \mathbf{y}) \mathrm{d} S_{\mathbf{y}}\]

Finally, to simplify the notation and make the proof concise, we consider the homogeneous Dirichlet boundary conditions, i.e
\[\left\{\begin{aligned}
		-\Delta u(\mathbf{x}) &=f(\mathbf{x}), & & \mathbf{x} \in \Omega \\
		u(\mathbf{x}) & = 0, & & \mathbf{x} \in \partial \Omega
	\end{aligned}\right.\]
and consequently the operator $\bar{\mathcal{L}}_{\delta}$ in~\eqref{eq:model} becomes
\[\bar{\mathcal{L}}_{\delta} u_{\delta}(\mathbf{x})= - \frac{1}{\delta^{2}} \int_{\Omega} \bar{R}_{\delta}(\mathbf{x}, \mathbf{y}) u_{\delta}(\mathbf{y}) \mathrm{d} \mathbf{y}\]
All results can be easily extended to nonhomogeneous situations as long as $b(\mathbf{x}) \in L^2(\partial \Omega)$.

\section{Main Results}
\label{sec:main}

\begin{theorem}[Well-Posedness]
    \label{thm:wellposed}
    For fixed $\delta>0$ and $f \in L^2(\Omega)$, there exists a unique solution $u_{\delta} \in$ $L^{2}(\Omega), v_{\delta} \in L^{2}(\partial \Omega)$ to the integral model~\eqref{eq:model}.
    
    Moreover, we have $u_{\delta} \in$ $H^1(\Omega)$ and the following estimate, with constant $C>0$ independent of $\delta$,
    \[\left\|u_{\delta}\right\|_{H^1(\Omega)} \leq C \left\|f\right\|_{L^2(\Omega)}\]
\end{theorem}

\begin{remark}
    \label{remark:h-1}
    By extending the domain $\Omega$ to $\Omega_\delta=\left\{x| \text{dist}(\mathbf{x}, \Omega) \leq 2\delta \right\}$ and setting the right-hand side of \Cref{eq:pim-bd} to zero, as mentioned in \Cref{remark:right_zero}, we can establish an $L^2$ estimate for $u_\delta$ under a more generalized condition $f\in H^{-1}(\Omega_\delta)$. The proposition and its corresponding proof are provided in \Cref{sec:remark_proof}.
\end{remark}

\begin{theorem}[Convergence]
    \label{thm:convergence}
    Let $f \in H^{1}(\Omega)$, $u$ be the solution to the Poisson equation~\eqref{eq:poisson}, and $\left(u_{\delta}, v_{\delta}\right)$ be the solution to the integral model~\eqref{eq:model}, then we have the following estimate, with constant $C>0$ independent of $\delta$,
    \[\left\|u-u_{\delta}\right\|_{H^{1}(\Omega)}\leq C \delta\left\|f\right\|_{H^{1}(\Omega)}\]
    \[\left\|\frac{\partial u}{\partial \mathbf{n}}-v_{\delta}\right\|_{L^{2}(\partial \Omega)} \leq C \delta^{\frac{1}{2}}\left\|f\right\|_{H^{1}(\Omega)}\]
\end{theorem}

The proof of two theorems will be provided in \Cref{sec:wellposed} and \Cref{sec:convergence} respectively. Methods used in the analysis are standard. The nonlocal model is carefully designed such that it is equivalent to a bilinear form. The wellposedness in $L^2$ space can be proved by Lax-Milgram theorem. Then the regularity is boosted to $H^1$ space using the special structure of the nonlocal model. Convergence is obtained by combining local truncation error and stability. The stability analysis is basically similar to the coercivity of the bilinear form. Together with the estimate of the local truncation error, we can prove the first order convergence in $H^1$ norm.

%We provide the proof sketch here. \cref{thm:wellposed} can be divided into two parts: the existence and uniqueness of solution, and $H^1$ estimation. The first part can be proved by using Lax-Milgram theorem on the bilinear form, with inequality results including continuous (\cref{prop:continuous}), coercive (\cref{prop:coercive}) and bounded (\cref{prop:bounded}). The second part is proved by estimating $\|u_\delta\|_{L^2(\Omega)}$ and $\|\nabla u_\delta\|_{L^2(\Omega)}$ separately.
% can be obtained by The proof of \cref{thm:convergence} is basically similar to \cref{sec:wellposed}, but differs in that we need extra consistency results considering the truncation error.

\section{\texorpdfstring{Well-Posedness (\Cref{thm:wellposed})}{Well-Posedness}}
\label{sec:wellposed}

\subsection{Existence and uniqueness}
In this subsection, we will prove the first part of \Cref{thm:wellposed}, i.e. the existence and uniqueness of solution $\left(u_{\delta}, v_{\delta}\right)$ to \Cref{eq:model}, by Lax-Milgram Theorem. We observe that in the second equation of \Cref{eq:model},
\[\bar{\mathcal{G}}_{\delta} v_{\delta}(\mathbf{x})=2 \int_{\partial \Omega} \bar{\bar{R}}_{\delta}(\mathbf{x}, \mathbf{y}) v_{\delta}(\mathbf{x}) \mathrm{d} S_{\mathbf{y}} = 2\bar{\bar{w}}_\delta(\mathbf{x}) v_{\delta}(\mathbf{x})\]

Thus we can eliminate $v_\delta$ in the first equation by noticing that
\begin{equation}
    v_\delta(\mathbf{x})=-\frac{1}{2\delta^2\bar{\bar{w}}_\delta(\mathbf{x})} \int_{\Omega} \left(\bar{R}_{\delta}(\mathbf{x}, \mathbf{y}) u_\delta(\mathbf{y}) +\delta^2 \bar{\bar{R}}_{\delta}(\mathbf{x}, \mathbf{y}) f(\mathbf{y}) \right)\mathrm{d} \mathbf{y}, \quad \mathbf{x}\in \partial\Omega
    \label{eq:v_eliminate}
\end{equation}
	
So the first equation becomes
\begin{align*}
	\frac{1}{\delta^2} & \int_{\Omega} R_{\delta}(\mathbf{x}, \mathbf{y})(u_\delta(\mathbf{x})-u_\delta(\mathbf{y})) \mathrm{d} \mathbf{y}+\frac{1}{\delta^2} \int_{\partial \Omega} \frac{\bar{R}_{\delta}(\mathbf{x}, \mathbf{y})}{\bar{\bar{w}}_\delta(\mathbf{y})} \int_{\Omega} \bar{R}_{\delta}(\mathbf{y}, \mathbf{s}) u_\delta(\mathbf{s}) \mathrm{d} \mathbf{s} \mathrm{d} S_{\mathbf{y}}\\
	= {}& \int_{\Omega}  \bar{R}_{\delta}(\mathbf{x}, \mathbf{y}) f(\mathbf{y})\mathrm{d} \mathbf{y} - \int_{\partial \Omega} \frac{\bar{R}_{\delta}(\mathbf{x}, \mathbf{y})}{\bar{\bar{w}}_\delta(\mathbf{y})} \int_{\Omega} \bar{\bar{R}}_{\delta}(\mathbf{y}, \mathbf{s}) f(\mathbf{s}) \mathrm{d} \mathbf{s} \mathrm{d} S_{\mathbf{y}} \numberthis \label{eq:lax_milgram}
\end{align*}

We introduce the bilinear form $B : L^2(\Omega)\times L^2 (\Omega)\rightarrow \mathbb{R}$,
\begin{align*}
    B[u_\delta, p_\delta] ={}& \frac{1}{\delta^2} \int_{\Omega} p_\delta(\mathbf{x}) \int_{\Omega} R_{\delta}(\mathbf{x}, \mathbf{y})(u_\delta(\mathbf{x})-u_\delta(\mathbf{y})) \mathrm{d} \mathbf{y} \mathrm{d} \mathbf{x} \\
    &+ \frac{1}{\delta^2} \int_{\Omega} p_\delta(\mathbf{x}) \int_{\partial \Omega} \frac{\bar{R}_{\delta}(\mathbf{x}, \mathbf{y})}{\bar{\bar{w}}_\delta(\mathbf{y})} \int_{\Omega} \bar{R}_{\delta}(\mathbf{y}, \mathbf{s}) u_\delta(\mathbf{s}) \mathrm{d} \mathbf{s} \mathrm{d} S_{\mathbf{y}} \mathrm{d} \mathbf{x}
\end{align*}
and denote $F : L^2(\Omega) \rightarrow \mathbb{R}$ as
\[\langle F, p_\delta \rangle = \int_{\Omega} p_\delta(\mathbf{x}) \int_{\Omega}  \bar{R}_{\delta}(\mathbf{x}, \mathbf{y}) f(\mathbf{y})\mathrm{d} \mathbf{y} \mathrm{d} \mathbf{x} - \int_{\Omega} p_\delta(\mathbf{x}) \int_{\partial \Omega} \frac{\bar{R}_{\delta}(\mathbf{x}, \mathbf{y})}{\bar{\bar{w}}_\delta(\mathbf{y})} \int_{\Omega} \bar{\bar{R}}_{\delta}(\mathbf{y}, \mathbf{s}) f(\mathbf{s}) \mathrm{d} \mathbf{s} \mathrm{d} S_{\mathbf{y}} \mathrm{d} \mathbf{x}\]
for any $p_\delta \in L^2(\Omega)$. We need to prove the existence and uniqueness of solution $u_{\delta}$ to the following equation
\[B[u_\delta,p_\delta]=\langle F, p_\delta \rangle\]
To utilize the Lax-Milgram Theorem, we need the following three estimates: \Cref{prop:continuous} (Continuity), \Cref{prop:coercive} (Coercivity) and \Cref{prop:bounded} (Boundedness).

% \begin{restatable}[Goldbach's conjecture]{thm}{goldbach}
% \label{thm:goldbach}
% Every even integer greater than 2 can be expressed as the sum of two primes.
% \end{restatable}

\begin{proposition}[Continuity]
\label{prop:continuous}
    For any $u_\delta, p_\delta \in L^2(\Omega)$, there exists a constant $C>0$ independent of $\delta$ such that
    \[B[u_\delta, p_\delta]\leq \frac{C}{\delta^2}\|u_\delta\|_{L^2(\Omega)}\|p_\delta\|_{L^2(\Omega)}\]
\end{proposition}
\begin{proof}
	\begin{align*}
		B[u_\delta, p_\delta]=& \frac{1}{\delta^2} \int_{\Omega} p_\delta(\mathbf{x}) \int_{\Omega} R_{\delta}(\mathbf{x}, \mathbf{y})u_\delta(\mathbf{x})\mathrm{d} \mathbf{y} \mathrm{d} \mathbf{x} - \frac{1}{\delta^2} \int_{\Omega} p_\delta(\mathbf{x}) \int_{\Omega} R_{\delta}(\mathbf{x}, \mathbf{y}) u_\delta(\mathbf{y}) \mathrm{d} \mathbf{y} \mathrm{d} \mathbf{x} \\
		& +\frac{1}{\delta^2} \int_{\Omega} p_\delta(\mathbf{x}) \int_{\partial \Omega} \frac{\bar{R}_{\delta}(\mathbf{x}, \mathbf{y})}{\bar{\bar{w}}_\delta(\mathbf{y})} \int_{\Omega} \bar{R}_{\delta}(\mathbf{y}, \mathbf{s}) u_\delta(\mathbf{s}) \mathrm{d} \mathbf{s} \mathrm{d} S_{\mathbf{y}} \mathrm{d} \mathbf{x}
	\end{align*}
	
	%Since $\delta$ is fixed, we can regard $\frac{1}{\delta^2}$ as a constant, so we omit $\frac{1}{\delta^2}$ in the following proof. 
	Recall the estimates we have established in \Cref{prop:estimate}. The first term
	\begin{align*}
		&\int_{\Omega} p_\delta(\mathbf{x}) \int_{\Omega} R_{\delta}(\mathbf{x}, \mathbf{y})u_\delta(\mathbf{x})\mathrm{d} \mathbf{y} \mathrm{d} \mathbf{x} = \int_{\Omega} p_{\delta}(\mathbf{x}) u_{\delta}(\mathbf{x}) \left(\int_{\Omega} R_{\delta}(\mathbf{x}, \mathbf{y})\mathrm{d} \mathbf{y} \right) \mathrm{d} \mathbf{x} \\
		\leq {} & \vphantom{\int}  C \int_{\Omega} p_{\delta}(\mathbf{x}) u_{\delta}(\mathbf{x}) \mathrm{d} \mathbf{x} \leq C\left\|p_{\delta}\right\|_{L^{2}(\Omega)}\left\|u_{\delta}\right\|_{L^{2}(\Omega)}
	\end{align*}
	
	The second term can be bounded by using Cauchy-Schwarz inequality twice
	\begin{align*}
		\int_{\Omega} p_\delta(\mathbf{x}) \int_{\Omega} R_{\delta}(\mathbf{x}, \mathbf{y}) & u_\delta(\mathbf{y})  \mathrm{d} \mathbf{y} \mathrm{d} \mathbf{x}	\leq \left[\int_{\Omega} p_\delta^2(\mathbf{x})\mathrm{d} \mathbf{x} \int_{\Omega} \left( \int_{\Omega} R_{\delta}(\mathbf{x}, \mathbf{y})u_\delta(\mathbf{y}) \mathrm{d} \mathbf{y} \right)^2 \mathrm{d} \mathbf{x} \right]^{\frac{1}{2}}\\
		\leq {}& \left\|p_{\delta}\right\|_{L^{2}(\Omega)} \left[ \int_{\Omega} \left(\int_{\Omega} R_{\delta}(\mathbf{x}, \mathbf{y})\mathrm{d} \mathbf{y} \right)\left(\int_{\Omega} R_{\delta}(\mathbf{x}, \mathbf{y})u_\delta^2(\mathbf{y}) \mathrm{d} \mathbf{y} \right) \mathrm{d} \mathbf{x} \right]^{\frac{1}{2}}\\
		\leq {}& C_1 \left\|p_{\delta}\right\|_{L^{2}(\Omega)} \left[ \int_{\Omega} \int_{\Omega} R_{\delta}(\mathbf{x}, \mathbf{y})u_\delta^2(\mathbf{y}) \mathrm{d} \mathbf{x} \mathrm{d} \mathbf{y} \right]^{\frac{1}{2}}\\
		= {}& C_1 \left\|p_{\delta}\right\|_{L^{2}(\Omega)} \left[ \int_{\Omega} \left(\int_{\Omega} R_{\delta}(\mathbf{x}, \mathbf{y})\mathrm{d} \mathbf{x} \right) u_\delta^2(\mathbf{y})  \mathrm{d} \mathbf{y} \right]^{\frac{1}{2}}\\
		\leq {} & C \left\|p_{\delta}\right\|_{L^{2}(\Omega)} \left[ \int_{\Omega} u_\delta^2(\mathbf{y}) \mathrm{d} \mathbf{y} \right]^{\frac{1}{2}} = C\left\|p_{\delta}\right\|_{L^{2}(\Omega)}\left\|u_{\delta}\right\|_{L^{2}(\Omega)}
	\end{align*}
	
	The third term
	\begin{align}
		&\int_{\Omega} p_\delta(\mathbf{x}) \int_{\partial \Omega} \frac{\bar{R}_{\delta}(\mathbf{x}, \mathbf{y})}{\bar{\bar{w}}_\delta(\mathbf{y})} \int_{\Omega} \bar{R}_{\delta}(\mathbf{y}, \mathbf{s}) u_\delta(\mathbf{s}) \mathrm{d} \mathbf{s} \mathrm{d} S_{\mathbf{y}} \mathrm{d} \mathbf{x} \nonumber\\
		={}&\int_{\partial \Omega} \frac{1}{\bar{\bar{w}}_\delta(\mathbf{y})} \left(\int_{\Omega} \bar{R}_{\delta}(\mathbf{x}, \mathbf{y}) p_\delta(\mathbf{x}) \mathrm{d} \mathbf{x} \right) \left(\int_{\Omega} \bar{R}_{\delta}(\mathbf{y}, \mathbf{s}) u_\delta(\mathbf{s}) \mathrm{d} \mathbf{s} \right) \mathrm{d} S_{\mathbf{y}} \nonumber\\
		\leq {}& C_1\delta \int_{\partial \Omega} \left(\int_{\Omega} \bar{R}_{\delta}(\mathbf{x}, \mathbf{y}) p_\delta(\mathbf{x}) \mathrm{d} \mathbf{x} \right) \left(\int_{\Omega} \bar{R}_{\delta}(\mathbf{y}, \mathbf{s}) u_\delta(\mathbf{s}) \mathrm{d} \mathbf{s} \right) \mathrm{d} S_{\mathbf{y}} \label{eq:continuous_third_term1} \\
		\leq {}& C_1\delta \left[\int_{\partial \Omega}\left(\int_{\Omega} \bar{R}_{\delta}(\mathbf{x}, \mathbf{y}) p_\delta(\mathbf{x}) \mathrm{d} \mathbf{x} \right)^2 \mathrm{d} S_{\mathbf{y}} \int_{\partial \Omega} \left(\int_{\Omega} \bar{R}_{\delta}(\mathbf{y}, \mathbf{s}) u_\delta(\mathbf{s}) \mathrm{d} \mathbf{s} \right)^2 \mathrm{d} S_{\mathbf{y}} \right]^{\frac{1}{2}} \label{eq:continuous_third_term2}
	\end{align}
	where in \Cref{eq:continuous_third_term1}, we use the estimate in \Cref{prop:estimate}. In \Cref{eq:continuous_third_term2}, we have
	\begin{align*}
		\int_{\partial \Omega}&\left(\int_{\Omega} \bar{R}_{\delta}(\mathbf{x}, \mathbf{y}) p_\delta(\mathbf{x}) \mathrm{d} \mathbf{x} \right)^2 \mathrm{d} S_{\mathbf{y}} \leq \int_{\partial \Omega}\left(\int_{\Omega} \bar{R}_{\delta}(\mathbf{x}, \mathbf{y}) \mathrm{d} \mathbf{x} \right)\left(\int_{\Omega} \bar{R}_{\delta}(\mathbf{x}, \mathbf{y}) p_\delta^2(\mathbf{x}) \mathrm{d} \mathbf{x} \right) \mathrm{d} S_{\mathbf{y}} \\
		\leq {}& C_1 \int_{\partial \Omega}\int_{\Omega} \bar{R}_{\delta}(\mathbf{x}, \mathbf{y}) p_\delta^2(\mathbf{x}) \mathrm{d} \mathbf{x} \mathrm{d} S_{\mathbf{y}} = C_1 \int_{\Omega} \left(\int_{\partial \Omega} \bar{R}_{\delta}(\mathbf{x}, \mathbf{y}) \mathrm{d} S_{\mathbf{y}}\right) p_\delta^2(\mathbf{x})  \mathrm{d} \mathbf{x} \\
		\leq {}& \vphantom{\int}  \frac{C}{\delta} \int_{\Omega} p_\delta^2(\mathbf{x})  \mathrm{d} \mathbf{x} = \frac{C}{\delta} \left\|p_{\delta}\right\|_{L^{2}(\Omega)}^2
	\end{align*}
	
	Similarly
	\[\int_{\partial \Omega} \left(\int_{\Omega} \bar{R}_{\delta}(\mathbf{y}, \mathbf{s}) u_\delta(\mathbf{s}) \mathrm{d} \mathbf{s} \right)^2 \mathrm{d} S_{\mathbf{y}} \leq \frac{C}{\delta} \left\|u_{\delta}\right\|_{L^{2}(\Omega)}^2\]
	
	Thus the third term
	\[\int_{\Omega} p_\delta(\mathbf{x}) \int_{\partial \Omega} \frac{\bar{R}_{\delta}(\mathbf{x}, \mathbf{y})}{\bar{\bar{w}}_\delta(\mathbf{y})} \int_{\Omega} \bar{R}_{\delta}(\mathbf{y}, \mathbf{s}) u_\delta(\mathbf{s}) \mathrm{d} \mathbf{s} \mathrm{d} S_{\mathbf{y}} \mathrm{d} \mathbf{x} \leq C \left\|p_{\delta}\right\|_{L^{2}(\Omega)}\left\|u_{\delta}\right\|_{L^{2}(\Omega)}\]
	
	Combine these three estimates, we prove the continuity of $B$.
\end{proof}

\begin{proposition}[Coercivity]
\label{prop:coercive}
    For any $u_\delta \in L^2(\Omega)$, there exists a constant $C>0$ independent of $\delta$ such that
    \[B[u_\delta, u_\delta]\geq C\|u_\delta\|_{L^2(\Omega)}^2\]
\end{proposition}

The proof of \Cref{prop:coercive} is more involved. We should first introduce two lemmas.

\begin{lemma}
	\label{lemma:r_rbar}
	If $\delta$ is small enough, then for any $u\in L^2(\Omega)$, there exists a constant $C>0$ independent of $\delta$ and $u$, such that
	\[\int_{\Omega} \int_{\Omega} R_{\delta}(\mathbf{x}, \mathbf{y})(u(\mathbf{x})-u(\mathbf{y}))^{2} \mathrm{d} \mathbf{x} \mathrm{d} \mathbf{y}\geq C \int_{\Omega} \int_{\Omega} \bar{R}_{\delta}(\mathbf{x}, \mathbf{y})(u(\mathbf{x})-u(\mathbf{y}))^{2} \mathrm{d} \mathbf{x} \mathrm{d} \mathbf{y} \]
\end{lemma}
\begin{proof}
	See \Cref{subsec:b.1}.
\end{proof}

\begin{lemma}[\cite{shi2017convergence}]
	\label{lemma:gradient}
	For any function $u \in L^{2}(\Omega)$, there exists a constant $C>0$ independent of $\delta$ and $u$, such that
	\[\frac{1}{\delta^2} \int_{\Omega} \int_{\Omega} R_\delta(\mathbf{x}, \mathbf{y})(u(\mathbf{x})-u(\mathbf{y}))^{2} \mathrm{d} \mathbf{x} \mathrm{d} \mathbf{y} \geq C \int_{\Omega}\left\|\nabla \left(\frac{1}{w_{\delta}(\mathbf{x})} \int_{\Omega} R_\delta(\mathbf{x}, \mathbf{y}) u(\mathbf{y}) \mathrm{d} \mathbf{y}\right)\right\|^{2} \mathrm{d} \mathbf{x}\]
	where $w_{\delta}(\mathbf{x})=\int_{\Omega} R_{\delta}(\mathbf{x}, \mathbf{y}) \mathrm{d} \mathbf{y}$
\end{lemma}

Next we are ready to derive that $B$ is coercive.
\begin{proof}
	\begin{align}
		B[u_\delta, u_\delta]= {}& \frac{1}{\delta^2} \int_{\Omega} u_\delta(\mathbf{x}) \int_{\Omega} R_{\delta}(\mathbf{x}, \mathbf{y})(u_\delta(\mathbf{x})-u_\delta(\mathbf{y})) \mathrm{d} \mathbf{y} \mathrm{d} \mathbf{x} \nonumber\\
		&+\frac{1}{\delta^2} \int_{\Omega} u_\delta(\mathbf{x}) \int_{\partial \Omega} \frac{\bar{R}_{\delta}(\mathbf{x}, \mathbf{y})}{\bar{\bar{w}}_\delta(\mathbf{y})} \int_{\Omega} \bar{R}_{\delta}(\mathbf{y}, \mathbf{s}) u_\delta(\mathbf{s}) \mathrm{d} \mathbf{s} \mathrm{d} S_{\mathbf{y}} \mathrm{d} \mathbf{x} \nonumber\\
		= {} & \frac{1}{2\delta^2} \int_{\Omega} \int_{\Omega} R_{\delta}(\mathbf{x}, \mathbf{y})(u_\delta(\mathbf{x}) - u_\delta(\mathbf{y}))^2 \mathrm{d} \mathbf{x} \mathrm{d} \mathbf{y} \label{eq:coercive_first_term1}\\
		& +\frac{1}{\delta^2} \int_{\partial \Omega} \frac{1}{\bar{\bar{w}}_\delta(\mathbf{x})} \left(\int_{\Omega} \bar{R}_{\delta}(\mathbf{x}, \mathbf{y}) u_\delta(\mathbf{y}) \mathrm{d} \mathbf{y}\right)^2 \mathrm{d} S_{\mathbf{x}} \label{eq:coercive_second_term1}
	\end{align}
	
	Define the smoothed version of $u$
	\begin{equation}
		\hat{u}_\delta(\mathbf{x})= \frac{1}{\bar{w}_\delta(\mathbf{x})}\int_{\Omega} \bar{R}_\delta(\mathbf{x}, \mathbf{y}) u_\delta(\mathbf{y})\mathrm{d} \mathbf{y}
		\label{eq:u_smoothed}
	\end{equation}
	
	Using \Cref{lemma:r_rbar} and \Cref{lemma:gradient}, in which we substitute the kernel function $R$ with $\bar{R}$, the first term in \Cref{eq:coercive_first_term1} can control
	\begin{align*}
		&\frac{1}{2\delta^2}\int_{\Omega} \int_{\Omega} R_{\delta}(\mathbf{x}, \mathbf{y})(u_\delta(\mathbf{x})-u_\delta(\mathbf{y}))^2 \mathrm{d} \mathbf{x} \mathrm{d} \mathbf{y} \geq \frac{C_1}{2\delta^2}\int_{\Omega} \int_{\Omega} \bar{R}_{\delta}(\mathbf{x}, \mathbf{y})(u_\delta(\mathbf{x})-u_\delta(\mathbf{y}))^2 \mathrm{d} \mathbf{x} \mathrm{d} \mathbf{y}\\
		\geq {}& C\int_{\Omega} \left\|\nabla\left( \frac{1}{\bar{w}_\delta(\mathbf{x})}\int_{\Omega} \bar{R}_\delta(\mathbf{x}, \mathbf{y}) u_\delta(\mathbf{y})\mathrm{d} \mathbf{y} \right)\right\|^2 \mathrm{d} \mathbf{x}= C\|\nabla\hat{u}_\delta(\mathbf{x})\|^2_{L^2(\Omega)}  \numberthis \label{eq:coercive_first_term2}
	\end{align*}
	
	and the second term in \Cref{eq:coercive_second_term1} can control
	\begin{align*}
		& \frac{1}{\delta^2}\int_{\partial \Omega} \frac{1}{\bar{\bar{w}}_\delta(\mathbf{x})} \left(\int_{\Omega} \bar{R}_\delta(\mathbf{x}, \mathbf{y}) u_\delta(\mathbf{y})\mathrm{d} \mathbf{y}\right)^2 \mathrm{d} S_\mathbf{x} \\
		= {} & \frac{1}{\delta^2}\int_{\partial \Omega} \frac{\bar{w}_\delta^2(\mathbf{x})}{\bar{\bar{w}}_\delta(\mathbf{x})} \left(\frac{1}{\bar{w}_\delta(\mathbf{x})}\int_{\Omega} \bar{R}_\delta(\mathbf{x}, \mathbf{y}) u_\delta(\mathbf{y})\mathrm{d} \mathbf{y}\right)^2 \mathrm{d} S_\mathbf{x} \\
		= {} & \frac{1}{\delta^2}\int_{\partial \Omega} \frac{\bar{w}_\delta^2(\mathbf{x})}{\bar{\bar{w}}_\delta(\mathbf{x})} \hat{u}_\delta^2(\mathbf{x})\mathrm{d} S_\mathbf{x} \geq \frac{C}{\delta} \|\hat{u}_\delta(\mathbf{x})\|^2_{L^2(\partial\Omega)} \geq C \|\hat{u}_\delta(\mathbf{x})\|^2_{L^2(\partial\Omega)} \numberthis \label{eq:coercive_second_term2} 
	\end{align*}
	
	In addition, the difference between $u_\delta(\mathbf{x})$ and its smoothed version $\hat{u}_\delta(\mathbf{x})$ can be controlled using \Cref{lemma:r_rbar},
	\begin{align*}
		\|u_\delta(\mathbf{x})-\hat{u}_\delta(\mathbf{x})\|^2_{L^2(\Omega)} & = \|\frac{1}{\bar{w}_\delta(\mathbf{x})} \int_{\Omega} \bar{R}_\delta(\mathbf{x}, \mathbf{y})(u_\delta(\mathbf{x})-u_\delta(\mathbf{y})) \mathrm{d} \mathbf{y}\|^2_{L^2(\Omega)}\\
		& =\int_{\Omega} \frac{1}{\bar{w}^2_\delta(\mathbf{x})} \left(\int_{\Omega} \bar{R}_\delta(\mathbf{x}, \mathbf{y})(u_\delta(\mathbf{x})-u_\delta(\mathbf{y})) \mathrm{d} \mathbf{y}\right)^2 \mathrm{d} \mathbf{x}\\
		& \leq \int_{\Omega} \frac{1}{\bar{w}_\delta(\mathbf{x})} \int_{\Omega} \bar{R}_\delta(\mathbf{x}, \mathbf{y})(u_\delta(\mathbf{x})-u_\delta(\mathbf{y}))^2 \mathrm{d} \mathbf{y} \mathrm{d} \mathbf{x} \\
		& \leq C_1\int_{\Omega} \int_{\Omega} \bar{R}_\delta(\mathbf{x}, \mathbf{y})(u_\delta(\mathbf{x})-u_\delta(\mathbf{y}))^2 \mathrm{d} \mathbf{y} \mathrm{d} \mathbf{x} \\
		& \leq C\int_{\Omega} \int_{\Omega} R_\delta(\mathbf{x}, \mathbf{y})(u_\delta(\mathbf{x})-u_\delta(\mathbf{y}))^2 \mathrm{d} \mathbf{y} \mathrm{d} \mathbf{x} \\
		& \leq \frac{C}{\delta^2}\int_{\Omega} \int_{\Omega} R_\delta(\mathbf{x}, \mathbf{y})(u_\delta(\mathbf{x})-u_\delta(\mathbf{y}))^2 \mathrm{d} \mathbf{y} \mathrm{d} \mathbf{x}  \numberthis \label{eq:coercive_difference} 
	\end{align*}
	
	Moreover, the Poincare inequality with boundary (Eq. 6.11.3 in \cite{maz2013sobolev}) gives us,
	\begin{equation}
		\|\nabla\hat{u}_\delta(\mathbf{x})\|^2_{L^2(\Omega)}+\|\hat{u}_\delta(\mathbf{x})\|^2_{L^2(\partial\Omega)}\geq C \|\hat{u}_\delta(\mathbf{x})\|^2_{L^2(\Omega)}
		\label{eq:poincare}
	\end{equation}
	
	Combining \Cref{eq:coercive_first_term2}, \Cref{eq:coercive_second_term2}, \Cref{eq:coercive_difference} and \Cref{eq:poincare}, we can get
	\begin{align*}
		& ~ B[u_\delta, u_\delta] \\
		& = \frac{1}{2\delta^2} \int_{\Omega} \int_{\Omega} R_{\delta}(\mathbf{x}, \mathbf{y})(u_\delta(\mathbf{x}) - u_\delta(\mathbf{y}))^2 \mathrm{d} \mathbf{x} \mathrm{d} \mathbf{y} +\frac{1}{\delta^2} \int_{\partial \Omega} \frac{1}{\bar{\bar{w}}_\delta(\mathbf{x})} \left(\int_{\Omega} \bar{R}_{\delta}(\mathbf{x}, \mathbf{y}) u_\delta(\mathbf{y}) \mathrm{d} \mathbf{y}\right)^2 \mathrm{d} S_{\mathbf{x}}\\
		& \geq \frac{C_1}{\delta^2}\int_{\Omega} \int_{\Omega} R_{\delta}(\mathbf{x}, \mathbf{y})(u_\delta(\mathbf{x})-u_\delta(\mathbf{y}))^2 \mathrm{d} \mathbf{x} \mathrm{d} \mathbf{y} + C_2\|\nabla\hat{u}_\delta(\mathbf{x})\|^2_{L^2(\Omega)}+C_2\|\hat{u}_\delta(\mathbf{x})\|^2_{L^2(\partial\Omega)}\\
		& \vphantom{\int} \geq C_3\|u_\delta(\mathbf{x})-\hat{u}_\delta(\mathbf{x})\|^2_{L^2(\Omega)} + C_3\|\hat{u}_\delta(\mathbf{x})\|^2_{L^2(\Omega)} \geq C\|u_\delta\|_{L^2(\Omega)}^2
	\end{align*}
\end{proof}

\begin{proposition}[Boundedness]
	\label{prop:bounded}
	For any $p_\delta \in L^2(\Omega)$, there exists a constant $C>0$ independent of $\delta$ such that
	\[\langle F, p_\delta \rangle \leq C \|p_\delta\|_{L^2(\Omega)}\]
\end{proposition}
\begin{proof}
	See \Cref{subsec:b.2}.
\end{proof}

Finally, we can prove the first part of \Cref{thm:wellposed}.
\begin{proof}(Proof of \Cref{thm:wellposed}, part 1)
    Combine \Cref{prop:continuous},\Cref{prop:coercive} and \Cref{prop:bounded}, we can prove the existence and uniqueness of $u_\delta\in L^2(\Omega)$ using Lax-Milgram Theorem. It is an easy corollary that $v_\delta \in L^2(\partial \Omega)$ also exists and is unique. Uniqueness is obvious, and as for existence
\begin{align*}
    &\left\|v_\delta \right\|_{L^2(\partial \Omega)}^2 = \int_{\Omega} \left(\frac{1}{2\delta^2\bar{\bar{w}}_\delta(\mathbf{x})} \int_{\Omega} \left(\bar{R}_{\delta}(\mathbf{x}, \mathbf{y}) u_\delta(\mathbf{y}) +\delta^2 \bar{\bar{R}}_{\delta}(\mathbf{x}, \mathbf{y}) f(\mathbf{y}) \right)\mathrm{d} \mathbf{y}\right)^2\mathrm{d} \mathbf{x} \\
    \leq {}& C_1\int_{\Omega} \left(\int_{\Omega} \left(\bar{R}_{\delta}(\mathbf{x}, \mathbf{y}) u_\delta(\mathbf{y}) +\delta^2 \bar{\bar{R}}_{\delta}(\mathbf{x}, \mathbf{y}) f(\mathbf{y}) \right)\mathrm{d} \mathbf{y}\right)^2\mathrm{d} \mathbf{x} \\
    \leq {}& 2C_1\int_{\Omega} \left(\int_{\Omega} \bar{R}_{\delta}(\mathbf{x}, \mathbf{y}) u_\delta(\mathbf{y}) \mathrm{d} \mathbf{y}\right)^2\mathrm{d} \mathbf{x}+2C_1\int_{\Omega} \left(\int_{\Omega} \delta^2 \bar{\bar{R}}_{\delta}(\mathbf{x}, \mathbf{y}) f(\mathbf{y})\mathrm{d} \mathbf{y}\right)^2\mathrm{d} \mathbf{x} \\
    \leq {}& \vphantom{\left(\int_{\Omega} \bar{R}_\delta(\mathbf{x}, \mathbf{y}) \right)^2} C \int_{\Omega}  u_\delta^2(\mathbf{y}) \mathrm{d} \mathbf{y} + C \int_{\Omega}  f^2(\mathbf{y}) \mathrm{d} \mathbf{y} = C\left\|u_\delta \right\|_{L^2(\Omega)}^2 + C \left\|f \right\|_{L^2(\Omega)}^2
\end{align*}

As we have proved $u_\delta \in L^2(\Omega)$ and we assume $f\in L^2(\Omega)$, we can get $v_\delta \in L^2(\partial \Omega)$. Since the equations are nonlocal, a weak solution is automatically a strong solution, satisfying the equation pointwise.
\end{proof}

\subsection{\texorpdfstring{$H^1$ estimation}{H1 estimation}}

\label{subsec:h1_estimation}
Then we will prove the second part of \Cref{thm:wellposed}, i.e. $u_\delta \in H^1(\Omega)$ and we will also derive an upper bound for its $H^1$ norm. First we provide the core inequality in this subsection,

\begin{proposition}
\label{prop:h1_core}
    If $u_\delta$ is the solution to \Cref{eq:model}, then
    \begin{align}
	\frac{1}{2\delta^2}\int_{\Omega} \int_{\Omega} R_{\delta}(\mathbf{x}, \mathbf{y})&(u_\delta(\mathbf{x})-u_\delta(\mathbf{y}))^2 \mathrm{d} \mathbf{x} \mathrm{d} \mathbf{y}+ \frac{1}{2\delta^2}\int_{\partial \Omega} \frac{1}{\bar{\bar{w}}_\delta(\mathbf{x})} \left(\int_{\Omega} \bar{R}_\delta(\mathbf{x}, \mathbf{y}) u_\delta(\mathbf{y})\mathrm{d} \mathbf{y}\right)^2 \mathrm{d} S_\mathbf{x} \nonumber\\
	\leq{} & \vphantom{\left(\int_{\Omega} \bar{R}_\delta(\mathbf{x}, \mathbf{y}) \right)^2} \int_{\Omega} u_\delta(\mathbf{x}) F_{\mathrm{in}}(\mathbf{x}) \mathrm{d} \mathbf{x} + \frac{\delta^2}{2} \int_{\partial \Omega} \frac{F^2_{\mathrm{bd}}(\mathbf{x}) }{\bar{\bar{w}}_\delta(\mathbf{x})} \mathrm{d} S_\mathbf{x} 
	\label{eq:h1_core}
    \end{align}
\end{proposition}

\begin{proof}
    Since we have proved $u_\delta\in L^2(\Omega)$, we may choose $u_\delta$ as test function and get
    \[B[u_\delta,u_\delta]=\langle F, u_\delta \rangle\]
    i.e.
    \begin{align*}
        & \frac{1}{2\delta^2} \int_{\Omega} \int_{\Omega} R_{\delta}(\mathbf{x}, \mathbf{y})(u_\delta(\mathbf{x}) - u_\delta(\mathbf{y}))^2 \mathrm{d} \mathbf{x} \mathrm{d} \mathbf{y} + \frac{1}{\delta^2} \int_{\partial \Omega} \frac{1}{\bar{\bar{w}}_\delta(\mathbf{x})} \left(\int_{\Omega} \bar{R}_{\delta}(\mathbf{x}, \mathbf{y}) u_\delta(\mathbf{y}) \mathrm{d} \mathbf{y}\right)^2 \mathrm{d} S_{\mathbf{x}}\\
        = & \int_{\Omega} u_\delta(\mathbf{x}) F_{\mathrm{in}}(\mathbf{x}) \mathrm{d} \mathbf{x} - \int_{\Omega} u_\delta(\mathbf{x}) \int_{\partial \Omega} \frac{\bar{R}_{\delta}(\mathbf{x}, \mathbf{y})}{\bar{\bar{w}}_\delta(\mathbf{y})} F_{\mathrm{bd}}(\mathbf{y}) \mathrm{d} S_{\mathbf{y}} \mathrm{d} \mathbf{x}
    \end{align*}
	
	We have $ -ab \leq \frac{1}{2}\delta^2 a^2+ \frac{1}{2\delta^2}b^2$, thus
	\begin{align*}
    	& -\int_{\Omega} u_\delta(\mathbf{x}) \int_{\partial \Omega} \frac{\bar{R}_{\delta}(\mathbf{x}, \mathbf{y})}{\bar{\bar{w}}_\delta(\mathbf{y})} F_{\mathrm{bd}}(\mathbf{y}) \mathrm{d} S_{\mathbf{y}} \mathrm{d} \mathbf{x} = -\int_{\partial \Omega} \frac{F_{\mathrm{bd}}(\mathbf{x})}{\bar{\bar{w}}_\delta(\mathbf{x})} \left(\int_{\Omega} \bar{R}_{\delta}(\mathbf{x}, \mathbf{y}) u_\delta(\mathbf{y}) \mathrm{d} \mathbf{y}\right) \mathrm{d} S_\mathbf{x} \\
	    \leq {} & \frac{\delta^2}{2} \int_{\partial \Omega} \frac{F^2_{\mathrm{bd}}(\mathbf{x}) }{\bar{\bar{w}}_\delta(\mathbf{x})} \mathrm{d} S_\mathbf{x} + \frac{1}{2\delta^2}\int_{\partial \Omega} \frac{1}{\bar{\bar{w}}_\delta(\mathbf{x})} \left(\int_{\Omega} \bar{R}_\delta(\mathbf{x}, \mathbf{y}) u_\delta(\mathbf{y})\mathrm{d} \mathbf{y}\right)^2 \mathrm{d} S_\mathbf{x} \numberthis \label{eq:h1_abupper}
	\end{align*}
	Putting \Cref{eq:h1_abupper} into the equality gives desired estimate.
\end{proof}

\Cref{eq:h1_core} is the core inequality we will use in the following proof. We will show that the left hand side of \Cref{eq:h1_core} can control the $L^2$ norm of $u_\delta$ and $\nabla u_\delta$, while the right hand side can be controlled by the $L^2$ norm of $f$. 

First of all, the $L^2$ norm $\|u_\delta\|_{L^2(\Omega)}$ can be controlled by \Cref{prop:coercive} as an easy corollary.

\begin{corollary}
\label{corol:l2_estimation}
    For any $u_\delta \in L^2(\Omega)$, there exists a constant $C>0$, such that
    \begin{align*}
        C\|u_\delta\|_{L^2(\Omega)}^2 \leq \frac{1}{2\delta^2}\int_{\Omega} \int_{\Omega} R_{\delta}(\mathbf{x}, \mathbf{y})&(u_\delta(\mathbf{x})-u_\delta(\mathbf{y}))^2 \mathrm{d} \mathbf{x} \mathrm{d} \mathbf{y} \\
        & + \frac{1}{2\delta^2}\int_{\partial \Omega} \frac{1}{\bar{\bar{w}}_\delta(\mathbf{x})} \left(\int_{\Omega} \bar{R}_\delta(\mathbf{x}, \mathbf{y}) u_\delta(\mathbf{y})\mathrm{d} \mathbf{y}\right)^2 \mathrm{d} S_\mathbf{x} \numberthis \label{eq:h1_l2_estimation}
    \end{align*}
\end{corollary}

\begin{proof}
Recall the definition of bilinear form $B$
\begin{align*}
    B[u_\delta, u_\delta]= \frac{1}{2\delta^2} \int_{\Omega} \int_{\Omega} &R_{\delta}(\mathbf{x}, \mathbf{y})(u_\delta(\mathbf{x}) - u_\delta(\mathbf{y}))^2 \mathrm{d} \mathbf{x} \mathrm{d} \mathbf{y} \\
    & +\frac{1}{\delta^2} \int_{\partial \Omega} \frac{1}{\bar{\bar{w}}_\delta(\mathbf{x})} \left(\int_{\Omega} \bar{R}_{\delta}(\mathbf{x}, \mathbf{y}) u_\delta(\mathbf{y}) \mathrm{d} \mathbf{y}\right)^2 \mathrm{d} S_{\mathbf{x}}
\end{align*}
The two terms on the left hand side only differs from current inequality in an additional $\frac{1}{2}$, but it can be absorbed in the constant $C$. Thus using the fact that $B$ is coercive, we can get the upper bound for $\|u_\delta\|_{L^2(\Omega)}^2$.
\end{proof}

Next we can show that $\|\nabla u_\delta\|_{L^2(\Omega)}$ is bounded by the left hand side of \Cref{eq:h1_core} plus some extra terms on $F_{\mathrm{in}}$ and $F_{\mathrm{bd}}$.  
\begin{proposition}
\label{prop:gradient_estimation}
    If $u_\delta \in L^2(\Omega)$ is the solution to \Cref{eq:model}, then there exists a constant $C>0$, such that
    \begin{align*}
        & ~C\|\nabla u_\delta\|_{L^2(\Omega)}^2 \\
        \leq {} & \vphantom{\left(\int_{\Omega} \bar{R}_\delta(\mathbf{x}, \mathbf{y}) \right)^2} \frac{1}{2\delta^2}\int_{\Omega} \int_{\Omega} R_{\delta}(\mathbf{x}, \mathbf{y})(u_\delta(\mathbf{x})-u_\delta(\mathbf{y}))^2 \mathrm{d} \mathbf{x} \mathrm{d} \mathbf{y} + \frac{1}{2\delta^2}\int_{\partial \Omega} \frac{1}{\bar{\bar{w}}_\delta(\mathbf{x})} \left(\int_{\Omega} \bar{R}_\delta(\mathbf{x}, \mathbf{y}) u_\delta(\mathbf{y})\mathrm{d} \mathbf{y}\right)^2 \mathrm{d} S_\mathbf{x}\\
        & \vphantom{\left(\int_{\Omega} \bar{R}_\delta(\mathbf{x}, \mathbf{y}) \right)^2} + \delta^3 \| F_{\mathrm{bd}}(\mathbf{x}) \|_{L^2(\partial \Omega)}^2 + \delta^2 \| F_{\mathrm{in}}(\mathbf{x}) \|_{L^2(\Omega)}^2 + \delta^4 \| \nabla F_{\mathrm{in}}(\mathbf{x}) \|_{L^2(\Omega)}^2  \numberthis \label{eq:h1_gradient_estimation}
    \end{align*}
\end{proposition}

\begin{proof}
See \Cref{subsec:4.8}.
\end{proof}

Finally we are ready to prove the second part of \Cref{thm:wellposed}.
\begin{proof}(Proof of Theorem \ref{thm:wellposed}, part 2)
To complete the proof of \Cref{thm:wellposed}, we derive the bound for the norm of $F_{\mathrm{in}}(\mathbf{x})$ and $F_{\mathrm{bd}}(\mathbf{x})$. Using similar techniques as in the proof of \Cref{prop:continuous}, we have
\[\|F_{\mathrm{in}}(\mathbf{x}) \|_{L^2(\Omega)}^2\leq C \left\|f \right\|_{L^2(\Omega)}^2,\quad \|F_{\mathrm{bd}}(\mathbf{x}) \|_{L^2(\partial \Omega)}^2\leq \frac{C}{\delta}  \left\|f \right\|_{L^2(\Omega)}^2\]
\[\|\nabla F_{\mathrm{in}}(\mathbf{x}) \|_{L^2(\Omega)}^2= \int_{\Omega} \left( \int_{\Omega} R_{\delta}(\mathbf{x}, \mathbf{y}) \frac{\mathbf{x}- \mathbf{y}}{2\delta^2} f(\mathbf{y}) \mathrm{d} \mathbf{y} \right)^2 \mathrm{d} \mathbf{x}\leq \frac{C}{\delta^2} \left\|f \right\|_{L^2(\Omega)}^2\]

Combine these estimates with \Cref{eq:h1_core}, \Cref{eq:h1_l2_estimation} and \Cref{eq:h1_gradient_estimation}, we have
\begingroup
\allowdisplaybreaks
\begin{align*}
    \vphantom{\left(\int_{\Omega} \bar{R}_\delta(\mathbf{x}, \mathbf{y}) \right)^2} & C\left\|u_\delta\right\|_{H^1(\Omega)}^2 = C\left\|u_\delta\right\|_{L^2(\Omega)}^2 + C\left\|\nabla u_\delta\right\|_{L^2(\Omega)}^2 \\
    \leq {} & \frac{1}{2\delta^2}\int_{\Omega} \int_{\Omega} R_{\delta}(\mathbf{x}, \mathbf{y})(u_\delta(\mathbf{x})-u_\delta(\mathbf{y}))^2 \mathrm{d} \mathbf{x} \mathrm{d} \mathbf{y} + \frac{1}{2\delta^2}\int_{\partial \Omega} \frac{1}{\bar{\bar{w}}_\delta(\mathbf{x})} \left(\int_{\Omega} \bar{R}_\delta(\mathbf{x}, \mathbf{y}) u_\delta(\mathbf{y})\mathrm{d} \mathbf{y}\right)^2 \mathrm{d} S_\mathbf{x} \\
	& \vphantom{\left(\int_{\Omega} \bar{R}_\delta(\mathbf{x}, \mathbf{y}) \right)^2} + \delta^3 \| F_{\mathrm{bd}}(\mathbf{x}) \|_{L^2(\partial \Omega)}^2 + \delta^2 \| F_{\mathrm{in}}(\mathbf{x}) \|_{L^2(\Omega)}^2 + \delta^4 \| \nabla F_{\mathrm{in}}(\mathbf{x}) \|_{L^2(\Omega)}^2 \\
	\leq {} & \int_{\Omega} u_\delta(\mathbf{x}) F_{\mathrm{in}}(\mathbf{x}) \mathrm{d} \mathbf{x} + \frac{\delta^2}{2}  \int_{\partial \Omega} \frac{F^2_{\mathrm{bd}}(\mathbf{x}) }{\bar{\bar{w}}_\delta(\mathbf{x})} \mathrm{d} S_\mathbf{x} + \delta^3 \| F_{\mathrm{bd}}(\mathbf{x}) \|_{L^2(\partial \Omega)}^2 \\
	& + \vphantom{\left(\int_{\Omega} \bar{R}_\delta(\mathbf{x}, \mathbf{y}) \right)^2} \delta^2 \| F_{\mathrm{in}}(\mathbf{x}) \|_{L^2(\Omega)}^2 + \delta^4 \| \nabla F_{\mathrm{in}}(\mathbf{x}) \|_{L^2(\Omega)}^2 \\
	\leq {} & \vphantom{\left(\int_{\Omega} \bar{R}_\delta(\mathbf{x}, \mathbf{y}) \right)^2} \|u_\delta\|_{L^2(\Omega)} \| F_{\mathrm{in}}(\mathbf{x}) \|_{L^2(\Omega)}+ C\delta^3 \| F_{\mathrm{bd}}(\mathbf{x}) \|_{L^2(\partial \Omega)}^2 \vphantom{\left(\int_{\Omega} \bar{R}_\delta(\mathbf{x}, \mathbf{y}) \right)^2} + \delta^2 \| F_{\mathrm{in}}(\mathbf{x}) \|_{L^2(\Omega)}^2 + \delta^4 \| \nabla F_{\mathrm{in}}(\mathbf{x}) \|_{L^2(\Omega)}^2 \\
	\leq {} & \vphantom{\left(\int_{\Omega} \bar{R}_\delta(\mathbf{x}, \mathbf{y}) \right)^2} C\|u_\delta\|_{H^1(\Omega)} \left\|f \right\|_{L^2(\Omega)}+ C\delta^2 \left\|f \right\|_{L^2(\Omega)}^2
\end{align*}
\endgroup

This gives us $u_{\delta} \in$ $H^1(\Omega)$ with the estimate
\[\|u_{\delta}\|_{H^1(\Omega)} \leq C \|f\|_{L^2(\Omega)}\]
\end{proof}

\section{\texorpdfstring{Convergence (\Cref{thm:convergence})}{Convergence}}
\label{sec:convergence}

In this section, we will prove that solution $(u_\delta, v_\delta)$ of \Cref{eq:model} converges to the solution of the Poisson equation \Cref{eq:poisson}.

Define
\[e_\delta(\mathbf{x}) := u(\mathbf{x})-u_\delta(\mathbf{x}), \quad \mathbf{x} \in \Omega\]
\[\tilde{e}_\delta(\mathbf{x}) := \frac{\partial u}{\partial \mathbf{n}}(\mathbf{x})-v_\delta(\mathbf{x}), \quad \mathbf{x} \in \partial \Omega\]

Similar to \Cref{prop:h1_core}, we can have an energy estimation inequality concerning the error,
\begin{corollary}
\label{corol:convergence_core}
    If $u_\delta$ is the solution to \Cref{eq:model}, then we have the following inequality for error $e_\delta(\mathbf{x})$,
    \begin{align*}
    	\frac{1}{2\delta^2}\int_{\Omega} \int_{\Omega} R_{\delta}(\mathbf{x}, \mathbf{y})&(e_\delta(\mathbf{x})-e_\delta(\mathbf{y}))^2 \mathrm{d} \mathbf{x} \mathrm{d} \mathbf{y}+ \frac{1}{2\delta^2}\int_{\partial \Omega} \frac{1}{\bar{\bar{w}}_\delta(\mathbf{x})} \left(\int_{\Omega} \bar{R}_\delta(\mathbf{x}, \mathbf{y}) e_\delta(\mathbf{y})\mathrm{d} \mathbf{y}\right)^2 \mathrm{d} S_\mathbf{x} \\
    	\leq {} & \int_{\Omega} e_\delta(\mathbf{x}) r_{\mathrm{in}}(\mathbf{x}) \mathrm{d} \mathbf{x} + \frac{\delta^2}{2} \int_{\partial \Omega} \frac{r^2_{\mathrm{bd}}(\mathbf{x}) }{\bar{\bar{w}}_\delta(\mathbf{x})} \mathrm{d} S_\mathbf{x}
    \end{align*}
where for $\mathbf{x} \in \Omega$,
\[r_{\mathrm{in}}(\mathbf{x})=\frac{1}{\delta^2} \int_{\Omega} R_{\delta}(\mathbf{x}, \mathbf{y})(u(\mathbf{x})-u(\mathbf{y})) \mathrm{d} \mathbf{y}-2 \int_{\partial \Omega} \bar{R}_{\delta}(\mathbf{x}, \mathbf{y}) \frac{\partial u}{\partial \mathbf{n}}(\mathbf{y}) \mathrm{d} S_{\mathbf{y}}-\int_{\Omega} \bar{R}_{\delta}(\mathbf{x}, \mathbf{y}) f(\mathbf{y}) \mathrm{d} \mathbf{y}\]
and for $\mathbf{x} \in \partial \Omega$,
\[r_{\mathrm{bd}}(\mathbf{x})=-\frac{1}{\delta^2} \int_{\Omega} \bar{R}_{\delta}(\mathbf{x}, \mathbf{y})u(\mathbf{y}) \mathrm{d} \mathbf{y}-2 \int_{\partial \Omega} \bar{\bar{R}}_{\delta}(\mathbf{x}, \mathbf{y})  \frac{\partial u}{\partial \mathbf{n}}(\mathbf{x}) \mathrm{d} S_{\mathbf{y}}-\int_{\Omega} \bar{\bar{R}}_{\delta}(\mathbf{x}, \mathbf{y}) f(\mathbf{y}) \mathrm{d} \mathbf{y}\]
\end{corollary}
\begin{proof}
In \Cref{eq:model}, move all terms containing $u$ and $\frac{\partial u}{\partial \mathbf{n}}$ to the right hand side, we get
\begin{align}
	\frac{1}{\delta^2} \int_{\Omega} R_{\delta}(\mathbf{x}, \mathbf{y})(e_\delta(\mathbf{x})-e_\delta(\mathbf{y})) \mathrm{d} \mathbf{y}-2 \int_{\partial \Omega} \bar{R}_{\delta}(\mathbf{x}, \mathbf{y}) \tilde{e}_\delta(\mathbf{y}) \mathrm{d} S_{\mathbf{y}} &= r_{\mathrm{in}}(\mathbf{x}), \quad \mathbf{x} \in \Omega \label{eq:convergence_interior} \\
	-\frac{1}{\delta^2} \int_{\Omega} \bar{R}_{\delta}(\mathbf{x}, \mathbf{y})e_\delta(\mathbf{y}) \mathrm{d} \mathbf{y} - 2 \int_{\partial \Omega} \bar{\bar{R}}_{\delta}(\mathbf{x}, \mathbf{y}) \tilde{e}_\delta(\mathbf{x}) \mathrm{d} S_{\mathbf{y}} &= r_{\mathrm{bd}}(\mathbf{x}), \quad \mathbf{x} \in \partial \Omega \label{eq:convergence_boundary}
\end{align}
here $r_{\mathrm{in}}(\mathbf{x})$ and $r_{\mathrm{bd}}(\mathbf{x})$ are interior and boundary truncation error defined for $x\in \Omega$ and $x\in \partial \Omega$ respectively. The remaining part is similar to the proof of \Cref{prop:h1_core}.
\end{proof}

Next corollary, concerning the $H^1$ estimation of $e_\delta$, comes directly from \Cref{corol:l2_estimation} and \Cref{prop:gradient_estimation},
\begin{corollary}
\label{corol:convergence_h1}
    If $u_\delta(\mathbf{x})$ is the solution to \Cref{eq:model}, then the error $e_\delta(\mathbf{x})$ satisfies
    \begin{align*}
    & C \|e_\delta\|_{H^1(\Omega)}^2 \\
    \leq {} & \frac{1}{2\delta^2}\int_{\Omega} \int_{\Omega} R_{\delta}(\mathbf{x}, \mathbf{y})(e_\delta(\mathbf{x})-e_\delta(\mathbf{y}))^2 \mathrm{d} \mathbf{x} \mathrm{d} \mathbf{y} + \frac{1}{2\delta^2}\int_{\partial \Omega} \frac{1}{\bar{\bar{w}}_\delta(\mathbf{x})} \left(\int_{\Omega} \bar{R}_\delta(\mathbf{x}, \mathbf{y}) e_\delta(\mathbf{y})\mathrm{d} \mathbf{y}\right)^2 \mathrm{d} S_\mathbf{x} \\
	& \vphantom{\left(\int_{\Omega} \bar{R}_\delta(\mathbf{x}, \mathbf{y}) \right)^2} + \delta^3 \| r_{\mathrm{bd}}(\mathbf{x}) \|_{L^2(\partial \Omega)}^2 + \delta^2 \| r_{\mathrm{in}}(\mathbf{x}) \|_{L^2(\Omega)}^2 + \delta^4 \| \nabla r_{\mathrm{in}}(\mathbf{x}) \|_{L^2(\Omega)}^2
\end{align*}
where $r_{\mathrm{in}}(\mathbf{x})$ and $r_{\mathrm{bd}}(\mathbf{x})$ are defined as same as in \Cref{corol:convergence_core}.
\end{corollary}

The main difference between this section and last section is that, we should derive consistency results considering the truncation errors $r_{\mathrm{in}}(\mathbf{x})$ and $r_{\mathrm{bd}}(\mathbf{x})$. In fact, $r_{\mathrm{in}}(\mathbf{x})$ and $r_{\mathrm{bd}}(\mathbf{x})$ are closely related. For $\mathbf{x} \in \partial \Omega, u(\mathbf{x})=0$, thus
\begin{align*}
	r_{\mathrm{bd}}(\mathbf{x}) &=\frac{1}{\delta^2} \int_{\Omega} \bar{R}_{\delta}(\mathbf{x}, \mathbf{y})(u(\mathbf{x})-u(\mathbf{y})) \mathrm{d} \mathbf{y}-2 \int_{\partial \Omega} \bar{\bar{R}}_{\delta}(\mathbf{x}, \mathbf{y}) \frac{\partial u}{\partial \mathbf{n}}(\mathbf{y}) \mathrm{d} S_{\mathbf{y}} \\
	&\quad -\int_{\Omega} \bar{\bar{R}}_{\delta}(\mathbf{x}, \mathbf{y}) f(\mathbf{y}) \mathrm{d} \mathbf{y} - 2 \int_{\partial \Omega} \bar{\bar{R}}_{\delta}(\mathbf{x}, \mathbf{y})\left(\frac{\partial u}{\partial \mathbf{n}}(\mathbf{x})-\frac{\partial u}{\partial \mathbf{n}}(\mathbf{y}) \right)  \mathrm{d} S_{\mathbf{y}}\\
	& := \bar{r}_{\mathrm{in}}(\mathbf{x})- 2 \int_{\partial \Omega} \bar{\bar{R}}_\delta(\mathbf{x}, \mathbf{y})\left(\frac{\partial u}{\partial \mathbf{n}}(\mathbf{x})-\frac{\partial u}{\partial \mathbf{n}}(\mathbf{y}) \right)  \mathrm{d} S_{\mathbf{y}}
\end{align*}
	
Here $\bar{r}_{\mathrm{in}}(\mathbf{x})$ is defined similarly as $r_{\mathrm{in}}(\mathbf{x})$ with $R$ replaced by $\bar{R}$, $\bar{R}$ replaced by $\bar{\bar{R}}$. The main result is the following theorem~\cite{shi2017convergence} concerning $r_{\mathrm{in}}(\mathbf{x})$, which can be easily extended to $\bar{r}_{\mathrm{in}}(\mathbf{x})$.
\begin{theorem}[\cite{shi2017convergence}]
    \label{thm:consistency}
    Let $u(\mathrm{x})$ be the solution of the problem \Cref{eq:poisson}. Write
    \[r_{\mathrm{in}}(\mathbf{x})=r_{\mathrm{it}}(\mathbf{x})+r_{\mathrm{bl}}(\mathbf{x})\]
    where
    \[r_{\mathrm{bl}}(\mathbf{x})=\int_{\partial \Omega} \bar{R}_\delta(\mathbf{x}, \mathbf{y}) (\mathbf{x}-\mathbf{y}) \cdot \mathbf{b}(\mathbf{y}) \mathrm{d} S_{\mathbf{y}}\]
    here $\mathbf{b}(\mathbf{y})=\sum_{j=1}^{d}n^{j}(\mathbf{y})\cdot \nabla\left(\nabla^{j} u(\mathbf{y})\right)$, $\mathbf{n}(\mathbf{y})=\left(n^{1}(\mathbf{y}), \cdots, n^{d}(\mathbf{y})\right)$ is the out normal vector of $\partial \Omega$ at $\mathbf{y}, \nabla^{j}$ is the $j$ th component of gradient $\nabla$.
    If $u\in H^3({\Omega})$, then there exists constants $C, \delta_{0}$ depending only on $\Omega$, so that for $\delta \leq \delta_0$.
    \[\left\|r_{\mathrm{it}}\right\|_{L^{2}(\Omega)} \leq C \delta\|u\|_{H^{3}(\Omega)}, \quad \left\|\nabla r_{\mathrm{it}} \right\|_{L^{2}(\Omega)} \leq C\|u\|_{H^{3}(\Omega)}\]
\end{theorem}

Using trace theorem and \Cref{assumption:kernel} (c) compact support property, $\left\|r_{\mathrm{bl}}\right\|_{L^{2}(\Omega)}$ can be controlled,
\begin{align*}
	\left\|r_{\mathrm{bl}}(\mathbf{x})\right\|_{L^{2}(\Omega)}^2 & =\left\| \int_{\partial \Omega} \bar{R}_\delta(\mathbf{x}, \mathbf{y})(\mathbf{x}-\mathbf{y}) \cdot \mathbf{b}(\mathbf{y}) \mathrm{d} S_{\mathbf{y}} \right\|_{L^{2}(\Omega)}^2 \\
	& \leq C_1\delta^2 \int_{\Omega}\left(\int_{\partial \Omega} \bar{R}_\delta(\mathbf{x}, \mathbf{y}) \mathbf{b}(\mathbf{y}) \mathrm{d} S_{\mathbf{y}} \right)^2 \mathrm{d} {\mathbf{x}}\\
	& \leq C_1\delta^2 \int_{\Omega}\left(\int_{\partial \Omega} \bar{R}_\delta(\mathbf{x}, \mathbf{y}) \mathrm{d} S_{\mathbf{y}} \right)\left(\int_{\partial \Omega} \bar{R}_\delta(\mathbf{x}, \mathbf{y}) \mathbf{b}^2(\mathbf{y}) \mathrm{d} S_{\mathbf{y}} \right) \mathrm{d} {\mathbf{x}} \\
	& \vphantom{\left(\int_{\partial \Omega} \bar{R}_\delta(\mathbf{x}, \mathbf{y}) \mathrm{d} {\mathbf{y}} \right)} \leq C_2 \delta \left\|\mathbf{b}\right\|_{L^{2}(\partial \Omega)}^2 \leq C_3 \delta\left\|\mathbf{b}\right\|_{H^{1}(\Omega)}^2 \leq C \delta \left\|u\right\|_{H^{3}(\Omega)}^2
\end{align*}
and $\left\|\nabla r_{\mathrm{bl}}(\mathbf{x})\right\|_{L^{2}(\Omega)}^2 \leq C \left\|u\right\|_{H^{3}(\Omega)}^2 $ can be derived similarly.

Next we provide a theorem concerning the truncation error on the boundary $r_{\mathrm{bd}}(\mathbf{x})$. 
\begin{theorem}
    \label{thm:boundary_consistency}
    Let $u(\mathrm{x})$ be the solution of the problem \Cref{eq:poisson}. If $u\in H^3({\Omega})$, then there exists constants $C, \delta_{0}$ depending only on $\Omega$, so that for $\delta \leq \delta_0$.
    \[\left\|r_{\mathrm{bd}}(\mathbf{x})\right\|_{L^{2}(\partial \Omega)}\leq C\left\|u\right\|_{H^3(\Omega)}  \]
\end{theorem}
\begin{proof}
See \Cref{subsec:5.4}.
\end{proof}

Furthermore, we need a theorem to utilize the special structure of $r_{\mathrm{bl}}(\mathbf{x})$ in \Cref{thm:consistency}.
\begin{theorem}[\cite{li2017point}]
    \label{thm:boundary}
    Let $g\in H^1({\Omega})$
    \[r(\mathbf{x})=\int_{\partial \Omega} \bar{R}_\delta(\mathbf{x}, \mathbf{y}) (\mathbf{x}-\mathbf{y}) \cdot \mathbf{b}(\mathbf{y}) \mathrm{d} S_{\mathbf{y}}\]
    If $\mathbf{b}\in H^1({\Omega})$, then there exists constants $C$ depending only on $\Omega$, so that
    \[\left|\int_{\Omega} g(\mathbf{x}) r(\mathbf{x}) \mathrm{d} {\mathbf{x}}\right| \leq C \delta \left\|\mathbf{b}\right\|_{H^{1}(\Omega)}\left\|g\right\|_{H^{1}(\Omega)}\]
\end{theorem}

Using \Cref{thm:boundary}, we choose $g=e_\delta$, and $\mathbf{b}$ defined as in \Cref{thm:consistency}. Notice that we have proved $u_\delta \in H^1(\Omega)$ in \Cref{thm:wellposed}, so $e_\delta=u-u_\delta$ indeed belongs to $H^1(\Omega)$. Then we get
\[\int_{\Omega} r_{\mathrm{bl}}(\mathbf{x})e_\delta(\mathbf{x}) \mathrm{d} \mathbf{x} \leq C_1\delta \left\|\mathbf{b}\right\|_{H^{1}(\Omega)} \left\|e_\delta\right\|_{H^1(\Omega)}\leq C\delta\left\|u\right\|_{H^{3}(\Omega)} \left\|e_\delta\right\|_{H^1(\Omega)} \]
	
Finally, combine the consistency results \Cref{thm:consistency}, \Cref{thm:boundary_consistency} with \Cref{corol:convergence_core} and \Cref{corol:convergence_h1}, we can prove \Cref{thm:convergence}.
\begin{proof}(Proof of \Cref{thm:convergence})
    \begin{align*}
    & ~\|e_\delta\|_{H^1(\Omega)}^2 \\
    \leq {} & \frac{C}{2\delta^2} \int_{\Omega} \int_{\Omega} R_{\delta}(\mathbf{x}, \mathbf{y})(e_\delta(\mathbf{x})-e_\delta(\mathbf{y}))^2 \mathrm{d} \mathbf{x} \mathrm{d} \mathbf{y} + \frac{C}{2\delta^2}\int_{\partial \Omega} \frac{1}{\bar{\bar{w}}_\delta(\mathbf{x})} \left(\int_{\Omega} \bar{R}_\delta(\mathbf{x}, \mathbf{y}) e_\delta(\mathbf{y})\mathrm{d} \mathbf{y}\right)^2 \mathrm{d} S_\mathbf{x} \\
	& \vphantom{\left(\int_{\Omega} \bar{R}_\delta(\mathbf{x}, \mathbf{y}) \right)^2} + C\delta^3 \| r_{\mathrm{bd}}(\mathbf{x}) \|_{L^2(\partial \Omega)}^2 + C\delta^2 \| r_{\mathrm{in}}(\mathbf{x}) \|_{L^2(\Omega)}^2 + C\delta^4 \| \nabla r_{\mathrm{in}}(\mathbf{x}) \|_{L^2(\Omega)}^2\\
	\leq {} & C\int_{\Omega} e_\delta(\mathbf{x}) r_{\mathrm{it}}(\mathbf{x}) \mathrm{d} \mathbf{x} + C\int_{\Omega} e_\delta(\mathbf{x}) r_{\mathrm{bl}}(\mathbf{x}) \mathrm{d} \mathbf{x} + C \delta^3 \left\|r_{\mathrm{bd}}(\mathbf{x})\right\|_{L^{2}(\partial \Omega)}^2 +  \\
	& \quad  C\delta^2 \| r_{\mathrm{it}}(\mathbf{x}) \|_{L^2(\Omega)}^2 + C\delta^2 \| r_{\mathrm{bl}}(\mathbf{x}) \|_{L^2(\Omega)}^2 + C\delta^4 \| \nabla r_{\mathrm{it}}(\mathbf{x}) \|_{L^2(\Omega)}^2+ C\delta^4 \| \nabla r_{\mathrm{bl}}(\mathbf{x}) \|_{L^2(\Omega)}^2\\
	\leq {}& \vphantom{\int} C\|e_\delta\|_{L^2(\Omega)} \|r_{\mathrm{it}}(\mathbf{x})\|_{L^2(\Omega)}+C\delta\|e_\delta\|_{H^1(\Omega)}\|u\|_{H^{3}(\Omega)} \vphantom{\int} + C\delta^3\|u\|_{H^{3}(\Omega)}^2 + C\delta^4\|u\|_{H^{3}(\Omega)}^2 \\
	\leq {} & \vphantom{\int} C\delta\|e_\delta\|_{H^1(\Omega)} \|u\|_{H^{3}(\Omega)}+ C\delta^3 \|u\|_{H^{3}(\Omega)}^2
\end{align*}

This gives us
\[\left\|e_\delta\right\|_{H^1(\Omega)}\leq C\delta\left\|u\right\|_{H^{3}(\Omega)} \]

The convergence rate of $v_\delta$ is an easy corollary. Multiply \eqref{eq:convergence_interior} by $e_\delta(\mathbf{x})$ and integrate over $\Omega$, multiply \eqref{eq:convergence_boundary} by $\tilde{e}_\delta(\mathbf{x})$ and integrate over $\partial \Omega$, then add two resulting equations. We get
\begin{align*}
    \frac{1}{2\delta^2}\int_{\Omega} \int_{\Omega} R_{\delta}(\mathbf{x}, \mathbf{y})& (e_\delta(\mathbf{x})-e_\delta(\mathbf{y}))^2 \mathrm{d} \mathbf{x} \mathrm{d} \mathbf{y} + 4\delta^2 \int_{\partial \Omega}\int_{\partial \Omega} \bar{\bar{R}}_\delta(\mathbf{x}, \mathbf{y}) \tilde{e}_\delta(\mathbf{x})^2 \mathrm{d} S_\mathbf{x} \mathrm{d} S_{\mathbf{y}} \\
    &=\int_{\Omega}e_\delta(\mathbf{x})r_{\mathrm{in}}(\mathbf{x}) \mathrm{d} \mathbf{x}-2\delta^2\int_{\partial \Omega}\tilde{e}_\delta(\mathbf{x})r_{\mathrm{bd}}(\mathbf{x}) \mathrm{d} S_\mathbf{x} \numberthis \label{eq:convergence_addeq}
\end{align*}

We simply put the following term in \Cref{eq:convergence_addeq} to zero
\[\frac{1}{2\delta^2}\int_{\Omega} \int_{\Omega} R_{\delta}(\mathbf{x}, \mathbf{y})(e_\delta(\mathbf{x})-e_\delta(\mathbf{y}))^2 \mathrm{d} \mathbf{x} \mathrm{d} \mathbf{y} \geq 0 \]

Then
\begin{align*}
	& C\delta \|\tilde{e}_\delta\|_{L^2(\partial \Omega)} ^2\leq 4\delta^2 \int_{\partial \Omega}\int_{\partial \Omega} \bar{\bar{R}}_\delta(\mathbf{x}, \mathbf{y}) \tilde{e}_\delta(\mathbf{x})^2 \mathrm{d} S_\mathbf{x} \mathrm{d} S_{\mathbf{y}} \\
	&\leq \int_{\Omega}e_\delta(\mathbf{x})r_{\mathrm{in}}(\mathbf{x}) \mathrm{d} \mathbf{x}-2\delta^2\int_{\partial \Omega}\tilde{e}_\delta(\mathbf{x})r_{\mathrm{bd}}(\mathbf{x}) \mathrm{d} S_\mathbf{x} \\
	& \leq \left|\int_{\Omega} e_\delta(\mathbf{x}) r_{\mathrm{it}}(\mathbf{x}) \mathrm{d} \mathbf{x} \right|+ \left|\int_{\Omega} e_\delta(\mathbf{x}) r_{\mathrm{bl}}(\mathbf{x}) \mathrm{d} \mathbf{x} \right|+ 2 \delta^2 \left\|r_{\mathrm{bd}}(\mathbf{x})\right\|_{L^{2}(\partial \Omega)}\left\|\tilde{e}_\delta\right\|_{L^{2}(\partial \Omega)} \\
	& \vphantom{\int} \leq C\|e_\delta\|_{L^2(\Omega)} \|r_{\mathrm{it}}(\mathbf{x})\|_{L^2(\Omega)}+C\delta\|e_\delta\|_{H^1(\Omega)}\|u\|_{H^{3}(\Omega)} + 2 \delta^2 \left\|r_{\mathrm{bd}}(\mathbf{x})\right\|_{L^{2}(\partial \Omega)}\left\|\tilde{e}_\delta\right\|_{L^{2}(\partial \Omega)} \\
	& \vphantom{\int} \leq C \delta^2 \|u\|_{H^{3}(\Omega)}^2 + C \delta^2 \|u\|_{H^{3}(\Omega)}\left\|\tilde{e}_\delta\right\|_{L^{2}(\partial \Omega)}
\end{align*}
where in the last inequality, we use the first order $H^1$ convergence that we have proved. Thus
\[\left\|\tilde{e}_\delta\right\|_{L^{2}(\partial \Omega)} \leq C \delta^{\frac{1}{2}} \|u\|_{H^{3}(\Omega)} \]
The result follows from the elliptic regularity of Poisson's equation, $f\in H^1(\Omega)$ implies $u\in H^3(\Omega)$, with the energy estimate.
\end{proof}

\section{Robin boundary condition}
\label{sec:robin}

In this section, we analyze the extension to Robin boundary condition. The differential equation now becomes
\begin{equation}
	\left\{\begin{aligned}
		-\Delta u(\mathbf{x}) &=f(\mathbf{x}), & & \mathbf{x} \in \Omega \\[3pt]
		u(\mathbf{x})+\mu \frac{\partial u}{\partial \mathbf{n}}&(\mathbf{x})  =0, & & \mathbf{x} \in \partial \Omega
	\end{aligned}\right.
	\label{eq:poisson_robin}
\end{equation}
 where $\mu>0$ is a given constant. We propose to change our model as
\[\frac{1}{\delta^2} \int_{\Omega} R_{\delta}(\mathbf{x}, \mathbf{y})(u_\delta(\mathbf{x})-u_\delta(\mathbf{y})) \mathrm{d} \mathbf{y}-2 \int_{\partial \Omega} \bar{R}_{\delta}(\mathbf{x}, \mathbf{y}) v_\delta(\mathbf{y}) \mathrm{d} S_{\mathbf{y}}= \int_{\Omega}  \bar{R}_{\delta}(\mathbf{x}, \mathbf{y}) f(\mathbf{y})\mathrm{d} \mathbf{y}\]
for $\mathbf{x}\in \Omega$, and
\[-\frac{1}{\delta^2} \int_{\Omega} \bar{R}_{\delta}(\mathbf{x}, \mathbf{y}) (\mu v_\delta(\mathbf{x})+u_\delta(\mathbf{y})) \mathrm{d} \mathbf{y}-2 \int_{\partial \Omega} \bar{\bar{R}}_{\delta}(\mathbf{x}, \mathbf{y}) v_\delta(\mathbf{x}) \mathrm{d} S_{\mathbf{y}}= \int_{\Omega}  \bar{\bar{R}}_{\delta}(\mathbf{x}, \mathbf{y}) f(\mathbf{y})\mathrm{d} \mathbf{y}\]
for $\mathbf{x}\in \partial\Omega$. We can similarly write the model as
\begin{equation}
	\left\{\begin{aligned}
		\mathcal{L}_{\delta} u_{\delta}(\mathbf{x})-\mathcal{G}_{\delta} v_{\delta}(\mathbf{x})&=F_{\mathrm{in}}(\mathbf{x}), && \mathbf{x} \in \Omega \\[5pt] \bar{\mathcal{L}}_{\delta} u_{\delta}(\mathbf{x})-\bar{\mathcal{G}}_{\delta} v_{\delta}(\mathbf{x}) &=F_{\mathrm{bd}}(\mathbf{x}), && \mathbf{x} \in \partial \Omega
	\end{aligned}\right.
	\label{eq:model_robin}
\end{equation}
where the only difference from the original model~\Cref{eq:model} with Dirichlet boundary is the definition of $\bar{\mathcal{L}}_{\delta}$,
\[\bar{\mathcal{L}}_{\delta} u_{\delta}(\mathbf{x})= - \frac{1}{\delta^{2}} \int_{\Omega} \bar{R}_{\delta}(\mathbf{x}, \mathbf{y}) u_{\delta}(\mathbf{y}) \mathrm{d} {\mathbf{y}} \quad \longrightarrow \quad -\frac{1}{\delta^2} \int_{\Omega} \bar{R}_{\delta}(\mathbf{x}, \mathbf{y}) (\mu v_\delta(\mathbf{x})+u_\delta(\mathbf{y})) \mathrm{d} \mathbf{y}\]

Then we can give similar theorems as for Dirichlet boundary condition.
\begin{theorem}[Well-Posedness (Robin)]
    \label{thm:wellposed_robin}
    For fixed $\delta>0$ and $f \in L^2(\Omega)$, there exists a unique solution $u_{\delta} \in$ $L^{2}(\Omega), v_{\delta} \in L^{2}(\partial \Omega)$ to the integral model \Cref{eq:model_robin}.
    
    Moreover, we have $u_{\delta} \in$ $H^1(\Omega)$ and the following estimate, with constant $C>0$ independent of $\delta$,
    \[\left\|u_{\delta}\right\|_{H^1(\Omega)} \leq C \left\|f\right\|_{L^2(\Omega)}\]
\end{theorem}
\begin{proof}
    Similar to former proof for \Cref{thm:wellposed}, we expand \eqref{eq:model_robin} and get
    \[v_\delta(\mathbf{x})=-\frac{1}{2\delta^2\hat{w}_\delta(\mathbf{x})} \int_{\Omega} \bar{R}_{\delta}(\mathbf{x}, \mathbf{y}) u_\delta(\mathbf{y}) \mathrm{d} \mathbf{y}-\frac{F_{\mathrm{bd}}(\mathbf{x})}{2\hat{w}_\delta(\mathbf{x})}, \quad \mathbf{x}\in \partial\Omega\] 
    
    where
    \[\hat{w}_\delta(\mathbf{x})=\bar{\bar{w}}_\delta(\mathbf{x})+\frac{\mu}{2\delta^2}\bar{w}_\delta(\mathbf{x})\]
    
    Then in the proof of well-posedness, we only have to replace all $\bar{\bar{w}}_\delta(\mathbf{x})$ with $\hat{w}_\delta(\mathbf{x})$. It is easy to check that all conclusions still hold.
\end{proof}

\begin{theorem}[Convergence (Robin)]
    \label{thm:convergence_robin}
    Let $f \in H^{1}(\Omega)$, $u$ be the solution to the Poisson model \Cref{eq:poisson_robin}, and $\left(u_{\delta}, v_{\delta}\right)$ be the solution to the integral model \Cref{eq:model_robin}, then we have the following estimate, with constant $C>0$ independent of $\delta$,
    \[\left\|u-u_{\delta}\right\|_{H^{1}(\Omega)}\leq C \delta\|f\|_{H^{1}(\Omega)}\]
    \[\left\|\frac{\partial u}{\partial \mathbf{n}}-v_{\delta}\right\|_{L^{2}(\partial \Omega)} \leq C \delta\|f\|_{H^{1}(\Omega)}\]
\end{theorem}

\begin{proof}
Note that the order of convergence on the boundary becomes $O(\delta)$, which is different from $O(\delta^{\frac{1}{2}})$ in \Cref{thm:convergence}. This is because the dominating term in $\hat{w}_\delta(\mathbf{x})$ becomes $\frac{\mu}{2\delta^2}\bar{w}_\delta(\mathbf{x})$, which is $O(\delta^{-2})$, rather than $\bar{\bar{w}}_\delta(\mathbf{x})$, which is $O(\delta^{-1})$, we need to modify several inequalities in the proof for Dirichlet boundary (\Cref{thm:convergence}). The detailed analysis is provided in \Cref{subsec:6.2}.
\end{proof}

\section{Laplacian spectra}
\label{sec:spectra}

Our convergence results for the Poisson equation can be applied to the analysis of Laplacian spectra. Specifically, we want to solve the following equation
\begin{equation}
	\left\{\begin{aligned}
		-\Delta u(\mathbf{x}) &= \lambda u(\mathbf{x}), & & \mathbf{x} \in \Omega \\[3pt]
		u(\mathbf{x}) & = 0, & & \mathbf{x} \in \partial \Omega
	\end{aligned}\right.
	\label{eq:spectra}
\end{equation}
to get the eigenvalues $\lambda$ of Laplacian operator $\Delta$. We discretize the Laplacian operator in a similar way as in \Cref{eq:model}, i.e.
\begin{equation}
	\left\{\begin{aligned}
		\mathcal{L}_{\delta} u_{\delta}(\mathbf{x})-\mathcal{G}_{\delta} v_{\delta}(\mathbf{x})&=\lambda F_{\mathrm{in}}(\mathbf{x}), && \mathbf{x} \in \Omega \\[5pt] \bar{\mathcal{L}}_{\delta} u_{\delta}(\mathbf{x})-\bar{\mathcal{G}}_{\delta} v_{\delta}(\mathbf{x}) &= 0, && \mathbf{x} \in \partial \Omega
	\end{aligned}\right.
	\label{eq:spectra_model}
\end{equation}
where the operators are defined as in \Cref{eq:model}. 
\begin{remark}
    Notice that in the second equation of \Cref{eq:spectra_model}, we set the right hand side to zero, which is different from \Cref{eq:model}. The reason is that we want symmetry in the following proof. Moreover, we will show that changing this term will not affect the convergence rate.
\end{remark}

% \begin{gather*}
%     \mathcal{L}_{\delta} u_{\delta}(\mathbf{x})=\frac{1}{\delta^{2}} \int_{\Omega}R_{\delta}(\mathbf{x}, \mathbf{y}) \left(u_{\delta}(\mathbf{x})-u_{\delta}(\mathbf{y})\right) \mathrm{d} {\mathbf{y}}\\
%     \bar{\mathcal{L}}_{\delta} u_{\delta}(\mathbf{x})=-\frac{1}{\delta^{2}} \int_{\Omega}\bar{R}_{\delta}(\mathbf{x}, \mathbf{y}) u_{\delta}(\mathbf{y}) \mathrm{d} {\mathbf{y}} \\
%     \mathcal{G}_{\delta} v_{\delta}(\mathbf{x})=2 \int_{\partial \Omega} \bar{R}_{\delta}(\mathbf{x}, \mathbf{y}) v_{\delta}(\mathbf{y}) \mathrm{d} {\mathbf{y}}, \quad \bar{\mathcal{G}}_{\delta} v_{\delta}(\mathbf{x})=2 \int_{\partial \Omega} \bar{\bar{R}}_{\delta}(\mathbf{x}, \mathbf{y}) v_{\delta}(\mathbf{x}) \mathrm{d} {\mathbf{y}} \\
%     \mathcal{P}_{\delta} u_\delta(\mathbf{x})=\int_{\Omega} \bar{R}_{\delta}(\mathbf{x}, \mathbf{y}) u_\delta(\mathbf{y}) \mathrm{d} {\mathbf{y}}, \quad \bar{\mathcal{P}}_{\delta} u_\delta(\mathbf{x})=\int_{\Omega} \bar{\bar{R}}_{\delta}(\mathbf{x}, \mathbf{y}) u_\delta(\mathbf{y}) \mathrm{d} {\mathbf{y}} 
% \end{gather*}

We start by defining two operators
\begin{itemize}
    \item $T: L^{2}(\Omega) \rightarrow H^{2}(\Omega)$ is the solution operator of the following differential equation, i.e., $u=T(f)$ solves:
\end{itemize}
    \[\left\{\begin{aligned}
    - \Delta u(\mathbf{x}) &=f(\mathbf{x}) & & \mathbf{x} \in \Omega \\[5pt]
    u(\mathbf{x}) &=0, & & \mathbf{x} \in \partial \Omega
    \end{aligned}\right.\]
    % with $\int_{\Omega} u(\mathbf{x})\mathrm{d} \mathbf{x}=0$
    % where $\bar{f}=\frac{1}{|\Omega|}\int_{\Omega}f(\mathbf{x})\mathrm{d} \mathbf{x}$ is a constant.
\begin{itemize}
    \item $T_\delta: L^{2}(\Omega) \rightarrow L^{2}(\Omega)$ is the solution operator of following integral equation, i.e. $u_\delta=T_\delta(f)$ solves
\end{itemize}
\begin{align*}
	\frac{1}{\delta^2} \int_{\Omega} R_{\delta}(\mathbf{x}, \mathbf{y})(u_\delta(\mathbf{x})-u_\delta(\mathbf{y})) \mathrm{d} \mathbf{y}-2 \int_{\partial \Omega} \bar{R}_{\delta}(\mathbf{x}, \mathbf{y}) v_\delta(\mathbf{y}) \mathrm{d} S_{\mathbf{y}} & = \int_{\Omega} \bar{R}_{\delta}(\mathbf{x}, \mathbf{y}) f(\mathbf{y}) \mathrm{d} {\mathbf{y}}, \quad \mathbf{x} \in \Omega\\
	-\frac{1}{\delta^2} \int_{\Omega} \bar{R}_{\delta}(\mathbf{x}, \mathbf{y})u_\delta(\mathbf{y}) \mathrm{d} \mathbf{y}-2 \int_{\partial \Omega} \bar{\bar{R}}_{\delta}(\mathbf{x}, \mathbf{y}) v_\delta(\mathbf{x}) \mathrm{d} S_{\mathbf{y}} & = 0, \qquad \mathbf{x} \in \partial \Omega
\end{align*}

It is easy to see that the equations $Tu=\lambda u$ and $T_\delta u_\delta=\lambda u_\delta$ is equivalent to the eigen problems \Cref{eq:spectra} and \Cref{eq:spectra_model} respectively. Namely their eigenvalues are reciprocal to each other and they share the same eigenspaces. The advantage of using the solution operators is that they are compact operators.

\begin{proposition}
\label{prop:operators}
    For any $\delta>0, T, T_\delta$ are compact operators from $H^{1}(\Omega)$ into $H^{1}(\Omega)$. All eigenvalues of $T, T_\delta$ are real numbers. All generalized eigenvectors of $T, T_\delta$ are eigenvectors.
\end{proposition}
\begin{proof}
    See \Cref{subsec:7.2}.
\end{proof}

Next, using the main results \Cref{thm:wellposed} and \Cref{thm:convergence} in this paper, we can get the following theorem bounding the norm of operators.
\begin{theorem}
    \label{thm:spectra_convergence}
    Under the assumptions in \Cref{assumption:kernel}, there exists a constant $C>0$ only depends on $\Omega$ and the kernel function $R$, such that
    \[\left\|T-T_\delta\right\|_{H^{1}} \leq C \delta, \quad \left\|T_\delta\right\|_{H^{1}} \leq C\]
\end{theorem}
\begin{proof}
    $\left\|T_\delta\right\|_{H^{1}} \leq C$ is obvious using the $H^1$ estimation of $u_\delta$ in \Cref{thm:wellposed}, given that $\|f\|_{L^2(\Omega)}\leq \|f\|_{H^1(\Omega)}$.
    
    As for $\left\|T-T_\delta\right\|_{H^{1}} \leq C \delta$, slight modification is necessary since we change the right hand side on the boundary in our model \Cref{eq:spectra_model}. The new zero term results in the change of $r_{\mathrm{bd}}(\mathbf{x})$, which now becomes
    \begin{align*}
        \tilde{r}_{\mathrm{bd}}(\mathbf{x}) & =-\frac{1}{\delta^2} \int_{\Omega} \bar{R}_{\delta}(\mathbf{x}, \mathbf{y})u(\mathbf{y}) \mathrm{d} \mathbf{y}-2 \int_{\partial \Omega} \bar{\bar{R}}_{\delta}(\mathbf{x}, \mathbf{y})  \frac{\partial u}{\partial \mathbf{n}}(\mathbf{x}) \mathrm{d} S_{\mathbf{y}} \\ 
        & = r_{\mathrm{bd}}(\mathbf{x}) + \int_{\Omega} \bar{\bar{R}}_{\delta}(\mathbf{x}, \mathbf{y}) f(\mathbf{y}) \mathrm{d} \mathbf{y}
    \end{align*}
    
    We have already proved that \[\|r_{\mathrm{bd}}(\mathbf{x})\|_{L^{2}(\partial \Omega)} \leq C\|f\|_{H^1(\Omega)}\]
    
    and the other term is a lower order term
    \[\left\|\int_{\Omega} \bar{\bar{R}}_{\delta}(\mathbf{x}, \mathbf{y}) f(\mathbf{y}) \mathrm{d} \mathbf{y}\right\|_{L^{2}(\partial \Omega)}^2\leq \frac{C}{\delta} \|f\|_{L^2(\Omega)} \leq \frac{C}{\delta} \|f\|_{H^1(\Omega)}\]
    The dominating term now becomes $O(1/\delta)$. Nonetheless, the convergence results still hold due to the high order of $\delta$ prefactors in \Cref{corol:convergence_h1}. The $\delta^3$ prefactor before $\|r_{\mathrm{bd}}(\mathbf{x})\|^2_{L^{2}(\partial \Omega)}$ now becomes $\delta^2$ due to the extra $1/\delta$, which gives us
    \[\|e_\delta\|_{H^1(\Omega)}^2 \leq C\delta\|e_\delta\|_{H^1(\Omega)} \|u\|_{H^{3}(\Omega)}+ C\delta^2 \|u\|_{H^{3}(\Omega)}^2\]
    
    Then we still have
    \[\left\|u-u_{\delta}\right\|_{H^{1}(\Omega)}\leq C \delta\|f\|_{H^{1}(\Omega)}\]
    
    Thus
    \[\left\|T-T_\delta\right\|_{H^{1}} \leq C \delta\]
\end{proof}

% \begin{align*}
    %     &\left\|\int_{\Omega} \bar{\bar{R}}_{\delta}(\mathbf{x}, \mathbf{y}) f(\mathbf{y}) \mathrm{d} \mathbf{y}\right\|_{L^{2}(\partial \Omega)}^2\\
    %     =& \int_{\partial \Omega} \left(\int_{\Omega} \bar{\bar{R}}_{\delta}(\mathbf{x}, \mathbf{y}) f(\mathbf{y}) \mathrm{d} \mathbf{y} \right)^2 \mathrm{d} \mathbf{x}\\
    %     \leq & \int_{\partial \Omega} \left(\int_{\Omega} \bar{\bar{R}}_{\delta}(\mathbf{x}, \mathbf{y}) \mathrm{d} \mathbf{y} \right) \left(\int_{\Omega} \bar{\bar{R}}_{\delta}(\mathbf{x}, \mathbf{y}) f^2(\mathbf{y}) \mathrm{d} \mathbf{y} \right) \mathrm{d} \mathbf{x} \\
    %     \leq & C_1 \int_{\partial \Omega} \int_{\Omega} \bar{\bar{R}}_{\delta}(\mathbf{x}, \mathbf{y}) f^2(\mathbf{y}) \mathrm{d} \mathbf{y} \mathrm{d} \mathbf{x}\\
    %     \leq &C\delta \|f\|_{L^2(\Omega)} \leq C\delta \|f\|_{H^1(\Omega)}
    % \end{align*}

Finally, we are ready to derive the convergence results for eigenvalues and eigenfunctions. The following results are stated with the help of the Riesz spectral projection. Let $X$ be a complex Banach space and $L: X \rightarrow X$ be a compact linear operator. The resolvent set $\rho(L)$ is given by the complex numbers $z \in \mathbb{C}$ such that $z-L$ is bijective. The spectrum of $L$ is $\sigma(L)=\mathbb{C} \backslash \rho(L)$. It is well known that $\sigma(L)$ is a countable set with no limit points other than zero. All non-zero values in $\sigma(L)$ are eigenvalues. If $\lambda$ is a nonzero eigenvalue of $L$, and given a closed smooth curve $\Gamma \subset \rho(L)$ which encloses the eigenvalue $\lambda$ and no other elements of $\sigma(L)$, the Riesz spectral projection associated with $\lambda$ is defined by
\[E(\lambda, L)=\frac{1}{2 \pi i} \int_{\Gamma}(z-L)^{-1} \mathrm{~d} z\]
% where $i=\sqrt{-1}$ is the unit imaginary.
Then we have the following theorem, giving us the convergence results of eigenvalues and eigenfunctions. The proof is similar to Theorem 4.3 in \cite{tao2020convergence}. Nonetheless, we provide the proof in \Cref{subsec:7.4} for completeness.

\begin{theorem}[\cite{tao2020convergence}]
    \label{thm:spectra_eigenfunction}
    Let $\lambda_{m}$ be the m-th largest eigenvalue of $T$ with multiplicity $\alpha_{m}$ and $\phi_{m}^{k}, k=1, \cdots, \alpha_{m}$ be the normalized eigenfunctions corresponding to $\lambda_{m}$. Let $\lambda_{m}^\delta$ be the m-th largest eigenvalue of $T_\delta$. Let $\gamma_{m}=\min _{j \leq m}\left|\lambda_{j}-\lambda_{j+1}\right|$ and
    \[\left\|\left(T-T_\delta\right) T_\delta\right\|_{H^{1}(\Omega)} \leq\left(\left|\lambda_{m}\right|-\gamma_{m} / 3\right) \gamma_{m} / 3\]
    Then there exists a constant $C$ depend on $\gamma_{m}$ and $\lambda_{m}$, such that
    \[\left|\lambda_m^{\delta}-\lambda_m\right| \leq 2 \left\|T-T_\delta\right\|_{H^1(\Omega)}\]
    \[\left\|\phi_{m}^{k}-E\left(\sigma_{m}^{\delta}, T_{\delta}\right) \phi_{m}^{k}\right\|_{H^{1}(\Omega)} \leq C\left(\left\|T-T_\delta\right\|_{H^{1}(\Omega)}+\left\|\left(T-T_\delta\right) T_\delta\right\|_{H^{1}(\Omega)}\right)\]
    Here $\sigma_{m}^{\delta}=\left\{\lambda_{j}^{\delta} \in \sigma\left(T_\delta\right): j \in I_{m}\right\}$ and $I_{m}=\left\{j \in \mathbb{N}: \lambda_{j}=\lambda_{m}\right\}$.
\end{theorem}

Combine \Cref{thm:spectra_convergence} with \Cref{thm:spectra_eigenfunction}, we can easily get first order convergence of both eigenvalues and eigenfunctions.
\begin{corollary}
Let $\lambda_{m}$ be the m-th largest eigenvalue of $T$ with multiplicity $\alpha_{m}$ and $\phi_{m}^{k}, k=1, \cdots, \alpha_{m}$ be the eigenfunctions corresponding to $\lambda_{m}$. Let $\lambda_{m}^\delta$ be the m-th largest eigenvalue of $T_\delta$. Then
\[\left|\lambda_{m}-\lambda_{m}^\delta \right|\leq C\delta, \quad \left\|\phi_{m}^{k}-E\left(\sigma_{m}^{\delta}, T_{\delta}\right) \phi_{m}^{k}\right\|_{H^{1}(\Omega)} \leq C\delta\]
\end{corollary}

\section{Conclusion}
\label{sec:conclusion}

In this work, we propose a nonlocal model to enforce the local Dirichlet boundary condition on Poisson equations. We prove the well-posedness of the proposed model and the first order convergence in $H^1$ norm. Our model provides a general framework to handle Dirichlet boundary condition for nonlocal diffusion problem with smooth domain in any dimension. The nonlocal model and the analysis can be naturally extended to manifold also.
The convergence rate of the proposed nonlocal model is first order in $H^1$ norm which is not optimal (second order). We are also working on the nonlocal diffusion model with optimal convergence rate for Dirichlet boundary condition. The results will be reported in the future paper.

\bibliographystyle{plain}
\bibliography{references}

\newpage
\appendix
\begin{center}
{\Large \bf Appendices}
\end{center}

\section{\texorpdfstring{Proof in \Cref{sec:prelim}}{Proof in Section 2}}
\subsection{\texorpdfstring{Proof of \Cref{prop:estimate}}{Proof of Propositon 2.1}}
\label{subsec:a.1}
\begin{proof}
	First we consider $\int_\Omega \tilde{R}_\delta(\mathbf{x},\mathbf{y})$. The upper bound is easy to prove using the non-negativity of $\tilde{R}$.
	\[\int_\Omega \tilde{R}_\delta(\mathbf{x},\mathbf{y})\mathrm{d} \mathbf{y} \le \int_{\mathbb{R}^n} \tilde{R}_\delta(\mathbf{x},\mathbf{y}) \mathrm{d} \mathbf{y} = C_1\]
	To prove the lower bound, we need to use the condition that $\partial \Omega$ is $C^2$ and $\tilde{R}$ is continuous and bounded. For $\mathbf{x}\in \partial \Omega$,
	\[\lim_{\delta\rightarrow 0} \int_\Omega \tilde{R}_\delta(\mathbf{x},\mathbf{y}) \mathrm{d} \mathbf{y} = \alpha_n \int_{\mathbf{x} + \mathbb{R}^n_+} \tilde{R}\left(\frac{\|\mathbf{x}-\mathbf{y}\|^2}{4}\right) \mathrm{d} \mathbf{y} =\frac{C_1}{2}
	\]
	where $\mathbb{R}^n_+=\{\mathbf{y}=(y_1,\cdots,y_n)\in \mathbb{R}^n: y_1\ge 0\}$. On the other hand, for $\mathbf{x}\in \Omega$, since $\Omega$ is open,
	\[\lim_{\delta \rightarrow 0} \int_\Omega \tilde{R}_\delta(\mathbf{x},\mathbf{y}) \mathrm{d} \mathbf{y} =
	\alpha_n \int_{\mathbb{R}^n} \tilde{R} \left( \frac{\|\mathbf{x}-\mathbf{y}\|^2}{4} \right)\mathrm{d} \mathbf{y} = C_1\]
	Thus for any $\mathbf{x} \in \Omega\cup \partial \Omega$, there exist $\delta_{\mathbf{x}}>0$ such that for any $\delta\le \delta_{\mathbf{x}}$, we have $\int_\Omega \tilde{R}_\delta(\mathbf{x},\mathbf{y}) \mathrm{d} \mathbf{y}>\frac{C_1}{3}$. Using the compactness of $\bar{\Omega}$, there exists $\delta_0>0$ such that for any $\mathbf{x}\in \bar{\Omega}$, $\delta\le \delta_{0}$, we have $\int_\Omega \tilde{R}_\delta(\mathbf{x},\mathbf{y}) \mathrm{d} \mathbf{y}>\frac{C_1}{3}$.
	
	The results for $\int_{\partial \Omega} \tilde{R}_\delta(\mathbf{x},\mathbf{y})\mathrm{d} S_\mathbf{y}$ can be derived similarly. The extra factor $\frac{1}{\delta}$ is because the integration domain changes from $\Omega \in \mathbb{R}^n$ to $\partial \Omega \in \mathbb{R}^{n-1}$.
\end{proof}

\section{\texorpdfstring{Proof in \Cref{sec:wellposed}}{Proof in Section 4}}

\subsection{\texorpdfstring{Proof of \Cref{lemma:r_rbar}}{Proof of Lemma 4.1}}
\label{subsec:b.1}
First we introduce a lemma.
\begin{lemma}[\cite{shi2017convergence}]
	\label{lemma:4_32}
	If $\delta$ is small enough, then for any $u\in L^2(\Omega)$, there exists a constant $C>0$ independent of $\delta$ and $u$, such that
	\[\int_{\Omega} \int_{\Omega} R\left(\frac{\|\mathbf{x}-\mathbf{y}\|^{2}}{4 \delta^2}\right)(u(\mathbf{x})-u(\mathbf{y}))^{2} \mathrm{d} \mathbf{x} \mathrm{d} \mathbf{y} \geq C \int_{\Omega} \int_{\Omega} R\left(\frac{\|\mathbf{x}-\mathbf{y}\|^{2}}{32 \delta^2}\right)(u(\mathbf{x})-u(\mathbf{y}))^{2} \mathrm{d} \mathbf{x} \mathrm{d} \mathbf{y} \]
\end{lemma}

Then we shall prove \Cref{lemma:r_rbar}.
\begin{proof}
	Using \Cref{lemma:4_32},
	\begin{align}
		&\int_{\Omega} \int_{\Omega} R_{\delta}(\mathbf{x}, \mathbf{y})(u(\mathbf{x})-u(\mathbf{y}))^{2} \mathrm{d} \mathbf{x} \mathrm{d} \mathbf{y} \nonumber\\
		= {}& C_\delta\int_{\Omega} \int_{\Omega} R\left(\frac{\|\mathbf{x}-\mathbf{y}\|^{2}}{4 \delta^2}\right)(u(\mathbf{x})-u(\mathbf{y}))^{2} \mathrm{d} \mathbf{x} \mathrm{d} \mathbf{y} \nonumber\\
		\geq {}& C_1 C_\delta \int_{\Omega} \int_{\Omega} R\left(\frac{|\mathbf{x}-\mathbf{y}|^{2}}{32 \delta^2}\right)(u(\mathbf{x})-u(\mathbf{y}))^{2} \mathrm{d} \mathbf{x} \mathrm{d} \mathbf{y} \nonumber\\
		\geq {}& C_1 C_\delta \int_{\Omega} \int_{|\mathbf{x}-\mathbf{y}|\leq 2\delta} R\left(\frac{|\mathbf{x}-\mathbf{y}|^{2}}{32 \delta^2}\right)(u(\mathbf{x})-u(\mathbf{y}))^{2} \mathrm{d} \mathbf{x} \mathrm{d} \mathbf{y} \nonumber\\
		\geq {}& C_1 C_\delta\gamma_0 \int_{\Omega} \int_{|\mathbf{x}-\mathbf{y}|\leq 2\delta} (u(\mathbf{x})-u(\mathbf{y}))^{2} \mathrm{d} \mathbf{x} \mathrm{d} \mathbf{y} \label{eq:lemma_r_rbar1}\\
		\geq {}& C C_\delta \int_{\Omega} \int_{|\mathbf{x}-\mathbf{y}|\leq 2 \delta} \bar{R}\left(\frac{|\mathbf{x}-\mathbf{y}|^{2}}{4 \delta^2}\right)(u(\mathbf{x})-u(\mathbf{y}))^{2} \mathrm{d} \mathbf{x} \mathrm{d} \mathbf{y} \label{eq:lemma_r_rbar2}\\
		= {}& C \int_{\Omega} \int_{\Omega} \bar{R}_{\delta}(\mathbf{x}, \mathbf{y})(u(\mathbf{x})-u(\mathbf{y}))^{2} \mathrm{d} \mathbf{x} \mathrm{d} \mathbf{y} \nonumber
	\end{align}
	Here in \Cref{eq:lemma_r_rbar1} we use the nondegeneracy property in \Cref{assumption:kernel} (d), and in \Cref{eq:lemma_r_rbar2} we use smoothness and compact support to get $\bar{R}$ is bounded.
\end{proof}

\subsection{\texorpdfstring{Proof of \Cref{prop:bounded}}{Proof of Proposition 4.3}}
\label{subsec:b.2}
\begin{proof}
	The technique we use is almost identical to that used in \Cref{prop:continuous}. Each term can be controlled.
	\begin{align*}
	    \langle F, p_\delta \rangle & = \int_{\Omega} p_\delta(\mathbf{x}) \int_{\Omega}  \bar{R}_{\delta}(\mathbf{x}, \mathbf{y}) f(\mathbf{y})\mathrm{d} \mathbf{y} \mathrm{d} \mathbf{x} - \int_{\Omega} p_\delta(\mathbf{x}) \int_{\partial \Omega} \frac{\bar{R}_{\delta}(\mathbf{x}, \mathbf{y})}{\bar{\bar{w}}_\delta(\mathbf{y})} \int_{\Omega} \bar{\bar{R}}_{\delta}(\mathbf{y}, \mathbf{s}) f(\mathbf{s}) \mathrm{d} \mathbf{s} \mathrm{d} S_{\mathbf{y}} \mathrm{d} \mathbf{x} \\
	    &\leq C\left\|p_{\delta}\right\|_{L^{2}(\Omega)}\left\|f\right\|_{L^{2}(\Omega)}
	\end{align*}
	Since $f\in L^2(\Omega)$, we get $F$ is bounded.
\end{proof}

\subsection{\texorpdfstring{Proof of \Cref{prop:gradient_estimation}}{Proof of Proposition 4.5}}
\label{subsec:4.8}
\begin{proof}
We expand the model, reformat it and get
\[u_\delta(\mathbf{x})=\frac{1}{w_\delta(\mathbf{x})}\int_{\Omega} R_{\delta}(\mathbf{x}, \mathbf{y})u_\delta(\mathbf{y}) \mathrm{d} \mathbf{y}+ \frac{2\delta^2}{w_\delta(\mathbf{x})} \int_{\partial \Omega} \bar{R}_{\delta}(\mathbf{x}, \mathbf{y}) v_\delta(\mathbf{y}) \mathrm{d} S_{\mathbf{y}}+\frac{\delta^2}{w_\delta(\mathbf{x})}F_{\mathrm{in}}(\mathbf{x})\]
	
We substitute $v_\delta(\mathbf{y})$ with \Cref{eq:v_eliminate} and get
\begin{align*}
    u_\delta(\mathbf{x})=&\frac{1}{w_\delta(\mathbf{x})}\int_{\Omega} R_{\delta}(\mathbf{x}, \mathbf{y})u_\delta(\mathbf{y}) \mathrm{d} \mathbf{y} - \frac{1}{w_\delta(\mathbf{x})}\int_{\partial \Omega} \frac{\bar{R}_{\delta}(\mathbf{x}, \mathbf{y}) \bar{w}_\delta(\mathbf{y}) }{\bar{\bar{w}}_\delta(\mathbf{y})} \hat{u}_\delta(\mathbf{y}) \mathrm{d} S_{\mathbf{y}} \\
    &-\frac{\delta^2}{w_\delta(\mathbf{x})}\int_{\partial \Omega} \frac{\bar{R}_{\delta}(\mathbf{x}, \mathbf{y})}{\bar{\bar{w}}_\delta(\mathbf{y})} F_{\mathrm{bd}}(\mathbf{y}) \mathrm{d} S_{\mathbf{y}} + \frac{\delta^2}{w_\delta(\mathbf{x})}F_{\mathrm{in}}(\mathbf{x})
\end{align*}
	
where $\hat{u}_\delta(\mathbf{x})$ is the smoothed version of $u_\delta(\mathbf{x})$ defined in \Cref{eq:u_smoothed}. Thus
\begin{align*}
	\|\nabla u_\delta(\mathbf{x})\|_{L^2(\Omega)}^2 & \leq C\|\nabla\big( \frac{1}{w_\delta(\mathbf{x})}\int_{\Omega} R_{\delta}(\mathbf{x}, \mathbf{y})u_\delta(\mathbf{y}) \mathrm{d} \mathbf{y}  \big) \|_{L^2(\Omega)}^2\\
	& \quad +C\|\nabla\big( \frac{1}{w_\delta(\mathbf{x})}\int_{\partial \Omega} \frac{\bar{R}_{\delta}(\mathbf{x}, \mathbf{y}) \bar{w}_\delta(\mathbf{y}) }{\bar{\bar{w}}_\delta(\mathbf{y})} \hat{u}_\delta(\mathbf{y}) \mathrm{d} S_{\mathbf{y}} \big) \|_{L^2(\Omega)}^2\\
	& \quad +C\|\nabla\big( \frac{\delta^2}{w_\delta(\mathbf{x})}\int_{\partial \Omega} \frac{\bar{R}_{\delta}(\mathbf{x}, \mathbf{y})}{\bar{\bar{w}}_\delta(\mathbf{y})} F_{\mathrm{bd}}(\mathbf{y}) \mathrm{d} S_{\mathbf{y}} \big) \|_{L^2(\Omega)}^2\\
	& \quad +C\|\nabla\big( \frac{\delta^2}{w_\delta(\mathbf{x})}F_{\mathrm{in}}(\mathbf{x}) \big) \|_{L^2(\Omega)}^2 \numberthis \label{eq:h1_gradient_core}
\end{align*}
	
The first term in \Cref{eq:h1_gradient_core} can be directly bounded using \Cref{lemma:gradient}
\[\|\nabla\big(\frac{1}{w_\delta(\mathbf{x})}\int_{\Omega} R_{\delta}(\mathbf{x}, \mathbf{y})u_\delta(\mathbf{y}) \mathrm{d} \mathbf{y} \big)\|_{L^2(\Omega)}^2 \leq \frac{C}{\delta^2}\int_{\Omega} \int_{\Omega} R_{\delta}(\mathbf{x}, \mathbf{y})(u_\delta(\mathbf{x})-u_\delta(\mathbf{y}))^2 \mathrm{d} \mathbf{x} \mathrm{d} \mathbf{y}\]
	
The second term in \Cref{eq:h1_gradient_core} is more complicated
\begin{align*}
	& \left\|\nabla\big( \frac{1}{w_\delta(\mathbf{x})}\int_{\partial \Omega} \frac{\bar{R}_{\delta}(\mathbf{x}, \mathbf{y}) \bar{w}_\delta(\mathbf{y}) }{\bar{\bar{w}}_\delta(\mathbf{y})} \hat{u}_\delta(\mathbf{y}) \mathrm{d} S_{\mathbf{y}} \big) \right\|_{L^2(\Omega)}^2 \numberthis \label{eq:h1_gradeint_second_term}\\
	= & \left\| \frac{\nabla w_\delta(\mathbf{x})}{w^2_\delta(\mathbf{x})}\int_{\partial \Omega} \frac{\bar{R}_{\delta}(\mathbf{x}, \mathbf{y}) \bar{w}_\delta(\mathbf{y}) }{\bar{\bar{w}}_\delta(\mathbf{y})} \hat{u}_\delta(\mathbf{y}) \mathrm{d} S_{\mathbf{y}} + \frac{1}{w_\delta(\mathbf{x})}\int_{\partial \Omega} \frac{\nabla\bar{R}_{\delta}(\mathbf{x}, \mathbf{y}) \bar{w}_\delta(\mathbf{y}) }{\bar{\bar{w}}_\delta(\mathbf{y})} \hat{u}_\delta(\mathbf{y}) \mathrm{d} S_{\mathbf{y}} \right\|_{L^2(\Omega)}^2\\
	\leq {}& 2\left\|\frac{\nabla w_\delta(\mathbf{x})}{w^2_\delta(\mathbf{x})}\int_{\partial \Omega} \frac{\bar{R}_{\delta}(\mathbf{x}, \mathbf{y}) \bar{w}_\delta(\mathbf{y}) }{\bar{\bar{w}}_\delta(\mathbf{y})} \hat{u}_\delta(\mathbf{y}) \mathrm{d} S_{\mathbf{y}} \right\|_{L^2(\Omega)}^2 \\
	& \hspace{150pt} + 2\left\| \frac{1}{w_\delta(\mathbf{x})}\int_{\partial \Omega} \frac{\nabla\bar{R}_{\delta}(\mathbf{x}, \mathbf{y}) \bar{w}_\delta(\mathbf{y}) }{\bar{\bar{w}}_\delta(\mathbf{y})} \hat{u}_\delta(\mathbf{y}) \mathrm{d} S_{\mathbf{y}}\right\|_{L^2(\Omega)}^2
\end{align*}

Here we need an estimate for $\left\|\nabla w_\delta(\mathbf{x})\right\|$.
\begin{align*}
    \left\|\nabla w_\delta(\mathbf{x})\right\| & = \left\|\int_{\Omega} \nabla_{\mathbf{x}}R_{\delta}(\mathbf{x}, \mathbf{y}) \mathrm{d} \mathbf{y} \right\| =\left\|\int_{\Omega} C_\delta R' \left(\frac{\|\mathbf{x} -\mathbf{y}\|^2}{4\delta^2}\right)\cdot\frac{\mathbf{x} -\mathbf{y}}{2\delta^2} \mathrm{d} \mathbf{y} \right\| \\
    &\leq \frac{1}{\delta} \int_{\Omega} C_\delta R'\left(\frac{\|\mathbf{x} -\mathbf{y}\|^2}{4\delta^2}\right) \mathrm{d} \mathbf{y} \leq \frac{C}{\delta}
\end{align*}
in which we use \Cref{assumption:kernel} (c) compact support property, so that $\left\|\mathbf{x} -\mathbf{y}\right\|\leq 2\delta$, and thus $\left\|\frac{\mathbf{x} -\mathbf{y}}{2\delta^2}\right\|\leq \frac{1}{\delta}$.

Then we deal with the two terms in \Cref{eq:h1_gradeint_second_term} separately,
\begin{align*}
	& \left\|\frac{\nabla w_\delta(\mathbf{x})}{w^2_\delta(\mathbf{x})}\int_{\partial \Omega} \frac{\bar{R}_{\delta}(\mathbf{x}, \mathbf{y}) \bar{w}_\delta(\mathbf{y}) }{\bar{\bar{w}}_\delta(\mathbf{y})} \hat{u}_\delta(\mathbf{y}) \mathrm{d} S_{\mathbf{y}} \right\|_{L^2(\Omega)}^2 \\
	\leq {}& \frac{C_1}{\delta^2} \int_{\Omega} \left(\int_{\partial \Omega} \frac{\bar{R}_{\delta}(\mathbf{x}, \mathbf{y}) \bar{w}_\delta(\mathbf{y}) }{\bar{\bar{w}}_\delta(\mathbf{y})} \hat{u}_\delta(\mathbf{y}) \mathrm{d} S_{\mathbf{y}} \right)^2 \mathrm{d} \mathbf{x}\\
	\leq {}& C_2 \int_{\Omega} \left(\int_{\partial \Omega} \bar{R}_{\delta}(\mathbf{x}, \mathbf{y}) \hat{u}_\delta(\mathbf{y}) \mathrm{d} S_{\mathbf{y}} \right)^2 \mathrm{d} \mathbf{x} \\
	\leq {}& C_2 \int_{\Omega} \left(\int_{\partial \Omega} \bar{R}_{\delta}(\mathbf{x}, \mathbf{y}) \hat{u}^2_\delta(\mathbf{y}) \mathrm{d} S_{\mathbf{y}} \right) \left(\int_{\partial \Omega} \bar{R}_{\delta}(\mathbf{x}, \mathbf{y}) \mathrm{d} S_{\mathbf{y}} \right) \mathrm{d} \mathbf{x} \\
	\leq {}& \frac{C_3}{\delta} \int_{\Omega} \left(\int_{\partial \Omega} \bar{R}_{\delta}(\mathbf{x}, \mathbf{y}) \hat{u}^2_\delta(\mathbf{y}) \mathrm{d} S_{\mathbf{y}} \right) \mathrm{d} \mathbf{x} \\
	= {}& \frac{C_3}{\delta} \int_{\partial \Omega} \hat{u}^2_\delta(\mathbf{y}) \int_{\Omega} \bar{R}_{\delta}(\mathbf{x}, \mathbf{y}) \mathrm{d} \mathbf{x} \mathrm{d} S_{\mathbf{y}} \\
	\leq {}& \frac{C_4}{\delta} \int_{\partial \Omega} \hat{u}^2_\delta(\mathbf{y}) \mathrm{d} S_{\mathbf{y}} \leq \frac{C_5}{\delta^2}\int_{\partial \Omega} \frac{\bar{w}_\delta^2(\mathbf{x})}{\bar{\bar{w}}_\delta(\mathbf{x})} \hat{u}_\delta^2(\mathbf{x})\mathrm{d} S_\mathbf{x} \\
	={} & \frac{C_5}{\delta^2}\int_{\partial \Omega} \frac{1}{\bar{\bar{w}}_\delta(\mathbf{x})} \left(\int_{\Omega} \bar{R}_\delta(\mathbf{x}, \mathbf{y}) u_\delta(\mathbf{y})\mathrm{d} \mathbf{y}\right)^2 \mathrm{d} S_\mathbf{x}
\end{align*}
and similarly
\begin{align*}
	& \left\| \frac{1}{w_\delta(\mathbf{x})}\int_{\partial \Omega} \frac{\nabla\bar{R}_{\delta}(\mathbf{x}, \mathbf{y}) \bar{w}_\delta(\mathbf{y}) }{\bar{\bar{w}}_\delta(\mathbf{y})} \hat{u}_\delta(\mathbf{y}) \mathrm{d} S_{\mathbf{y}} \right\|_{L^2(\Omega)}^2 \\
	= & \left\| \frac{1}{w_\delta(\mathbf{x})}\int_{\partial \Omega} \frac{R_{\delta}(\mathbf{x}, \mathbf{y}) \bar{w}_\delta(\mathbf{y}) }{\bar{\bar{w}}_\delta(\mathbf{y})} \hat{u}_\delta(\mathbf{y}) \frac{x-y}{2\delta^2} \mathrm{d} S_{\mathbf{y}}\right\|_{L^2(\Omega)}^2 \\
	\leq & \frac{C_1}{\delta^2} \int_{\Omega} \left(\int_{\partial \Omega} \frac{R_{\delta}(\mathbf{x}, \mathbf{y})\bar{w}_\delta(\mathbf{y}) }{\bar{\bar{w}}_\delta(\mathbf{y})} \hat{u}_\delta(\mathbf{y}) \mathrm{d} S_{\mathbf{y}} \right)^2 \mathrm{d} \mathbf{x} \\
	\leq & \frac{C_2}{\delta^2}\int_{\partial \Omega} \frac{1}{\bar{\bar{w}}_\delta(\mathbf{x})} \left(\int_{\Omega} \bar{R}_\delta(\mathbf{x}, \mathbf{y}) u_\delta(\mathbf{y})\mathrm{d} \mathbf{y}\right)^2 \mathrm{d} S_\mathbf{x}
\end{align*}

Thus \Cref{eq:h1_gradeint_second_term} can be controlled by
\begin{align*}
    &\left\|\nabla \big( \frac{1}{w_\delta(\mathbf{x})}\int_{\partial \Omega} \frac{\bar{R}_{\delta} (\mathbf{x}, \mathbf{y}) \bar{w}_\delta(\mathbf{y}) }{\bar{\bar{w}}_\delta(\mathbf{y})} \hat{u}_\delta(\mathbf{y}) \mathrm{d} S_{\mathbf{y}} \big) \right\|_{L^2(\Omega)}^2 \\
    \leq {}&\frac{C}{\delta^2}\int_{\partial \Omega} \frac{1}{\bar{\bar{w}}_\delta(\mathbf{x})} \left(\int_{\Omega} \bar{R}_\delta(\mathbf{x}, \mathbf{y}) u_\delta(\mathbf{y})\mathrm{d} \mathbf{y}\right)^2 \mathrm{d} S_\mathbf{x}
\end{align*}

The third and fourth term in \Cref{eq:h1_gradient_core} can be calculated similarly
\[\|\nabla\big( \frac{\delta^2}{w_\delta(\mathbf{x})}\int_{\partial \Omega} \frac{\bar{R}_{\delta}(\mathbf{x}, \mathbf{y})}{\bar{\bar{w}}_\delta(\mathbf{y})} F_{\mathrm{bd}}(\mathbf{y}) \mathrm{d} S_{\mathbf{y}} \big) \|_{L^2(\Omega)}^2 \leq C\delta^3 \| F_{\mathrm{bd}}(\mathbf{x}) \|_{L^2(\partial \Omega)}^2\]
\[\|\nabla\big( \frac{\delta^2}{w_\delta(\mathbf{x})} F_{\mathrm{in}}(\mathbf{x}) \big) \|_{L^2(\Omega)}^2 \leq C\delta^2 \| F_{\mathrm{in}}(\mathbf{x}) \|_{L^2(\Omega)}^2 + C\delta^4 \| \nabla F_{\mathrm{in}}(\mathbf{x}) \|_{L^2(\Omega)}^2\]
	
Combine all estimates for the four terms in \Cref{eq:h1_gradient_core}, we have controlled $\|\nabla u_\delta\|_{L^2(\Omega)}$ as in \Cref{eq:h1_gradient_estimation}.
\end{proof}

\section{\texorpdfstring{Proof in \Cref{sec:convergence}}{Proof in Section 5}}

\subsection{\texorpdfstring{Proof of \Cref{thm:boundary_consistency}}{Proof of Theorem 5.2}}
\label{subsec:5.4}
\begin{proof}
We investigate the order of $\|r_{\mathrm{bd}}(\mathbf{x})\|_{L^{2}(\partial \Omega)}$.

\begin{align*}
    &\quad \left\| r_{\mathrm{bd}}(\mathbf{x}) \right\|_{L^2(\partial \Omega)}^2 \\
    &= \left\|\bar{r}_{\mathrm{in}}(\mathbf{x})- 2 \int_{\partial \Omega} \bar{\bar{R}}_\delta(\mathbf{x}, \mathbf{y})\left(\frac{\partial u}{\partial \mathbf{n}}(\mathbf{x})-\frac{\partial u}{\partial \mathbf{n}}(\mathbf{y}) \right)  \mathrm{d} S_{\mathbf{y}} \right\|_{L^2(\partial \Omega)}^2\\
    & = \left\|\bar{r}_{\mathrm{it}}(\mathbf{x}) + \bar{r}_{\mathrm{bl}}(\mathbf{x})- 2 \int_{\partial \Omega} \bar{\bar{R}}_\delta(\mathbf{x}, \mathbf{y})\left(\frac{\partial u}{\partial \mathbf{n}}(\mathbf{x})-\frac{\partial u}{\partial \mathbf{n}}(\mathbf{y}) \right)  \mathrm{d} S_{\mathbf{y}} \right\|_{L^2(\partial \Omega)}^2\\
    & \vphantom{\left\|\int_{\partial \Omega} \bar{\bar{R}}_\delta(\mathbf{x}, \mathbf{y})\left(\frac{\partial u}{\partial \mathbf{n}}(\mathbf{x})-\frac{\partial u}{\partial \mathbf{n}}(\mathbf{y}) \right)  \mathrm{d} {\mathbf{y}} \right\|_{L^2(\partial \Omega)}^2} \leq C \left\|\bar{r}_{\mathrm{it}}(\mathbf{x})\right\|_{L^2(\partial \Omega)}^2 + C\left\| \bar{r}_{\mathrm{bl}}(\mathbf{x}) \right\|_{L^2(\partial \Omega)}^2 + C\left\|\int_{\partial \Omega} \bar{\bar{R}}_\delta(\mathbf{x}, \mathbf{y})\left(\frac{\partial u}{\partial \mathbf{n}}(\mathbf{x})-\frac{\partial u}{\partial \mathbf{n}}(\mathbf{y}) \right)  \mathrm{d} S_{\mathbf{y}} \right\|_{L^2(\partial \Omega)}^2  \numberthis \label{eq:convergence_consistency_core}
\end{align*}

Using trace theorem and \Cref{thm:consistency}, the first term in \Cref{eq:convergence_consistency_core} can be controlled,
\begin{align*}
	\left\|\bar{r}_{\mathrm{it}}(\mathbf{x})\right\|_{L^{2}(\partial \Omega)} & \leq C_1 \left\|\bar{r}_{\mathrm{it}}(\mathbf{x})\right\|_{H^1(\Omega)} = C_1 (\left\|\bar{r}_{\mathrm{it}}(\mathbf{x})\right\|_{L^2(\Omega)} +\left\|\nabla \bar{r}_{\mathrm{it}}(\mathbf{x})\right\|_{L^2(\Omega)})\\
	& \leq C_1 (C_2 \delta \|u\|_{H^{3}(\Omega)}+C_2 \|u\|_{H^{3}(\Omega)}) \leq C\|u\|_{H^{3}(\Omega)}
\end{align*}

Using \Cref{assumption:kernel} (c) compact support property, we have $\left\|\mathbf{x} -\mathbf{y}\right\|\leq 2\delta$. Thus the second term in \Cref{eq:convergence_consistency_core} can be upper bounded,
\begin{align*}
	\left\|\bar{r}_{\mathrm{bl}}(\mathbf{x})\right\|_{L^{2}(\partial \Omega)}^2 & =\left\| \int_{\partial \Omega} \bar{\bar{R}}_\delta(\mathbf{x}, \mathbf{y})(\mathbf{x}-\mathbf{y}) \cdot \mathbf{b}(\mathbf{y}) \mathrm{d} S_{\mathbf{y}} \right\|_{L^{2}(\partial \Omega)}^2 \\
	& \leq C_1 \delta^2 \int_{\partial \Omega}\left(\int_{\partial \Omega} \bar{\bar{R}}_\delta(\mathbf{x}, \mathbf{y}) \mathbf{b}(\mathbf{y}) \mathrm{d} S_{\mathbf{y}} \right)^2 \mathrm{d} S_{\mathbf{x}}\\
	& \leq C_1 \delta^2 \int_{\partial \Omega}\left(\int_{\partial \Omega} \bar{\bar{R}}_\delta(\mathbf{x}, \mathbf{y}) \mathrm{d} S_{\mathbf{y}} \right)\left(\int_{\partial \Omega} \bar{\bar{R}}_\delta(\mathbf{x}, \mathbf{y}) \mathbf{b}^2(\mathbf{y}) \mathrm{d} S_{\mathbf{y}} \right) \mathrm{d} S_{\mathbf{x}} \\
	& \vphantom{\left(\int_{\partial \Omega} \bar{\bar{R}}_\delta(\mathbf{x}, \mathbf{y}) \mathrm{d} {\mathbf{y}} \right)} \leq C_2 \left\|\mathbf{b}\right\|_{L^{2}(\partial \Omega)}^2 \leq C_3\left\|\mathbf{b}\right\|_{H^{1}(\Omega)}^2 \leq C\left\|u\right\|_{H^{3}(\Omega)}^2
\end{align*}
	
As for the third term in \Cref{eq:convergence_consistency_core}, we need a local parametrization of the boundary $\partial \Omega$. Here we use the following proposition, which basically says there exists a local parametrization of small distortion and the parameter domain is convex and big enough.

\begin{proposition}[\cite{shi2017convergence}]
Assume both $\mathcal{M}$ and $\partial \mathcal{M}$ are compact and $C^2$ smooth.
$\sigma$ is the minimum of the reaches of $\mathcal{M}$ and $\partial \mathcal{M}$. For any point $\mathbf{x}\in \mathcal{M}$, there is a neighborhood $U \subset \mathcal{M}$ of $\mathbf{x}$, so that there is a parametrization $\Phi\in C^2: \Omega \subset \mathbb{R}^k \rightarrow U$ satisfying the following conditions. For any $\rho \leq 0.1$, 
\begin{enumerate}
\item[(i)] $\Omega$ is convex and contains at least half of the ball
$B_{\Phi^{-1}(\mathbf{x})}(\frac{\rho}{5} \sigma)$, 
i.e., $vol(\Omega \cap B_{\Phi^{-1}(\mathbf{x})}(\frac{\rho}{5} \sigma)) > \frac{1}{2}(\frac{\rho}{5}\sigma)^k w_k$ where $w_k$ is the volume of unit ball in $\mathbf{R}^k$;
\item[(ii)] $B_{\mathbf{x}}(\frac{\rho}{10} \sigma) \cap \mathcal{M} \subset U$. 
\item[(iii)] The determinant the Jacobian of $\Phi$ is bounded:
$(1-2\rho)^k \leq |D\Phi|  \leq (1+2\rho)^k$ over $\Omega$.  
\item[(iv)] For any points $\mathbf{y}, \mathbf{z} \in U$, 
$1-2\rho \leq \frac{|\mathbf{y}-\mathbf{z}|}{\left|\Phi^{-1}(\mathbf{y}) - \Phi^{-1}(\mathbf{z})\right|}  \leq 1+3\rho$.
\end{enumerate}
\label{prop:local_param}
\end{proposition}

We choose $\mathcal{M}=\partial \Omega$ in the proposition, then $\partial \mathcal{M}=\varnothing$. According to the \Cref{assumption:domain}, $(\partial \Omega, \varnothing)$ clearly satisfies the smoothness condition. Let $\rho=0.1$, $\sigma$ be the minimum of the reach of $\partial \Omega$ and $t= \rho \sigma/20$. For any $\mathbf{x}\in \partial \Omega$, denote
\begin{eqnarray}
\label{eq:net}
  B_{\mathbf{x}}^t=\left\{\mathbf{y}\in \partial \Omega: \|\mathbf{x}-\mathbf{y}\|\le t\right\},\quad B_{\mathbf{x}}^{2\delta}=\left\{\mathbf{y}\in \partial \Omega: \|\mathbf{x}-\mathbf{y}\|\le 2\delta\right\}
\end{eqnarray}
and we assume $\delta$ is small enough such that $2\delta\le t$.

Since the boundary $\partial \Omega$ is compact, there exists a $t$-net, $\mathcal{N}_t=\{\mathbf{q}_i\in \partial \Omega,\;i=1,\cdots,N\}$, such that
\[\partial \Omega \subset \bigcup_{i=1}^N B_{\mathbf{q}_i}^t\]
and there exists a partition of $\partial \Omega$, $\{\mathcal{O}_i, \;i=1,\cdots,N\}$, such that $\mathcal{O}_i\cap \mathcal{O}_j=\varnothing,\; i\neq j$ and
\[\partial \Omega=\bigcup_{i=1}^N \mathcal{O}_i,\quad \mathcal{O}_i\subset B_{\mathbf{q}_i}^t,\quad i=1,\cdots,N.\]

Using \Cref{prop:local_param}, there exist a parametrization $\Phi_i: \Omega_i\subset\mathbb{R}^k \rightarrow U_i\subset \partial \Omega,\; i=1,\cdots, N$, such that
\begin{itemize}
\item[1.] (Convexity) $B_{\mathbf{q}_i}^{2t} \subset U_i$ and $\Omega_i$ is convex.
\item[2.] (Smoothness) $\Phi_i\in C^2(\Omega_i)$;
\item[3.] (Locally small deformation) 
For any points $\theta_1, \theta_2\in \Omega_i$,
\begin{equation}
    \frac{1}{2}\left|\theta_1-\theta_2\right| \leq \left\|\Phi_i(\theta_1)-\Phi_i(\theta_2)\right\| \leq 2\left|\theta_1-\theta_2\right|
    \label{eq:convergence_small_deformation}
\end{equation}
\end{itemize}

Using the partition $\{\mathcal{O}_i, \;i=1,\cdots,N\}$, for any $\mathbf{y}\in \partial \Omega$, there exists unique $J(\mathbf{y})\in \{1,\cdots,N\}$
such that $\mathbf{y}\in \mathcal{O}_{J(\mathbf{y})}\subset B_{\mathbf{q}_{J(\mathbf{y})}}^t$. Moreover, using the condition $2\delta \le t$, we have $B_{\mathbf{y}}^{2\delta} \subset B_{\mathbf{q}_{J(\mathbf{y})}}^{2t}\subset U_{J(\mathbf{y})}$.
Then $\Phi_{J(\mathbf{y})}^{-1}(\mathbf{x})$ and $\Phi_{J(\mathbf{y})}^{-1}(\mathbf{y})$ are both well defined for any $\mathbf{x}\in B_{\mathbf{y}}^{2\delta}$. Thus for any $\mathbf{y}\in \partial \Omega$, $\mathbf{x}\in B_{\mathbf{y}}^{2\delta}$, we can define
\[\alpha=\Phi_{J(\mathbf{y})}^{-1}(\mathbf{y}), \qquad \xi=\Phi_{J(\mathbf{y})}^{-1}(\mathbf{x})-\Phi_{J(\mathbf{y})}^{-1}(\mathbf{y})\]

Then we are ready to deal with the third term in \Cref{eq:convergence_consistency_core}.
\begingroup
\allowdisplaybreaks
\begin{align*}
	&\left\|\int_{\partial \Omega} \bar{\bar{R}}_\delta(\mathbf{x}, \mathbf{y})\left(\frac{\partial u}{\partial \mathbf{n}}(\mathbf{x})-\frac{\partial u}{\partial \mathbf{n}}(\mathbf{y}) \right)  \mathrm{d} S_{\mathbf{y}} \right\|_{L^2(\partial \Omega)}^2 \numberthis \label{eq:convergence_boundary_truncation1}\\
	=& \int_{\partial \Omega} \left(\int_{\partial \Omega} \bar{\bar{R}}_\delta(\mathbf{x}, \mathbf{y})\left(\frac{\partial u}{\partial \mathbf{n}}(\mathbf{x})-\frac{\partial u}{\partial \mathbf{n}}(\mathbf{y}) \right)  \mathrm{d} S_{\mathbf{y}} \right)^2 \mathrm{d} S_\mathbf{x} \\
	\leq & \frac{C}{\delta} \int_{\partial \Omega} \int_{\partial \Omega} \bar{\bar{R}}_\delta(\mathbf{x}, \mathbf{y})\left(\frac{\partial u}{\partial \mathbf{n}}(\mathbf{x})-\frac{\partial u}{\partial \mathbf{n}}(\mathbf{y}) \right)^2  \mathrm{d} S_{\mathbf{y}} \mathrm{d} S_\mathbf{x} \\
	= & \frac{C}{\delta} \sum_{i=1}^N \int_{\partial \Omega} \int_{\mathcal{O}_i} \bar{\bar{R}}_\delta(\mathbf{x}, \mathbf{y})\left(\frac{\partial u}{\partial \mathbf{n}}(\mathbf{x})-\frac{\partial u}{\partial \mathbf{n}}(\mathbf{y}) \right)^2  \mathrm{d} S_{\mathbf{y}} \mathrm{d} S_\mathbf{x} \\
	= & \frac{C}{\delta} \sum_{i=1}^N \int_{\mathcal{O}_i} \int_{B_{\mathbf{y}}^{2\delta}} \bar{\bar{R}}_\delta(\mathbf{x}, \mathbf{y})\left(\frac{\partial u}{\partial \mathbf{n}}(\mathbf{x})-\frac{\partial u}{\partial \mathbf{n}}(\mathbf{y}) \right)^2 \mathrm{d} S_{\mathbf{x}} \mathrm{d} S_\mathbf{y}\\
	= & \frac{C}{\delta} \sum_{i=1}^N \int_{\mathcal{O}_i} \int_{B_{\mathbf{y}}^{2\delta}} \bar{\bar{R}}_\delta(\mathbf{x}, \mathbf{y})\left(\frac{\partial u}{\partial \mathbf{n}}\left(\Phi_i(\alpha+\xi)\right)-\frac{\partial u}{\partial \mathbf{n}}\left(\Phi_i(\alpha)\right) \right)^2 \mathrm{d} S_{\mathbf{x}} \mathrm{d} S_\mathbf{y}
\end{align*}
\endgroup
% = & \frac{C}{\delta} \int_{\partial \Omega} \int_{\partial \Omega} \bar{\bar{R}}_\delta(\mathbf{x}, \mathbf{y})\left(\int_{0}^{1} f'(s) \mathrm{d} s \right)^2  \mathrm{d} {\mathbf{y}} \mathrm{d} {\mathbf{x}} \\
% = & \frac{C}{\delta} \int_{\partial \Omega} \int_{\partial \Omega} \bar{\bar{R}}_\delta(\mathbf{x}, \mathbf{y})\left(\int_{0}^{1} \nabla v \left(\mathbf{y}+s(\mathbf{x}-\mathbf{y})\right)\cdot (\mathbf{x}-\mathbf{y})\mathrm{d} s \right)^2  \mathrm{d} {\mathbf{y}} \mathrm{d} {\mathbf{x}} \\
% \leq& C \delta \max_{0\leq s \leq 1} \int_{\partial \Omega} \int_{\partial \Omega} \bar{\bar{R}}_\delta(\mathbf{x}, \mathbf{y}) \left|D^2 u\left(\mathbf{y}+s(\mathbf{x}-\mathbf{y})\right)\right|^2  \mathrm{d} {\mathbf{y}} \mathrm{d} {\mathbf{x}} 
	
% Denote $\mathbf{z}=\mathbf{y}+s(\mathbf{x}-\mathbf{y})$, and do integration by substitution, we have
% \begin{align*}
%     & \int_{\partial \Omega} \int_{\partial \Omega} \bar{\bar{R}}_\delta(\mathbf{x}, \mathbf{y}) \left|D^2 u\left(\mathbf{y}+s(\mathbf{x}-\mathbf{y})\right)\right|^2  \mathrm{d} {\mathbf{y}} \mathrm{d} {\mathbf{x}} \\
%     = & \int_{\partial \Omega} \int_{\partial \Omega} C_\delta \frac{1}{s^{n-1}} \bar{\bar{R}} \left(\frac{\|\mathbf{y} -\mathbf{z}\|^2}{4s^2\delta^2}\right) \left|D^2 u(\mathbf{z})\right|^2  \mathrm{d} {\mathbf{y}} \mathrm{d} {\mathbf{z}}\\
%     \leq & \frac{C_1}{\delta} \int_{\partial \Omega} \left|D^2 u(\mathbf{z})\right|^2 \mathrm{d} {\mathbf{z}} = \frac{C_1}{\delta} \|D^2 u\|_{L^2(\partial \Omega)} \leq \frac{C}{\delta} \|u\|_{H^3(\Omega)}
% \end{align*}

Introduce an auxiliary function (we temporarily simplify $\Phi_i$ to $\Phi$ to avoid subscript redundancy)
\[h(s)=\frac{\partial u}{\partial \mathbf{n}}\left(\Phi(\alpha+s\xi)\right), \quad s\in [0,1]\]

Then
\begin{align*}
    h'(s) &= \frac{\mathrm{d}}{\mathrm{d}s} \left( \nabla u \left(\Phi(\alpha+s\xi)\right)\cdot \mathbf{n}\left(\Phi(\alpha+s\xi)\right) \right) \\
    & = \frac{\mathrm{d}}{\mathrm{d}s} \left( \nabla u \left(\Phi(\alpha+s\xi)\right)\right) \cdot \mathbf{n}\left(\Phi(\alpha+s\xi)\right) +  \nabla u \left(\Phi(\alpha+s\xi)\right)\cdot \frac{\mathrm{d}}{\mathrm{d}s} \left( \mathbf{n}\left(\Phi(\alpha+s\xi)\right) \right)
\end{align*}

In the first term,
\[\frac{\mathrm{d}}{\mathrm{d}s} \left( \nabla u \left(\Phi(\alpha+s\xi)\right)\right)= \sum_{i, j=1}^{n}\nabla^{j} \nabla^{i} u \cdot \nabla\Phi^{j} \cdot \xi \]
where $\Phi^{j}$ denotes the $j$-th component of $\Phi$. In the second term, we use the fact that the normal vector $\mathbf{n}$ is orthogonal to the tangent vector, i.e.
\[\mathbf{n}\left(\Phi(\alpha+s\xi)\right) \cdot \nabla^{l} \Phi(\alpha+s\xi) =0 \quad \text{for} \quad l=1,2,\cdots,k \]

Take derivative to $s$ on both sides, we get
\[\frac{\mathrm{d}}{\mathrm{d}s} \left(\mathbf{n}\left(\Phi(\alpha+s\xi)\right)\right) \cdot \nabla^{l} \Phi(\alpha+s\xi) + \mathbf{n}\left(\Phi(\alpha+s\xi)\right) \cdot \frac{\mathrm{d}}{\mathrm{d}s} \nabla^{l} \Phi(\alpha+s\xi) =0 \]

Since we can choose parametrization such that $\left\{\nabla^{l} \Phi(\alpha+s\xi)\right\}_{l=1}^k$ form an orthogonal basis of the tangent space, we have
\begin{align*}
    \left\|\frac{\mathrm{d}}{\mathrm{d}s} \left(\mathbf{n}\left(\Phi(\alpha+s\xi)\right)\right)\right\|^2 & = \sum_{l=1}^k \left( \frac{\frac{\mathrm{d}}{\mathrm{d}s} \left(\mathbf{n}\left(\Phi(\alpha+s\xi)\right)\right) \cdot \nabla^{l} \Phi(\alpha+s\xi)}{\left\|\nabla^{l} \Phi(\alpha+s\xi)\right\|}\right)^2\\
    & = \sum_{l=1}^k \left(\frac{- \mathbf{n}\left(\Phi(\alpha+s\xi)\right) \cdot \frac{\mathrm{d}}{\mathrm{d}s} \nabla^{l} \Phi(\alpha+s\xi)}{\left\|\nabla^{l} \Phi(\alpha+s\xi)\right\|}\right)^2\\
    & = \sum_{l=1}^k \left(\frac{- \mathbf{n}\left(\Phi(\alpha+s\xi)\right) \cdot \sum_{j=1}^n \nabla \nabla^{l} \Phi^j(\alpha+s\xi)\cdot \xi}{\left\|\nabla^{l} \Phi(\alpha+s\xi)\right\|}\right)^2
\end{align*}

Since $\|\mathbf{n}\|=1$, $\xi\leq 2\|\mathbf{x}-\mathbf{y}\| \leq 4 \delta$, and $\Phi\in C^2$, we have
\begin{align*}
    |h'(s)| & \leq C\delta\left\|\sum_{i, j=1}^{n}\nabla^{j} \nabla^{i} u \cdot \nabla\Phi^{j} \right\|+ C\delta \|\nabla u \| \cdot \frac{\left\|\sum_{j=1}^n \nabla \nabla^{l} \Phi^j\right\|}{\left\|\nabla^{l} \Phi\right\|}\\
    & \leq C\delta \left(\left\|\sum_{i, j=1}^{n}\nabla^{j} \nabla^{i} u \right\| + \|\nabla u \|\right):=C\delta~ D^{1,2}(u)
\end{align*}

Finally, we can get
\begin{align*}
    &\int_{\mathcal{O}_i} \int_{B_{\mathbf{y}}^{2\delta}} \bar{\bar{R}}_\delta(\mathbf{x}, \mathbf{y})\left(\frac{\partial u}{\partial \mathbf{n}}\left(\Phi_i(\alpha+\xi)\right)-\frac{\partial u}{\partial \mathbf{n}}\left(\Phi_i(\alpha)\right) \right)^2 \mathrm{d} S_{\mathbf{x}} \mathrm{d} S_\mathbf{y} \numberthis \label{eq:convergence_boundary_truncation2}\\
    \leq & C\delta^2 \int _0^1 \int_{\mathcal{O}_i} \int_{B_{\mathbf{y}}^{2\delta}} \bar{\bar{R}}_\delta(\mathbf{x}, \mathbf{y})\left|D^{1,2} u(\Phi_i(\alpha+s\xi))\right|^{2} \mathrm{d} S_{\mathbf{x}}  \mathrm{d} S_\mathbf{y} \mathrm{d} s\\
    \leq & C\delta^2 \max_{0 \leq s \leq 1} \int_{\mathcal{O}_i} \int_{B_{\mathbf{y}}^{2\delta}} \bar{\bar{R}}_\delta(\mathbf{x}, \mathbf{y})\left|D^{1,2} u(\Phi_i(\alpha+s\xi))\right|^{2} \mathrm{d} S_{\mathbf{x}}  \mathrm{d} S_\mathbf{y}
\end{align*}

Let $\mathbf{z}_{i}=\Phi_{i}(\alpha+s \xi), 0 \leq s \leq 1$, then for any $\mathbf{y} \in \mathcal{O}_{i} \subset B_{\mathbf{q}_{i}}^{t}$ and $\mathbf{x} \in B_{\mathbf{y}}^{2\delta}$, by \Cref{eq:convergence_small_deformation}
\[\left\|\mathbf{z}_{i}-\mathbf{y}\right\| \leq 2 s|\xi| \leq 4 s\|\mathbf{x}-\mathbf{y}\| \leq 8 s \delta, \quad\left\|\mathbf{z}_{i}-\mathbf{q}_{i}\right\| \leq\left\|\mathbf{z}_{i}-\mathbf{y}\right\|+\left\|\mathbf{y}-\mathbf{q}_{i}\right\| \leq 8 s \delta+t \]
	
We can assume that $\delta$ is small enough such that $8\delta \leq t$, then we have $\mathbf{z}_{i} \in B_{\mathbf{q}_{i}}^{2t}$. After changing of variable, we obtain
\begin{align*}
	& \int_{\mathcal{O}_{i}} \int_{B_{\mathbf{y}}^{2\delta}} \bar{\bar{R}}_\delta(\mathbf{x}, \mathbf{y})\left|D^{1,2} u\left(\Phi_{i}(\alpha+s \xi)\right)\right|^{2} \mathrm{~d} S_{\mathbf{x}} \mathrm{d} S_{\mathbf{y}} \numberthis \label{eq:convergence_boundary_truncation3}\\
	\leq & \frac{C_1}{\gamma_{0}} \int_{\mathcal{O}_{i}} \int_{B_{\mathbf{q}_{i}}^{2t}} \frac{1}{s^{k}} \bar{\bar{R}}\left(\frac{\left\|\mathbf{z}_{i}-\mathbf{y}\right\|^{2}}{128 s^{2} \delta^2}\right)\left|D^{1,2} u\left(\mathbf{z}_{i}\right)\right|^{2} \mathrm{~d} S_{\mathbf{z}_{i}} \mathrm{d} S_{\mathbf{y}} \\
	= & \frac{C_1}{\gamma_{0}} \int_{\mathcal{O}_{i}} \frac{1}{s^{k}} \bar{\bar{R}}\left(\frac{\left\|\mathbf{z}_{i}-\mathbf{y}\right\|^{2}}{128 s^{2} \delta^2}\right) \mathrm{d} S_{\mathbf{y}} \int_{B_{\mathbf{q}_{i}}^{2t}} \left|D^{1,2} u\left(\mathbf{z}_{i}\right)\right|^{2} \mathrm{d} S_{\mathbf{z}_{i}} \\
	\leq & \frac{C}{\delta} \int_{B_{\mathbf{q}_{i}}^{2 t}}\left|D^{1,2} u(\mathbf{x})\right|^{2} \mathrm{d} S_{\mathbf{x}}
\end{align*}
	
Combine \Cref{eq:convergence_boundary_truncation1}, \Cref{eq:convergence_boundary_truncation2} and \Cref{eq:convergence_boundary_truncation3}, we have
\[\left\|\int_{\partial \Omega} \bar{\bar{R}}_\delta(\mathbf{x}, \mathbf{y})\left(\frac{\partial u}{\partial \mathbf{n}}(\mathbf{x})-\frac{\partial u}{\partial \mathbf{n}}(\mathbf{y}) \right)  \mathrm{d} S_{\mathbf{y}} \right\|_{L^2(\partial \Omega)}^2 \leq C\|u\|_{H^3(\Omega)} \]

Thus the boundary truncation error \Cref{eq:convergence_consistency_core} can be bounded by
\[\|r_{\mathrm{bd}}(\mathbf{x})\|_{L^{2}(\partial \Omega)}\leq C\|u\|_{H^3(\Omega)}  \]
\end{proof}

\section{\texorpdfstring{Proof in \Cref{sec:robin}}{Proof in Section 6}}
\subsection{\texorpdfstring{Proof of \Cref{thm:convergence_robin}}{Proof of Theorem 6.2}}
\label{subsec:6.2}
\begin{proof}
Move all terms containing $u$ and $\frac{\partial u}{\partial \mathbf{n}}$ to the left hand side, we get
\begin{align}
	\frac{1}{\delta^2} \int_{\Omega} R_{\delta}(\mathbf{x}, \mathbf{y})(e_\delta(\mathbf{x})-e_\delta(\mathbf{y})) \mathrm{d} \mathbf{y} - 2 \int_{\partial \Omega} \bar{R}_{\delta}(\mathbf{x}, \mathbf{y}) \tilde{e}_\delta(\mathbf{y}) \mathrm{d} S_{\mathbf{y}} &= r_{\mathrm{in}}(\mathbf{x}), \quad \mathbf{x} \in \Omega \label{eq:r_in_robin}\\
	-\frac{1}{\delta^2} \int_{\Omega} \bar{R}_{\delta}(\mathbf{x}, \mathbf{y})(\mu\tilde{e}_\delta(\mathbf{x})+e_\delta(\mathbf{y})) \mathrm{d} \mathbf{y} - 2 \int_{\partial \Omega} \bar{\bar{R}}_{\delta}(\mathbf{x}, \mathbf{y}) \tilde{e}_\delta(\mathbf{x}) \mathrm{d} S_{\mathbf{y}} &= r_{\mathrm{bd}}(\mathbf{x}), \quad \mathbf{x} \in \partial \Omega \label{eq:r_bd_robin}
\end{align}
where for $\mathbf{x} \in \Omega$,
\[r_{\mathrm{in}}(\mathbf{x})=\frac{1}{\delta^2} \int_{\Omega} R_{\delta}(\mathbf{x}, \mathbf{y})(u(\mathbf{x})-u(\mathbf{y})) \mathrm{d} \mathbf{y}-2 \int_{\partial \Omega} \bar{R}_{\delta}(\mathbf{x}, \mathbf{y}) \frac{\partial u}{\partial \mathbf{n}}(\mathbf{y}) \mathrm{d} S_{\mathbf{y}} - \int_{\Omega} \bar{R}_{\delta}(\mathbf{x}, \mathbf{y}) f(\mathbf{y}) \mathrm{d} \mathbf{y}\]
and for $\mathbf{x} \in \partial \Omega$,
\begin{align*}
	& \quad r_{\mathrm{bd}}(\mathbf{x}) \\
	& = -\frac{1}{\delta^2} \int_{\Omega} \bar{R}_{\delta}(\mathbf{x}, \mathbf{y})(\mu \frac{\partial u}{\partial \mathbf{n}}(\mathbf{x})+ u(\mathbf{y})) \mathrm{d} \mathbf{y}-2 \int_{\partial \Omega} \bar{\bar{R}}_{\delta}(\mathbf{x}, \mathbf{y})  \frac{\partial u}{\partial \mathbf{n}}(\mathbf{x}) \mathrm{d} S_{\mathbf{y}} -\int_{\Omega} \bar{\bar{R}}_{\delta}(\mathbf{x}, \mathbf{y}) f(\mathbf{y}) \mathrm{d} \mathbf{y} \\
	& = \frac{1}{\delta^2} \int_{\Omega} \bar{R}_{\delta}(\mathbf{x}, \mathbf{y})(u(\mathbf{x})- u(\mathbf{y})) \mathrm{d} \mathbf{y}-2 \int_{\partial \Omega} \bar{\bar{R}}_{\delta}(\mathbf{x}, \mathbf{y})  \frac{\partial u}{\partial \mathbf{n}}(\mathbf{x}) \mathrm{d} S_{\mathbf{y}} - \int_{\Omega} \bar{\bar{R}}_{\delta}(\mathbf{x}, \mathbf{y}) f(\mathbf{y}) \mathrm{d} \mathbf{y}
\end{align*}

The truncation errors in the domain and on the boundary are the same as those defined in \Cref{sec:convergence} (note that $u(\mathbf{x})=0$ in Dirichlet boundary condition). Thus the consistency results remain the same, we only need to focus on the stability part.

Multiply \eqref{eq:r_in_robin} by $e_\delta(\mathbf{x})$ and integrate over $\Omega$, multiply \eqref{eq:r_bd_robin} by $\tilde{e}_\delta(\mathbf{x})$ and integrate over $\partial \Omega$, multiply the second with $-2\delta^2$ and add two equations, we get
\begin{align*}
    &\quad \frac{1}{2\delta^2}\int_{\Omega} \int_{\Omega} R_{\delta}(\mathbf{x}, \mathbf{y})(e_\delta(\mathbf{x})-e_\delta(\mathbf{y}))^2 \mathrm{d} \mathbf{x} \mathrm{d} \mathbf{y} + 4\delta^2 \int_{\partial \Omega} \hat{w}_\delta(\mathbf{x}) \tilde{e}_\delta(\mathbf{x})^2 \mathrm{d} S_\mathbf{x}\\ &=\int_{\Omega}e_\delta(\mathbf{x})r_{\mathrm{in}}(\mathbf{x}) \mathrm{d} \mathbf{x}-2\delta^2\int_{\partial \Omega}\tilde{e}_\delta(\mathbf{x})r_{\mathrm{bd}}(\mathbf{x}) \mathrm{d} S_\mathbf{x}
\end{align*}

here we use the notation defined in the proof of \Cref{thm:wellposed_robin}, $\hat{w}_\delta(\mathbf{x})=\bar{\bar{w}}_\delta(\mathbf{x})+\frac{\mu}{2\delta^2}\bar{w}_\delta(\mathbf{x})$. Similarly, we will replace all $\bar{\bar{w}}_\delta(\mathbf{x})$ with $\hat{w}_\delta(\mathbf{x})$. But different from \Cref{thm:wellposed_robin}, now the order of $\delta$ matters since $\delta$ is no longer fixed, so we should be careful when doing substitution. We will enumerate all the estimation containing the newly introduced term $\hat{w}_\delta(\mathbf{x})$, and check the inequalities. For notation simplicity, we define the smoothed version of $e_\delta$
\[\hat{e}_\delta(\mathbf{x})= \frac{1}{\bar{w}_\delta(\mathbf{x})}\int_{\Omega} \bar{R}_\delta(\mathbf{x}, \mathbf{y}) e_\delta(\mathbf{y})\mathrm{d} \mathbf{y}\]

\textbf{(1)} When controlling $\|e_\delta\|_{L^2(\Omega)}$, similar to \Cref{eq:coercive_second_term2},
\begin{align*}
	& \frac{1}{4\delta^2}\int_{\partial \Omega} \frac{1}{\hat{w}_\delta(\mathbf{x})} \left(\int_{\Omega} \bar{R}_\delta(\mathbf{x}, \mathbf{y}) e_\delta(\mathbf{y})\mathrm{d} \mathbf{y}\right)^2 \mathrm{d} S_\mathbf{x} \\
	= {}& \frac{1}{4\delta^2}\int_{\partial \Omega} \frac{\bar{w}_\delta^2(\mathbf{x})}{\hat{w}_\delta(\mathbf{x})} \left(\frac{1}{\bar{w}_\delta(\mathbf{x})}\int_{\Omega} \bar{R}_\delta(\mathbf{x}, \mathbf{y}) e_\delta(\mathbf{y})\mathrm{d} \mathbf{y}\right)^2 \mathrm{d} S_\mathbf{x} \\
	= {}& \frac{1}{4\delta^2}\int_{\partial \Omega} \frac{\bar{w}_\delta^2(\mathbf{x})}{\hat{w}_\delta(\mathbf{x})} \hat{e}_\delta^2(\mathbf{x})\mathrm{d} S_\mathbf{x} \geq C\|\hat{e}_\delta(\mathbf{x})\|^2_{L^2(\partial\Omega)}
\end{align*}
where difference lies in last inequality. Here the inequality becomes tight, while in original inference we only use $\frac{1}{\delta} \geq 1$,

\textbf{(2)} When controlling $\|\nabla e_\delta\|_{L^2(\Omega)}$, similar to \Cref{eq:h1_gradient_core}
\begin{align*}
	\|\nabla e_\delta(\mathbf{x})\|_{L^2(\Omega)}^2 & \leq C\|\nabla\big( \frac{1}{w_\delta(\mathbf{x})}\int_{\Omega} R_{\delta}(\mathbf{x}, \mathbf{y})e_\delta(\mathbf{y}) \mathrm{d} \mathbf{y}  \big) \|_{L^2(\Omega)}^2\\
	& \quad +C\|\nabla\big( \frac{1}{w_\delta(\mathbf{x})}\int_{\partial \Omega} \frac{\bar{R}_{\delta}(\mathbf{x}, \mathbf{y}) \bar{w}_\delta(\mathbf{y}) }{\hat{w}_\delta(\mathbf{y})} \hat{e}_\delta(\mathbf{y}) \mathrm{d} S_{\mathbf{y}} \big) \|_{L^2(\Omega)}^2\\
	& \quad +C\|\nabla\big( \frac{\delta^2}{w_\delta(\mathbf{x})}\int_{\partial \Omega} \frac{\bar{R}_{\delta}(\mathbf{x}, \mathbf{y})}{\hat{w}_\delta(\mathbf{y})} r_{\mathrm{bd}}(\mathbf{y}) \mathrm{d} S_{\mathbf{y}} \big) \|_{L^2(\Omega)}^2\\
	& \quad +C\|\nabla\big( \frac{\delta^2}{w_\delta(\mathbf{x})}r_{\mathrm{in}}(\mathbf{x}) \big) \|_{L^2(\Omega)}^2		
\end{align*}

We will deal with the two terms containing $\hat{w}_\delta(\mathbf{x})$.
\begin{align*}
	& \left\|\nabla\big( \frac{1}{w_\delta(\mathbf{x})}\int_{\partial \Omega} \frac{\bar{R}_{\delta}(\mathbf{x}, \mathbf{y}) \bar{w}_\delta(\mathbf{y}) }{\hat{w}_\delta(\mathbf{y})} \hat{e}_\delta(\mathbf{y}) \mathrm{d} S_{\mathbf{y}} \big) \right\|_{L^2(\Omega)}^2 \\
	\leq {}& 2\left\|\frac{\nabla w_\delta(\mathbf{x})}{w^2_\delta(\mathbf{x})}\int_{\partial \Omega} \frac{\bar{R}_{\delta}(\mathbf{x}, \mathbf{y}) \bar{w}_\delta(\mathbf{y}) }{\hat{w}_\delta(\mathbf{y})} \hat{e}_\delta(\mathbf{y}) \mathrm{d} S_{\mathbf{y}} \right\|_{L^2(\Omega)}^2 \\
	& \hspace{150pt} + 2\left\| \frac{1}{w_\delta(\mathbf{x})}\int_{\partial \Omega} \frac{\nabla\bar{R}_{\delta}(\mathbf{x}, \mathbf{y}) \bar{w}_\delta(\mathbf{y}) }{\hat{w}_\delta(\mathbf{y})} \hat{e}_\delta(\mathbf{y}) \mathrm{d} S_{\mathbf{y}}  \right\|_{L^2(\Omega)}^2
\end{align*}

The first term
\begin{align*}
	& \left\|\frac{\nabla w_\delta(\mathbf{x})}{w^2_\delta(\mathbf{x})}\int_{\partial \Omega} \frac{\bar{R}_{\delta}(\mathbf{x}, \mathbf{y}) \bar{w}_\delta(\mathbf{y})}{\hat{w}_\delta(\mathbf{y})} \hat{u}_\delta(\mathbf{y}) \mathrm{d} S_{\mathbf{y}} \right\|_{L^2(\Omega)}^2 \\
	\leq {}& \frac{C_1}{\delta^2} \int_{\Omega} \left(\int_{\partial \Omega} \frac{\bar{R}_{\delta}(\mathbf{x}, \mathbf{y}) \bar{w}_\delta(\mathbf{y})}{\hat{w}_\delta(\mathbf{y})} \hat{u}_\delta(\mathbf{y}) \mathrm{d} S_{\mathbf{y}} \right)^2 \mathrm{d} \mathbf{x}\\
	\leq {}& C_2\delta^2 \int_{\Omega} \left(\int_{\partial \Omega} \bar{R}_{\delta}(\mathbf{x}, \mathbf{y}) \hat{u}_\delta(\mathbf{y}) \mathrm{d} S_{\mathbf{y}} \right)^2 \mathrm{d} \mathbf{x} \\
	\leq {}& C_3\delta \int_{\partial \Omega} \hat{u}^2_\delta(\mathbf{y}) \mathrm{d} S_{\mathbf{y}} \leq \frac{C_4}{\delta^2}\int_{\partial \Omega} \frac{\bar{w}_\delta^2(\mathbf{x})}{\hat{w}_\delta(\mathbf{x})} \hat{u}_\delta^2(\mathbf{x})\mathrm{d} S_\mathbf{x}\\
	= {} &\frac{C_4}{\delta^2}\int_{\partial \Omega} \frac{1}{\hat{w}_\delta(\mathbf{x})} \left(\int_{\Omega} \bar{R}_\delta(\mathbf{x}, \mathbf{y}) u_\delta(\mathbf{y})\mathrm{d} \mathbf{y}\right)^2 \mathrm{d} S_\mathbf{x}
\end{align*}

Compared with \Cref{eq:h1_gradeint_second_term}, the inequality becomes loose: we use $\delta \leq 1$ in the last inequality. And the second term above can be checked similarly. The other term containing $\hat{w}_\delta(\mathbf{x})$ can be bounded by
\[\left\|\nabla\left( \frac{\delta^2}{w_\delta(\mathbf{x})}\int_{\partial \Omega} \frac{\bar{R}_{\delta}(\mathbf{x}, \mathbf{y})}{\hat{w}_\delta(\mathbf{y})} r_{\mathrm{bd}}(\mathbf{y}) \mathrm{d} S_{\mathbf{y}} \right) \right\|_{L^2(\Omega)}^2 \leq C\delta^5 \left\| r_{\mathrm{bd}}(\mathbf{x}) \right\|_{L^2(\partial \Omega)}^2\]
where in Dirichlet condition the order is $O(\delta^3)$.

Finally, the convergence order can be derived
\begin{align*}
	& \quad \|e_\delta\|_{H^1(\Omega)}^2 \\
	\leq {}& \frac{C}{2\delta^2} \int_{\Omega} \int_{\Omega} R_{\delta}(\mathbf{x}, \mathbf{y})(e_\delta(\mathbf{x})-e_\delta(\mathbf{y}))^2 \mathrm{d} \mathbf{x} \mathrm{d} \mathbf{y} + \frac{C}{2\delta^2}\int_{\partial \Omega} \frac{1}{\hat{w}_\delta(\mathbf{x})} \left(\int_{\Omega} \bar{R}_\delta(\mathbf{x}, \mathbf{y}) e_\delta(\mathbf{y})\mathrm{d} \mathbf{y}\right)^2 \mathrm{d} S_\mathbf{x} \\
	& \vphantom{\left(\int_{\Omega} \bar{R}_\delta(\mathbf{x}, \mathbf{y}) \right)^2} + C\delta^5 \| r_{\mathrm{bd}}(\mathbf{x}) \|_{L^2(\partial \Omega)}^2 + C\delta^2 \| r_{\mathrm{in}}(\mathbf{x}) \|_{L^2(\Omega)}^2 + C\delta^4 \| \nabla r_{\mathrm{in}}(\mathbf{x}) \|_{L^2(\Omega)}^2\\
	\leq {} & C\int_{\Omega} e_\delta(\mathbf{x}) r_{\mathrm{it}}(\mathbf{x}) \mathrm{d} \mathbf{x} + C\int_{\Omega} e_\delta(\mathbf{x}) r_{\mathrm{bl}}(\mathbf{x}) \mathrm{d} \mathbf{x} + C\delta^3 \| r_{\mathrm{bd}}(\mathbf{x}) \|_{L^2(\partial \Omega)}^2\\
	& \vphantom{\int} + C \delta^5 \left\|r_{\mathrm{bd}}(\mathbf{x})\right\|_{L^{2}(\partial \Omega)}^2 + C\delta^2 \| r_{\mathrm{in}}(\mathbf{x}) \|_{L^2(\Omega)}^2 + C\delta^4 \| \nabla r_{\mathrm{in}}(\mathbf{x}) \|_{L^2(\Omega)}^2\\
	\leq {}& \vphantom{\int} C\|e_\delta\|_{L^2(\Omega)} \|r_{\mathrm{it}}(\mathbf{x})\|_{L^2(\Omega)}+C\delta\|e_\delta\|_{H^1(\Omega)}\|u\|_{H^{3}(\Omega)}\\
	& \vphantom{\int} + C\delta^3\|u\|_{H^{3}(\Omega)}^2 + C\delta^4\|u\|_{H^{3}(\Omega)}^2 + C\delta^4\|u\|_{H^{3}(\Omega)}^2 \\
	\leq {} & \vphantom{\int} C\delta\|e_\delta\|_{H^1(\Omega)} \|u\|_{H^{3}(\Omega)}+ C\delta^3 \|u\|_{H^{3}(\Omega)}^2
\end{align*}

Thus
\[\left\|e_\delta\right\|_{H^1(\Omega)}\leq C\delta\left\|u\right\|_{H^{3}(\Omega)} \]

Moreover, the convergence rate of $v_\delta$ can be improved
\begin{align*}
	&\vphantom{\int} \quad C\|\tilde{e}_\delta\|_{L^2(\partial \Omega)}^2 \leq 4\delta^2 \int_{\partial \Omega}  \hat{w}_\delta(\mathbf{x}) \tilde{e}_\delta(\mathbf{x})^2 \mathrm{d} S_\mathbf{x} \leq \int_{\Omega}e_\delta(\mathbf{x})r_{\mathrm{in}}(\mathbf{x}) \mathrm{d} \mathbf{x} - 2\delta^2\int_{\partial \Omega}\tilde{e}_\delta(\mathbf{x})r_{\mathrm{bd}}(\mathbf{x}) \mathrm{d} S_\mathbf{x} \\
	& \leq \left|\int_{\Omega} e_\delta(\mathbf{x}) r_{\mathrm{it}}(\mathbf{x}) \mathrm{d} \mathbf{x} \right|+ \left|\int_{\Omega} e_\delta(\mathbf{x}) r_{\mathrm{bl}}(\mathbf{x}) \mathrm{d} \mathbf{x} \right|+ 2 \delta^2 \left\|r_{\mathrm{bd}}(\mathbf{x})\right\|_{L^{2}(\partial \Omega)}\left\|\tilde{e}_\delta\right\|_{L^{2}(\partial \Omega)} \\
	& \vphantom{\int} \leq C\|e_\delta\|_{L^2(\Omega)} \|r_{\mathrm{it}}(\mathbf{x})\|_{L^2(\Omega)}+C\delta\|e_\delta\|_{H^1(\Omega)}\|u\|_{H^{3}(\Omega)} + 2 \delta^2 \left\|r_{\mathrm{bd}}(\mathbf{x})\right\|_{L^{2}(\partial \Omega)}\left\|\tilde{e}_\delta\right\|_{L^{2}(\partial \Omega)} \\
	& \vphantom{\int} \leq C \delta^2 \|u\|_{H^{3}(\Omega)}^2 + C \delta^2 \|u\|_{H^{3}(\Omega)}\left\|\tilde{e}_\delta\right\|_{L^{2}(\partial \Omega)}
\end{align*}

It gives us
\[\left\|\tilde{e}_\delta\right\|_{L^{2}(\partial \Omega)} \leq C \delta \|u\|_{H^{3}(\Omega)} \]
which is slightly better than $O(\delta^{\frac{1}{2}})$
\end{proof}
\section{\texorpdfstring{Proof in \Cref{sec:spectra}}{Proof in Section 7}}

\subsection{\texorpdfstring{Proof of \Cref{prop:operators}}{Proof of Proposition 7.1}}
\label{subsec:7.2}

\begin{proof}
For $T$ the conclusion is well known. To show the compactness of $T_\delta$, we notice that we can rewrite $T_\delta u_\delta(\mathbf{x})$ as in \Cref{subsec:h1_estimation}
\begin{align*}
	\vphantom{\int}T_\delta u_\delta(\mathbf{x})= & \frac{1}{w_\delta(\mathbf{x})}\int_{\Omega} R_{\delta}(\mathbf{x}, \mathbf{y})T_\delta u_\delta(\mathbf{y}) \mathrm{d} \mathbf{y} - \frac{1}{w_\delta(\mathbf{x})}\int_{\partial \Omega} \frac{\bar{R}_{\delta}(\mathbf{x}, \mathbf{y})}{\bar{\bar{w}}_\delta(\mathbf{y})}\int_{\Omega} \bar{R}_\delta(\mathbf{y}, \mathbf{s}) T_\delta u_\delta(\mathbf{s})\mathrm{d} \mathbf{s} \mathrm{d} S_{\mathbf{y}} \\
	& + \frac{\delta^2}{w_\delta(\mathbf{x})}\int_{\Omega} \bar{R}_{\delta}(\mathbf{x}, \mathbf{y}) u_\delta(\mathbf{y}) \mathrm{d} \mathbf{y}
\end{align*}
% -\frac{\delta^2}{w_\delta(\mathbf{x})}\int_{\partial \Omega} \frac{\bar{R}_{\delta}(\mathbf{x}, \mathbf{y})}{\bar{\bar{w}}_\delta(\mathbf{y})} \int_{\Omega} \bar{\bar{R}}_{\delta}(\mathbf{y}, \mathbf{s}) u_\delta(\mathbf{s}) \mathrm{d} \mathbf{s} \mathrm{d} \mathbf{y}

With \Cref{assumption:kernel}(a), $R \in C^1$. Then, direct calculation gives that that $T_{\delta} u_\delta \in C^1$. This implies the compactness of $T_\delta$ in $H^1$.

Let $\lambda$ be an eigenvalue of $T_\delta$ and $u_\delta$ is the corresponding eigenfunction. We can eliminate $v_\delta$ in the first equation using \Cref{eq:v_eliminate} and get
\begin{align*}
	&\lambda \left(\frac{1}{\delta^2}\int_{\Omega} R_{\delta}(\mathbf{x}, \mathbf{y})(u_\delta(\mathbf{x})-u_\delta(\mathbf{y})) \mathrm{d} \mathbf{y}+\frac{1}{\delta^2}\int_{\partial \Omega} \frac{\bar{R}_{\delta}(\mathbf{x}, \mathbf{y})}{\bar{\bar{w}}_\delta(\mathbf{y})} \int_{\Omega} \bar{R}_{\delta}(\mathbf{y}, \mathbf{s}) u_\delta(\mathbf{s}) \mathrm{d} \mathbf{s} \mathrm{d} S_\mathbf{y} \right) \numberthis \label{eq:spectral_lambda}\\
	=& \frac{1}{\delta^2}\int_{\Omega} R_{\delta}(\mathbf{x}, \mathbf{y})(T_\delta u_\delta(\mathbf{x})-T_\delta u_\delta(\mathbf{y})) \mathrm{d} \mathbf{y}+\frac{1}{\delta^2}\int_{\partial \Omega} \frac{\bar{R}_{\delta}(\mathbf{x}, \mathbf{y})}{\bar{\bar{w}}_\delta(\mathbf{y})} \int_{\Omega} \bar{R}_{\delta}(\mathbf{y}, \mathbf{s}) T_\delta u_\delta(\mathbf{s}) \mathrm{d} \mathbf{s} \mathrm{d} S_\mathbf{y} \\
	=& \int_{\Omega}  \bar{R}_{\delta}(\mathbf{x}, \mathbf{y}) u_\delta(\mathbf{y})\mathrm{d} \mathbf{y}
\end{align*}

Multiply $u_\delta^*(\mathbf{x})$, which is the complex conjugate of $u_\delta(\mathbf{x})$, to both sides and integrate over $\Omega$. Then using symmetry of the kernel functions, we can get
\[\left( \int_{\Omega} u_\delta^*(\mathbf{x}) \int_{\Omega} R_{\delta}(\mathbf{x}, \mathbf{y})u_\delta(\mathbf{y}) \mathrm{d} \mathbf{y} \mathrm{d} \mathbf{x} \right)^*= \int_{\Omega} u_\delta^*(\mathbf{x}) \int_{\Omega} R_{\delta}(\mathbf{x}, \mathbf{y})u_\delta(\mathbf{y}) \mathrm{d} \mathbf{y} \mathrm{d} \mathbf{x} \]
and,
\begin{align*}
	&\left( \int_{\Omega} u_\delta^*(\mathbf{x}) \int_{\Omega} R_{\delta}(\mathbf{x}, \mathbf{y})(u_\delta(\mathbf{x})-u_\delta(\mathbf{y})) \mathrm{d} \mathbf{y} \mathrm{d} \mathbf{x} \right)^*\\
	=& \int_{\Omega} u_\delta(\mathbf{x}) \int_{\Omega} R_{\delta}(\mathbf{x}, \mathbf{y})(u_\delta^*(\mathbf{x})-u_\delta^*(\mathbf{y})) \mathrm{d} \mathbf{y} \mathrm{d} \mathbf{x} \\
	=& \int_{\Omega} u_\delta^*(\mathbf{x}) \int_{\Omega} R_{\delta}(\mathbf{x}, \mathbf{y})(u_\delta(\mathbf{x})-u_\delta(\mathbf{y})) \mathrm{d} \mathbf{y} \mathrm{d} \mathbf{x}
\end{align*}
and,
\begin{align*}
	&\left(\int_{\Omega} u_\delta^*(\mathbf{x}) \int_{\partial \Omega} \frac{\bar{R}_{\delta}(\mathbf{x}, \mathbf{y})}{\bar{\bar{w}}_\delta(\mathbf{y})} \int_{\Omega} \bar{R}_{\delta}(\mathbf{y}, \mathbf{s}) u_\delta(\mathbf{s}) \mathrm{d} \mathbf{s} \mathrm{d} S_\mathbf{y} \mathrm{d} \mathbf{x}\right)^*\\
	=&\int_{\Omega} u_\delta(\mathbf{x}) \int_{\partial \Omega} \frac{\bar{R}_{\delta}(\mathbf{x}, \mathbf{y})}{\bar{\bar{w}}_\delta(\mathbf{y})} \int_{\Omega} \bar{R}_{\delta}(\mathbf{y}, \mathbf{s}) u_\delta^*(\mathbf{s}) \mathrm{d} \mathbf{s} \mathrm{d} S_\mathbf{y} \mathrm{d} \mathbf{x}\\
	=& \int_{\partial \Omega} \frac{1}{\bar{\bar{w}}_\delta(\mathbf{y})} \left(\int_{\Omega} \bar{R}_{\delta}(\mathbf{x}, \mathbf{y}) u_\delta(\mathbf{x}) \mathrm{d} \mathbf{x} \right)\left( \int_{\Omega} \bar{R}_{\delta}(\mathbf{y}, \mathbf{s}) u_\delta^*(\mathbf{s}) \mathrm{d} \mathbf{s} \right) \mathrm{d} S_\mathbf{y}\\
	=& \int_{\partial \Omega} \frac{1}{\bar{\bar{w}}_\delta(\mathbf{y})} \left(\int_{\Omega} \bar{R}_{\delta}(\mathbf{x}, \mathbf{y}) u_\delta^*(\mathbf{x}) \mathrm{d} \mathbf{x} \right)\left( \int_{\Omega} \bar{R}_{\delta}(\mathbf{y}, \mathbf{s}) u_\delta(\mathbf{s}) \mathrm{d} \mathbf{s} \right) \mathrm{d} S_\mathbf{y}\\
	=& \int_{\Omega} u_\delta^*(\mathbf{x}) \int_{\partial \Omega} \frac{\bar{R}_{\delta}(\mathbf{x}, \mathbf{y})}{\bar{\bar{w}}_\delta(\mathbf{y})} \int_{\Omega} \bar{R}_{\delta}(\mathbf{y}, \mathbf{s}) u_\delta(\mathbf{s}) \mathrm{d} \mathbf{s} \mathrm{d} S_\mathbf{y} \mathrm{d} \mathbf{x}
\end{align*}

Thus we get $\lambda\in \mathbb{R}$.

Now, we turn to study the eigenfunctions. Let $u_\delta$ be a generalized eigenfunction of $T_\delta$ with multiplicity $m>1$ associate with eigenvalue $\lambda$. Let $p_\delta=\left(T_\delta-\lambda\right)^{m-1} u_\delta, q_\delta=\left(T_\delta-\lambda\right)^{m-2} u_\delta$, then $p_\delta$ is an eigenfunction of $T_\delta$ and
\[T_\delta p_\delta=\lambda p_\delta, \quad\left(T_\delta-\lambda\right) q_\delta=p_\delta\]

Denote $L_\delta$ as
\[L_\delta f(\mathbf{x})= \frac{1}{\delta^2}\int_{\Omega} R_{\delta}(\mathbf{x}, \mathbf{y})(f(\mathbf{x})-f(\mathbf{y})) \mathrm{d} \mathbf{y}+\frac{1}{\delta^2}\int_{\partial \Omega} \frac{\bar{R}_{\delta}(\mathbf{x}, \mathbf{y})}{\bar{\bar{w}}_\delta(\mathbf{y})} \int_{\Omega} \bar{R}_{\delta}(\mathbf{y}, \mathbf{s}) f(\mathbf{s}) \mathrm{d} \mathbf{s} \mathrm{d} S_\mathbf{y}\]

By applying $L_\delta$ on both sides of above two equations and use \Cref{eq:spectral_lambda}, we have
\begin{align*}
	\lambda L_\delta p_\delta &=L_\delta\left(T_\delta p_\delta \right)=\int_{\Omega} \bar{R}_\delta(\mathbf{x}, \mathbf{y}) p_\delta(\mathbf{y}) \mathrm{d} \mathbf{y} \\
	L_\delta p_\delta &=L_\delta\left(T_\delta q_\delta \right)-\lambda L_\delta q_\delta=\int_{\Omega} \bar{R}_\delta(\mathbf{x}, \mathbf{y}) q_\delta(\mathbf{y}) \mathrm{d} \mathbf{y}-\lambda L_\delta q_\delta
\end{align*}

Using above two equations, and use symmetry of kernel functions again, we get
\begin{align*}
	0 &=\left\langle q_\delta, \lambda L_\delta p_\delta -\int_{\Omega} \bar{R}_\delta(\mathbf{x}, \mathbf{y}) p_\delta (\mathbf{y}) \mathrm{d} \mathbf{y}\right \rangle \\
	&=\left\langle\lambda L_\delta q_\delta-\int_{\Omega} \bar{R}_\delta(\mathbf{x}, \mathbf{y}) q_\delta (\mathbf{y}) \mathrm{d} \mathbf{y}, p_\delta \right\rangle \\
	& \vphantom{\int}= - \left\langle L_\delta p_\delta, p_\delta \right\rangle
\end{align*}

Moreover, we have proved in \Cref{prop:coercive} that
\[\left\langle L_\delta p_\delta, p_\delta \right\rangle \geq C \|p_\delta\|_{L^2(\Omega)}^2\]

which implies that $\left(T_\delta-\lambda\right)^{m-1} u_\delta = p_\delta =0$. This proves that $u$ is a generalized eigenfunction of $T_\delta$ with multiplicity $m-1$. Repeating this process, we can show that $u$ is actually an eigenfunction of $T_\delta$.
\end{proof}

\subsection{\texorpdfstring{Proof of \Cref{thm:spectra_eigenfunction}}{Proof of Theorem 7.2}}
\label{subsec:7.4}

We need some results regarding the perturbation of the compact operators.

% \begin{theorem}[\cite{atkinson1967numerical}]
%     Let $\left(X,\|\cdot\|_{X}\right)$ be an arbitrary Banach space. Let $S$ and $T$ be compact linear operators on $X$ into $X$. Let $z \in \rho(T)$. Assume
%     $$
%     \|(T-S) S\|_{X} \leq \frac{|z|}{\left\|(z-T)^{-1}\right\|_{X}}
%     $$
%     Then $z \in \rho(S)$ and $(z-S)^{-1}$ has the bound
%     $$
%     \left\|(z-S)^{-1}\right\|_{X} \leq \frac{1+\|S\|_{X}\left\|(z-T)^{-1}\right\|_{X}}{|z|-\left\|(z-T)^{-1}\right\|_{X}\|(T-S) S\|_{X}}
%     $$
% \end{theorem}

\begin{theorem}[\cite{atkinson1967numerical}]
    \label{thm:spectra_perturbation}
    Let $\left(X,\|\cdot\|_{X}\right)$ be an arbitrary Banach space. Let $S$ and $T$ be compact linear operators on $X$ into $X$. Let $z_{0} \in \mathbb{C}, z_{0} \neq 0$ and let $\epsilon>0$ be less than $\left|z_{0}\right|$, denote the circumference $\left|z-z_{0}\right|=\epsilon$ by $\Gamma$ and assume $\Gamma \subset \rho(T)$. Denote the interior of $\Gamma$ by $U$. Let $\sigma_{T}=U \cap \sigma(T) \neq \emptyset . \sigma_{S}=U \cap \sigma(S)$. Let $E\left(\sigma_{S}, S\right)$ and $E\left(\sigma_{T}, T\right)$ be the corresponding spectral projections of $S$ for $\sigma_{S}$ and $T$ for $\sigma_{T}$, i.e.
    \[E\left(\sigma_{S}, S\right)=\frac{1}{2 \pi i} \int_{\Gamma}(z-S)^{-1} \mathrm{~d} z, \quad E\left(\sigma_{T}, T\right)=\frac{1}{2 \pi i} \int_{\Gamma}(z-T)^{-1} \mathrm{~d} z\]
    Assume
    \[\|(T-S) S\|_{X} \leq \min _{z \in \Gamma} \frac{|z|}{\left\|(z-T)^{-1}\right\|_{X}}\]
    Then, we have
    
    (1). Dimension $E\left(\sigma_{S}, S\right) X=E\left(\sigma_{T}, T\right) X$, thereby $\sigma_{S}$ is nonempty and of the same multiplicity as $\sigma_{T}$.
    
    (2). For every $x \in X$,
    \[\left\|E\left(\sigma_{T}, T\right) x-E\left(\sigma_{S}, S\right) x\right\|_{X} \leq \frac{2 M \epsilon}{c_{0}}\left(\|(T-S) x\|_{X}+ M \|x\|_{X}\|(T-S) S\|_{X}\right)\]
    where $M=\max _{z \in \Gamma}\left\|(z-T)^{-1}\right\|_{X}, c_{0}=\min _{z \in \Gamma}|z|$.
\end{theorem}

To use the above theorem, we further need two lemma in \cite{tao2020convergence}.
\begin{lemma}[\cite{tao2020convergence}]
    \label{lemma:spectra}
    Let $T$ be the solution operator and $z \in \rho(T)$, then
    $$
    \left\|(z-T)^{-1}\right\|_{H^{1}} \leq \max _{n \in \mathbb{N}} \frac{1}{\left|z-\lambda_{n}\right|},
    $$
    where $\left\{\lambda_{n}\right\}_{n \in \mathbb{N}}$ is the set of eigenvalues of $T$.
\end{lemma}

\begin{lemma}[\cite{tao2020convergence}]
	\label{lemma:spectra_eigenvalue}
	Let $T_\delta$ be the solution operator of the integral equation and $\lambda_{n}$ be eigenvalues of $T$, then
	\[\sigma\left(T_\delta\right) \subset \bigcup_{n \in \mathbb{N}} B\left(\lambda_{n}, 2\left\|T-T_\delta\right\|_{H^{1}}\right)\]
\end{lemma}

% \begin{lemma}[\cite{shi2015convergence}]
%     Let $T_\delta$ be the solution operator. For any $z \in \mathbb{C} \backslash \bigcup_{n \in \mathbb{N}} B\left(\lambda_{n}, r_{0}\right)$ with $r_{0}>\left\|T-T_\delta\right\|_{H^{1}}$, then
%     $$
%     \left\|\left(z-T_\delta\right)^{-1}\right\|_{C^{1}} \leq \max \left\{\frac{2|\Omega|}{|z| t^{(k+2) / 4}}\left(\min _{n \in \mathbb{N}}\left|z-\lambda_{n}\right|-\left\|T-T_\delta\right\|_{H^{1}}\right)^{-1}, \frac{2}{|z|}\right\}
%     $$
% \end{lemma}

Then we provide the proof of \Cref{thm:spectra_eigenfunction}.
\begin{proof}
    Let $\Gamma_{j}=\left\{z \in \mathbb{C}:\left|z-\lambda_{j}\right|=\gamma_{j} / 3\right\}$, $U_{j}$ be the area enclosed by $\Gamma_{j}$, $\sigma_j^\delta=\sigma\left(T_\delta\right) \bigcap U_{j}$.
    Using the definition of $\Gamma_{j}$ and $\gamma_{m}$, we know $\Gamma_{j} \subset \rho(T)$ and $\Gamma_{j} \subset \rho\left(T_\delta\right)$ for any $j \leq m $.
    
    In order to apply \Cref{thm:spectra_perturbation}, we need to verify the condition
    \begin{equation}
    	\label{eq:spectra_condition}
    	\left\|\left(T-T_\delta\right) T_\delta\right\|_{H^{1}} \leq \min _{z \in \Gamma_{m}} \frac{|z|}{\left\|(z-T)^{-1}\right\|_{H^{1}}}
    \end{equation}
    
    Using \Cref{lemma:spectra} and the choice of $\Gamma_{j}$, we have
    \begin{align*}
        &\min _{z \in \Gamma_{m}} \frac{|z|}{\left\|(z-T)^{-1}\right\|_{H^{1}}}\\
        \geq{}& \frac{\min _{z \in \Gamma_{m}}|z|}{\max _{z \in \Gamma_{m}}\left\|(z-T)^{-1}\right\|_{H^{1}}}\\
        \geq{}&\left(\left|\lambda_{m}\right|-\gamma_{m} / 3\right) \min _{z \in \Gamma_{m}, j \in \mathbb{N}} \left|z-\lambda_{j}\right|\\
        ={}&\left(\left|\lambda_{m}\right|-\gamma_{m} / 3\right) \gamma_{m} / 3
    \end{align*}
    
    Then, using the assumption that $\left\|\left(T-T_\delta\right) T_\delta\right\|_{H^1} \leq\left(\left|\lambda_{m}\right|-\gamma_{m} / 3\right) \gamma_{m} / 3$, the condition \Cref{eq:spectra_condition} is therefore satisfied.
    
    Thus, using \Cref{thm:spectra_perturbation}, we have
    \[\dim \left(E\left(\lambda_{m}, T \right)\right)=\dim \left(E\left(\sigma_{m}^\delta, T_\delta\right) \right)\]
    which, in combination with \Cref{lemma:spectra_eigenvalue}, implies
    \[\left|\lambda_m^{\delta}-\lambda_m\right| \leq 2 \left\|T-T_\delta\right\|_{H^1}\]
    
    Moreover, for any $x \in E\left(\lambda_{m}, T\right)$, since $\max _{z \in \Gamma_{m}}\left\|(z-T)^{-1}\right\|_{H^{1}}\leq 3/\gamma_{m}$ from \Cref{lemma:spectra},
    \begin{align*}
        &\left\|x-E\left(\sigma_m^\delta, T_\delta\right) x\right\|_{H^{1}} \\
        \leq {}&  \frac{2 \cdot 3/\gamma_{m} \cdot \gamma_{m} / 3}{\min _{z \in \Gamma_{m}}|z|}\left(\left\|\left(T-T_\delta\right) x\right\|_{H^{1}}+ \frac{3}{\gamma_{m}} \left\|\left(T-T_\delta\right) T_\delta\right\|_{H^{1}}\|x\|_{H^{1}}\right)\\
        \leq {}& \frac{2}{\left|\lambda_{m}\right|-\gamma_{m} / 3}\left(\left\|\left(T-T_\delta\right) x\right\|_{H^{1}}+ \frac{3}{\gamma_{m}} \left\|\left(T-T_\delta\right) T_\delta\right\|_{H^{1}}\|x\|_{H^{1}}\right)\\
        = {} & \vphantom{\frac{1}{\left|\lambda_{m}\right|-\gamma_{m} / 3}} C \left(\left\|\left(T-T_\delta\right) x\right\|_{H^{1}}+\left\|\left(T-T_\delta\right) T_\delta\right\|_{H^{1}}\|x\|_{H^{1}}\right)
    \end{align*}
	We thus get the convergence of eigenspace.
\end{proof}

\section{\texorpdfstring{Proof in \Cref{remark:h-1}}{Proof in Remark 3.1}}
\label{sec:remark_proof}

If we want to generalize to the setting $f\in H^{-1}$, we should first revise the right hand side of the model \Cref{eq:model} to the following duality pairing form,
\[F_{\mathrm{in}}(\mathbf{x})=\langle \bar{R}_{\delta}(\mathbf{x}, \mathbf{y}), f(\mathbf{y}) \rangle, \quad F_{\mathrm{bd}}(\mathbf{x})=\langle \bar{\bar{R}}_{\delta}(\mathbf{x}, \mathbf{y}), f(\mathbf{y}) \rangle \]

Notice that, for the duality pairing between $f\in H^{-1}(\Omega)$ and the kernel function $R_\delta(\mathbf{x},\mathbf{y})$ (as a function of $\mathbf{y}$) to be well defined, we need $R_\delta(\mathbf{x},\mathbf{y}) \in H_0^1(\Omega)$. However, if we take $\mathbf{x}$ close to $\partial \Omega$ such that $\text{dist}(\mathbf{x},\partial \Omega)<\delta$, then according to Assumption 2.2(d), we have $R_\delta(\mathbf{x},\mathbf{y})\geq C_\delta\gamma_0>0$. Thus, the definition of duality pairings between $f$ and kernel functions is unclear. Alternatively, we propose to extend the domain as $\Omega_\delta=\left\{x| \text{dist}(\mathbf{x}, \Omega) \leq 2\delta \right\}$, and define $f \in H^{-1}(\Omega_\delta)$. It is easy to see that, as a consequence of Assumption 2.2(c) and $\mathbf{x}\in \Omega$, $R_\delta(\mathbf{x},\mathbf{y})$ now belongs to $H_0^1(\Omega_\delta)$. Then we are ready to provide the $L^2$ estimate of $u_\delta$.

\begin{proposition}
    If $f \in H^{-1}(\Omega_\delta)$, then $u_{\delta} \in$ $L^2(\Omega)$ and the following estimate, with constant $C>0$ independent of $\delta$,
    \[\left\|u_{\delta}\right\|_{L^2(\Omega)} \leq \frac{C}{\delta} \left\|f\right\|_{H^{-1}(\Omega)}\]
\end{proposition}
\begin{proof}
    % \[E(u_\delta)=\frac{1}{2\delta^2}\int_{\Omega} \int_{\Omega} R_{\delta}(\mathbf{x}, \mathbf{y})(u_\delta(\mathbf{x})-u_\delta(\mathbf{y}))^2 \mathrm{d} \mathbf{x} \mathrm{d} \mathbf{y}+ \frac{1}{4\delta^2}\int_{\partial \Omega} \frac{1}{\bar{\bar{w}}_\delta(\mathbf{x})} \left(\int_{\Omega} \bar{R}_\delta(\mathbf{x}, \mathbf{y}) u_\delta(\mathbf{y})\mathrm{d} \mathbf{y}\right)^2 \mathrm{d} \mathbf{x} \]
    Since we change the model by eliminating the right-hand side of the second equation in \Cref{eq:model}, we have $F_{\mathrm{bd}}(\mathbf{x})=\langle \bar{\bar{R}}_{\delta}(\mathbf{x}, \mathbf{y}) , f(\mathbf{y}) \rangle=0$ for $\mathbf{x} \in \partial \Omega$. Thus, combine \Cref{prop:h1_core} and \Cref{corol:l2_estimation}, we have
    \[C\|u_\delta\|_{L^2(\Omega)}^2 \leq \int_{\Omega} u_\delta(\mathbf{x}) F_{\mathrm{in}}(\mathbf{x}) \mathrm{d} \mathbf{x}\]
    It is by the  definition of dual pairing that we have
    \[\int_{\Omega} u_\delta(\mathbf{x}) F_{\mathrm{in}}(\mathbf{x}) \mathrm{d} \mathbf{x} =\langle \int_{\Omega} \bar{R}_\delta(\mathbf{x}, \mathbf{y}) u_\delta(\mathbf{x}) \mathrm{d} \mathbf{x}, f(\mathbf{y}) \rangle \leq \left\|f\right\|_{H^{-1}(\Omega_\delta)} \left\|\tilde{u}_\delta\right\|_{H_0^1(\Omega_\delta)}\]
    where $\tilde{u}_\delta(\mathbf{x}):=\int_{\Omega} \bar{R}_\delta(\mathbf{x}, \mathbf{y}) u_\delta(\mathbf{y}) \mathrm{d} \mathbf{y}$. We will therefore control the $H_0^1$ norm ($H^1$ norm) of $\tilde{u}_\delta$. By utilizing Cauchy-Schwarz inequality twice, it is obvious that
    \[\left\|\tilde{u}_\delta\right\|_{L^2(\Omega_\delta)}^2\leq C\left\|u_\delta\right\|_{L^2(\Omega)}^2\]
    
    Moreover, we have
    \begin{align*}
        \left\|\nabla \tilde{u}_\delta\right\|_{L^2(\Omega_\delta)}^2 = \int_{\Omega_\delta}\left(\int_{\Omega} \nabla_{\mathbf{x}} \bar{R}_\delta(\mathbf{x}, \mathbf{y}) u_\delta(\mathbf{y}) \mathrm{d} \mathbf{y} \right)^2 \mathrm{d} \mathbf{x} & \leq \frac{C}{\delta^2} \int_{\Omega_\delta}\left(\int_{\Omega} R_\delta(\mathbf{x}, \mathbf{y}) u_\delta(\mathbf{y}) \mathrm{d} \mathbf{y} \right)^2 \mathrm{d} \mathbf{x} \\
        & \leq \frac{C}{\delta^2} \left\|u_\delta\right\|_{L^2(\Omega)}^2
    \end{align*}
    
    Thus
    \[ \left\|\tilde{u}_\delta\right\|_{H^1(\Omega_\delta)}^2=\left\|\tilde{u}_\delta\right\|_{L^2(\Omega_\delta)}^2+\left\|\nabla \tilde{u}_\delta\right\|_{L^2(\Omega_\delta)}^2 \leq \frac{C}{\delta^2} \left\|u_\delta\right\|_{L^2(\Omega)}^2 \leq \frac{C}{\delta^2} \left\|f\right\|_{H^{-1}(\Omega_\delta)} \left\|\tilde{u}_\delta\right\|_{H^1(\Omega_\delta)}\]
    which gives us $\left\|\tilde{u}_\delta\right\|_{H^1(\Omega_\delta)}\leq \frac{C}{\delta^2} \left\|f\right\|_{H^{-1}(\Omega_\delta)}$. Finally, we have
    \[C\|u_\delta\|_{L^2(\Omega)}^2 \leq \left\|f\right\|_{H^{-1}(\Omega_\delta)} \left\|\tilde{u}_\delta\right\|_{H^1(\Omega_\delta)} \leq \frac{C}{\delta^2} \left\|f\right\|_{H^{-1}(\Omega_\delta)}^2\]
    which proves 
    \[\|u_\delta\|_{L^2(\Omega)}\leq \frac{C}{\delta} \left\|f\right\|_{H^{-1}(\Omega_\delta)}\]
\end{proof}

\end{document}